\documentclass[11pt]{amsart}
\usepackage[margin=1in]{geometry}
\usepackage{amsmath, graphicx, subfigure, epsfig,amssymb,cite,bbm,bm, enumitem}
\usepackage{xcolor,mathtools}
\usepackage[normalem]{ulem}
\numberwithin{equation}{section}

\usepackage{hyperref}
\hypersetup{
    colorlinks=true,
    linkcolor=blue,
    filecolor=magenta,      
    urlcolor=cyan,
}

\newcommand{\Cf}{C_f}
\newcommand{\CL}{C_{34}}
\newcommand{\CR}{C_R}
\newcommand{\Clap}{C_{\z}}

\newcommand{\e}{\epsilon}

\newcommand{\HH}{H}
\newcommand{\br}{\mathbb{R}}

\newcommand{\ind}{\mathbbm{1}}
\newcommand{\pa}{\partial}

\newcommand{\la}{\lambda}

\newcommand{\cmin}{C_{\min}}

\newcommand{\be}{\begin{equation}}
\newcommand{\ee}{\end{equation}}
\def\bs#1\es{
    \begin{equation}\begin{split}
    #1
    \end{split}\end{equation}
}
\def\bsn#1\esn{
    \begin{equation*}\begin{split}
    #1
    \end{split}\end{equation*}
}
\newcommand{\g}{q}
\newcommand{\z}{w}
\renewcommand{\omega}{\z}
\renewcommand{\psi}{\phi}
\renewcommand{\delta}{\phi}
\renewcommand{\Theta}{\Phi}
\newcommand{\dd}{\,\mathrm{d}}
\newcommand{\E}{\mathbb E}
\newcommand{\Rem}{\mathrm{Rem}}
\newcommand{\Wat}{{\mathrm{Wat}}}
\newcommand{\Lap}{{\mathrm{Lap}}}
\newcommand{\CA}{\mathcal{A}}

\newcommand{\les}{\lesssim}
\newcommand{\us}{{\mathcal U_H}}
\newcommand{\U}{{\mathcal U_{H,1}}}

\newcommand{\BZ}{\mathbb Z}

\newtheorem{theorem}{Theorem}[section]
\newtheorem{lemma}[theorem]{Lemma}
\newtheorem{corollary}[theorem]{Corollary}
\newtheorem{proposition}[theorem]{Proposition}
\newtheorem{assumption}{Assumption}

\theoremstyle{definition}

\newtheorem{definition}[theorem]{Definition}
\newtheorem{remark}[theorem]{Remark}
\newtheorem{example}[theorem]{Example}

\begin{document}

\title[Rare event asymptotics in high dimensions]{Asymptotic analysis of rare events in high dimensions}
\author[A Katsevich]{Anya Katsevich, Alexander Katsevich}

\begin{abstract} 
Understanding rare events is critical across domains ranging from signal processing to reliability and structural safety, extreme-weather forecasting, and insurance. The analysis of rare events is a computationally challenging problem, particularly in high dimensions $d$. 
In this work, we develop the first asymptotic high-dimensional theory of rare events. First, we exploit asymptotic integral methods recently developed by the first author to provide an asymptotic expansion of rare event probabilities. The expansion employs the geometry of the rare event boundary and the local behavior of the log probability density. Generically, the expansion is valid if $d^2\ll\lambda$, where $\lambda$ characterizes the extremity of the event. We prove this condition is necessary by constructing an example in which the first-order remainder is bounded above and below by $d^2/\la$. We also provide a nonasymptotic remainder bound which specifies the precise dependence of the remainder on $d$, $\la$, the density, and the boundary, and which shows that in certain cases, the condition $d^2\ll \lambda$ can be relaxed. As an application of the theory, we derive asymptotic approximations to rare probabilities under the standard Gaussian density in high dimensions. In the second part of our work, we provide an asymptotic approximation to \emph{densities} conditional on rare events. This gives rise to simple procedure for approximately sampling conditionally on the rare event using independent Gaussian and exponential random variables.
\end{abstract}

\maketitle

\section{Introduction} 

Understanding rare events is critical across many domains—from reliability and structural safety \cite{choi2007reliability,ditlevsen1996structural} to finance and insurance~\cite{embrechts1997modelling}, signal processing and communications~\cite{BenRached2016UnifiedIS, Li2017LargeDeviations}, and extreme-weather forecasting~\cite{robRES}. Rare events also arise in the context of Bayesian inference, where extreme tails of the posterior distribution of a parameter given observed data are of interest~\cite{Straub2016BayesianRareEvents, FriedliLinde2024RareEvents}.

Rare event analysis is inherently extremely computationally challenging, since by definition, a faithful sampling of the full state space only rarely explores the rare event. The problem is even more challenging in high-dimensional state spaces. 

In this work, we analyze high-dimensional rare events using asymptotic techniques. We study the asymptotic framework of a concentrating probability density of the form $\pi(u)\propto e^{-\la z(u)}$, $u\in\br^{d+1}$, $\la\gg1$, and a ``fixed" (with respect to $\la$) event $D\subset\br^{d+1}$, with $d$ allowed to grow large with $\la$. In science and engineering applications, $\la$ is a relevant parameter controlling the rarity of the event. In the statistical context where $\pi$ is a posterior distribution, $\la$ denotes the number of independent observations. 

We study the standard setting in which the \emph{instanton} $u^*=\arg\min_{u\in D}z(u)$, the most likely point in $D$ under $\pi$, lies on the boundary $\pa D$. We consider $D$ to be a rare event not because $e^{-\la z(u^*)}$ is exponentially small but rather because $\la\|\nabla z(u^*)\|\gg1$, which makes $e^{-\la z(u)}$ decay rapidly as $u$ moves away from $u=u^*$ into the interior of $D$. Thus only an immediate neighborhood of the instanton $u^*$ contributes significantly to the probability of $D$. We derive asymptotic approximations of three quantities: 
\begin{enumerate}
\item the probability of the rare event $\pi(D)$, 
\item the expectation $\E_{X\sim\pi}[g(X)\mid X\in D]$ of some function $g$ defined on the state space (e.g. expected damage given a natural disaster occurs), and 
\item the probability \emph{density} $\pi\vert_D(u)=\pi(u)\ind_D(u)/\pi(D)$ itself. 
\end{enumerate}
We approximate these quantities in terms of simpler, explicit expressions involving the derivatives of $z$, $g$, and the boundary surface $\pa D$ at $u^*$. One of the uses for (3), i.e. for approximating $\pi\vert_D$ by a simpler $\hat\pi$, is to more cheaply explore the rare event landscape. We discuss other uses for $\hat\pi$ in Section~\ref{intro:sig} below.

We review the literature on the asymptotics of these three quantities. The first two have been studied both from the perspective of classical asymptotic integral analysis, and in the more specialized literature focusing on rare events. The quantities $\pi(D)$ and $\E_{U\sim\pi}[g(U)\mid U\in D]$ are ``classical" because they are given in terms of Laplace-type integrals, i.e. integrals with a large parameter $\la$ in the exponent:
\be\label{intro:piD}\pi(D)=\frac{\int_{D}e^{-\la z(u)}\dd u}{\int_{\br^{d+1}}e^{-\la z(u)}\dd u},\qquad \E_{U\sim\pi}[g(U)\mid U\in D] =\int gd\pi\vert_D= \frac{\int_{D}g(u)e^{-\la z(u)}\dd u}{\int_{D}e^{-\la z(u)}\dd u}.\ee The four Laplace-type integrals in these expressions fall into two categories: ones in which the minimizer of the exponent is on the boundary of the region of integration (the three integrals over $D$), and one in which the minimizer of the exponent lies in the interior (the integral over $\br^{d+1}$). In fixed $d$, there are well-known, rigorously justified and classical asymptotic expansions (AEs) for both categories of Laplace-type integrals; see e.g.~\cite[Chapter 8]{bleistein1975asymptotic}.

In the rare event literature, the asymptotic approach to computing $\pi(D)$ is standard in the field of reliability analysis~\cite{hu2021second}. There, a multitude of asymptotic methods, called ``second-order reliability methods" (SORMs) are of the following general type: first, the density $\pi$ is transformed into a standard Gaussian density by a change of variables, and second, the boundary of the transformed rare event set is approximated by a quadratic function near the instanton. See~\cite{hu2021second} for a review of SORMs. These methods are closely related to the classical theory of~\cite{bleistein1975asymptotic}; for example, the method given in~\cite{breitung1984asymptotic} directly applies to~\cite{bleistein1975asymptotic}.

Recent work has extended the classical theory of Laplace-type asymptotics with interior minimum (such as the denominator $\int_{\br^{d+1}}e^{-\la z(u)}\dd u$) to the high-dimensional regime; see~\cite{tang2025laplace} and~\cite{A24}. \emph{Yet there is no high-dimensional asymptotic theory for Laplace-type integrals $\int_{D}g(u)e^{-\la z(u)}\dd u$ with boundary minimum}, which are of course the more pertinent ones for rare event analysis.

There is significantly less work on asymptotic approximations of the third quantity, the restricted probability density $\pi\vert_D$ to a rare event $D$. We are only aware of the single work~\cite{lapinski2019multivariate}, which gives such an approximation in the case of fixed $d$ and a flat boundary $\pa D$. On the other hand, the analogous unrestricted problem is well-studied. Here, one approximates a Laplace-type density $\pi(u)\propto e^{-\la z(u)}$ by replacing $z$ with its quadratic Taylor approximation about the global minimizer of $z$. This yields a Gaussian density called the Laplace approximation (LA) to $\pi$, and it is known to be asymptotically exact. The LA is in widespread use in Bayesian inference for approximately sampling from posteriors, which are of Laplace-type, with $\la$ denoting sample size. The work~\cite{schillings2020convergence} establishes the accuracy of the LA in fixed dimension, and the works~\cite{helin2022non,bp,spokoiny2023dimension,katskew,katsBVM,katspok} have extended the theory of the LA into the high-dimensional regime. However, \emph{there is no high-dimensional asymptotic theory for approximating Laplace-type densities conditional on rare events $D$.}

Given the lack of asymptotic theory in the important high-dimensional regime, recent work on rare event algorithms has largely focused on simulation-based techniques, such as sequential sampling, multi-level splitting, and subset simulation~\cite{cerou2007adaptive,au2001estimation,cerou2012sequential}, as well as importance sampling techniques~\cite{kahn1953methods, siegmund1976importance}, which have in recent years been combined with the cross-entropy method to find a good proposal distribution~\cite{rubinstein2004cross, uribe2021cross, el2021improvement,tong2023large}. We note that importance sampling involves sampling from a proposal distribution $\hat\pi$ which, ideally, is close to the target $\pi\vert_D$. In this sense, the problem of approximating $\pi\vert_D$ by some simpler $\hat\pi$ has indeed been considered in the high-dimensional rare event literature. Yet there is no analysis in the literature of whether $\hat\pi$ approximates $\pi\vert_D$ in any quantifiable sense.

We now describe our results, which constitute {\bf the first asymptotic high-dimensional theory of rare events.}

\subsection{Summary of results}
We assume $\nabla z(u^*)$ is parallel to $(0_d,1)$ without loss of generality. Here $0_d$ refers to the zero vector in $\br^d$.  Thus, the tail $\pi(u)\propto e^{-\la z(u)}$ becomes more extreme as one moves upward in the $(d+1)$st coordinate away from $u^*$. Our assumptions imply that 
\begin{enumerate}
\item The boundary $\pa D$ can be locally parameterized as $(x, \psi(x))$ in a neighborhood of $u^*$, for some sufficiently smooth function $\psi:\br^d\to \br$. 
\item The boundary $\pa D$ is tangent to the level set $\{z(u)=z(u^*)\}$ at $u^*$. 
\item The set $D$ lies above the level set $\{z(u)=z(u^*)\}$, i.e. $(0_d,1)$ is an interior normal to $D$ at $u^*\in\pa D$.
\end{enumerate}
We consider the generic case where the level set and $\pa D$ are tangent to exactly second order. In other words, the difference between the two surfaces can be locally approximated by a positive-definite quadratic form $x^\top Hx/2$, with $H\succ0$. Finally, we assume $z$ grows at least linearly at infinity. Under these generic conditions, we prove several types of asymptotic results. 
\begin{enumerate}[label=\arabic*.]
\item\textbf{AE to order $N$.} In our first result, we derive an asymptotic expansion of $\int_Dg(u)e^{-\la z(u)}\dd u$ to arbitrary order $N$. We impose some additional restrictions on the magnitude of derivatives of the functions $\g(x,y):=g(x,\psi(x)+y)$ and $\z(x,y):=z(x,\psi(x)+y)$, where $x\in\br^d$ and $y\in\br$. Roughly speaking, we require that
\bs\label{intro:conds}
\max_{j= j_\ell,\dots,J_\ell}\|\nabla_x^\ell\pa_y^j\g(u)\|_H& \les d^{\lceil\frac\ell2\rceil},\\
\max_{j= j_\ell,\dots,J_\ell}\|\nabla_x^\ell\pa_y^j\z(u)\|_H &\les d^{\max(0,\lceil\frac\ell2\rceil-2)},
\es
uniformly over $u$ in certain small neighborhoods of $u^*$. Here, $\|\cdot\|_H$ is an $H$-weighted operator norm and $\lceil\cdot\rceil$ is the ceiling function. The maximum orders $J_\ell$ of $y$-partial derivative satisfy $J_\ell+\ell\leq 2N$ for $\g$ and $J_\ell+\ell\leq 2N+2$ for $\z$. Under these conditions, we prove that
\bs\label{intro:expand}
\int_{D}g(u)e^{-\la z(u)}\dd u =  \frac{(2\pi/\la)^{d/2}e^{-\la z(0_{d+1})}}{\la(\det H)^{1/2}}\left(\sum_{m=0}^{N-1}a_m(d^2/\la)^{m}+ \mathcal O((d^2/\la)^N)\right),
\es where the $|a_m|$ are bounded by constants depending only on the suppressed constants in~\eqref{intro:conds}. The terms of the expansion $a_m(d^2/\la)^m$ \emph{are the same as in the fixed $d$ case}. They are more commonly written as $\sum_mb_m\la^{-m}$, with $b_m=a_md^{2m}$, but we have pulled out a factor of $d^{2m}$ to emphasize that when $d$ is non-negligible, we have $|b_m|\les d^{2m}$. Thus~\eqref{intro:conds} and~\eqref{intro:expand} specify how large $d$, and the derivatives of $\g,\z$ can be to ensure the classical AE continues to hold. 
\item\textbf{Refined AE, $N=1$.} Our second result refines the first, in the case that $N=1$ and $g\equiv1$. We prove that
\be\label{intro:expand2}
\int_{D}e^{-\la z(u)}\dd u =  \frac{(2\pi/\la)^{d/2}e^{-\la z(0_{d+1})}}{\la\sqrt{\det H}}\left(1+ \Rem_1\right), \qquad|\Rem_1|\les C(\z)\frac{d^2}{\la}.
\ee This holds provided an explicit non-asymptotic inequality is satisfied, involving the derivatives of $\z$. The suppressed constant in the bound on $\Rem_1$ depends only on quantities which are not expected to grow with $d$ or $\la$, and $C(\z)$ is an explicit expression involving mixed partial derivatives of $\z$ in a neighborhood of $u^*$, of total order at most four. See Proposition~\ref{prop:explicitL2} for more details. This result refines~\eqref{intro:conds}-\eqref{intro:expand} in that it offers more flexibility. Indeed, the bound on the remainder $C(\z)d^2/\la$ is small under several asymptotic regimes. For example, the setting closest to~\eqref{intro:conds}-\eqref{intro:expand} is the case where $C(\z)$ is bounded and $d^2\ll\la$. But we could also have that $C(\z)\les d^{-1}$, in which case only $d\ll\la$ is required. Note that the quantity $C(\z)$ in turn depends on the derivatives of $z$ and the boundary $\psi$.
\item\textbf{Specialization to convex $z$ and $D$.} In the case that $z$ and $D$ are convex, we formulate simpler versions of the conditions under which~\eqref{intro:expand} and~\eqref{intro:expand2} hold, and the suppressed constant in~\eqref{intro:expand2} is now a function of a more explicit quantity.
\item\textbf{Rare high-dimensional Gaussian probabilities.} Next, we provide expansions for rare probabilities $\mathbb P(\mathcal N(0_{d+1}, I_{d+1})\in D_\la)$. Here, the density is fixed with $\la$, while the event $D_\la$ becomes more extreme as $\la\to\infty$. This is encoded by letting $D_\la = \sqrt\la D+(0_d,\sqrt\la)$ for a set $D$ which, in the most natural form of our results, only weakly depends on $\la$. Through a rescaling, we bring these probabilities into the form $\int_De^{-\la z(u)}\dd u$, divided by a Gaussian normalizing constant, and apply~\eqref{intro:expand},~\eqref{intro:expand2}. The remainder bounds are formulated in terms of derivatives of the boundary of $D$. For example, the specialization of~\eqref{intro:expand2} to the Gaussian case takes the following simple form:
\bs\label{intro:gauss}
\mathbb P\left(\mathcal N\left(0_{d+1}, \,I_{d+1}\right)\in  D_\la\right) &= \frac{e^{-\la/2}\left(1+\Rem_1\right)}{\sqrt{2\pi\la\det H}},\quad|\Rem_1| \les\left(\delta_2^2+\delta_3^2 +\delta_4( R\e)\right)\frac{d^2}{\la} + \frac1\la,
\es where $\delta_2,\delta_3$ are operator norms of $\nabla^2\psi$ and $\nabla^3\psi$ at the instanton $u^*$, and $\delta_4(R\e)$ is the supremum of the operator norm of $\nabla^4\psi$ in a neighborhood of $u^*$. Since~\eqref{intro:gauss} is a nonasymptotic bound, it can even be applied in the case where $D$ does depend strongly on $\la$, as we demonstrate in Example~\ref{ex:quad2}.
\item\textbf{Lower bound.} We give an example in which $\Rem_1$ in~\eqref{intro:expand2} satisfies $C_1d^2/\la\leq\Rem_1\leq C_2d^2/\la$ for some absolute constants $C_1,C_2$. This shows that our remainder bound is tight.
\item\textbf{Asymptotic approximation to $\pi\vert_D$.} Finally, we construct a probability density $\hat\pi$ which asymptotically approximates $\pi\vert_D$ under the same conditions under which~\eqref{intro:expand2} holds. We show that $\mathrm{TV}(\pi\vert_D, \hat\pi)\leq C(\z)d/\sqrt\la$, where $C(\z)$ is similar to the corresponding coefficient in~\eqref{intro:expand2}. The probability density $\hat\pi$ is constructed by Taylor expanding $\pi$ at the instanton and retaining quadratic terms in $x$ (the $d$-dimensional variable parameterizing the boundary) and a linear term in $y$ (the variable specifying vertical distance above the boundary). Thus, $\hat \pi$ can be viewed as a canonical approximation to $\pi$. Samples from $\hat\pi$ can be constructed by drawing a $d$-dimensional Gaussian $X$ and a scalar exponential random variable $Y$ independent of $X$, and returning $(X, Y+\hat\psi(X))$. Here, $\hat\psi$ is a quadratic approximation of the boundary $\psi$. Since $Y\geq0$, this construction ensures that the sampled points lie above the (approximate) boundary $(x,\hat\psi(x))$. The parameters of the Gaussian and exponential random variables depend purely on the first- and second-order derivatives of $z$ and $\psi$ at the instanton.
\end{enumerate}

\subsection{Significance}\label{intro:sig}

The significance of the high-dimensional asymptotic theory of rare events developed here is manifold. First, asymptotic methods yield explicit, analyzable formulas for the quantities of interest (e.g., rare-event probabilities). These expressions illuminate the underlying mechanisms—including boundary effects—and enable systematic sensitivity analysis with respect to physically meaningful parameters. By contrast, simulation-based techniques are essentially ``black-box'' tools that return numbers without structure, and in high dimensions the large number of parameters makes comprehensive exploration computationally impractical. Explicit formulas, on the other hand, allow one to identify the most influential parameters and to quantify their impact on target functionals in a direct and transparent way. Likewise, combining the explicit formulas with downstream analyses of the rare event enables a more comprehensive, fully analytic investigation of its mechanisms and impacts.

Second, asymptotic methods are computationally attractive, as they require only local derivative information at the instanton $u^*$. Of course, this requires first finding $u^*$, but this problem is common to the vast majority of numerical methods. Despite their popularity and widespread use in numerical rare-event analysis, \emph{no rigorous guarantees had previously been available} for these methods in the high-dimensional settings that arise in modern applications. Our theory substantially extends the validity of asymptotic expansions into this high-dimensional regime. In particular, it yields verifiable criteria for establishing that the remainder terms are small, thereby ensuring that the expansion may be safely employed. The bounds are not explicit only due to constants which are straightforward to quantify. We omit the computation of these constants in order to keep the paper's focus on the central concepts. 

We give a compelling example for the need for this high-dimensional theory in the setting of Gaussian rare event probabilities with a quartic boundary near the instanton. In this example, we prove that
\be\label{intro:quart-tight}
\mathbb P\left(\mathcal N\left(0_{d+1}, \,I_{d+1}\right)\in  D_\la\right) = \frac{e^{-\la/2}}{\sqrt{2\pi\la}}\left(1 -\frac{d^2+2d+24}{24\la} + \mathcal O((d^2/\la)^2)\right).
\ee Interestingly enough, the leading order term $\frac{e^{-\la/2}}{\sqrt{2\pi\la}}$ is independent of $d$, yet our expansion shows that it is truly leading order \emph{only if} $d^2\ll \la$. See Example~\ref{ex:quart} for more details.

Another significant feature of our results is the convenience of the asymptotic formulae we develop, when compared to the popular SORM. As mentioned above, SORM involves a coordinate transformation which turns the probability density into a standard Gaussian and correspondingly transforms the rare event set $D$. This transformation must be available analytically to carry out SORM, which approximates the boundary of the transformed set by a quadratic. But there are various numerical challenges in computing this transformation, as outlined e.g. in~\cite{normalization2020}. These challenges are only exacerbated in high-dimensional spaces. For example, the determinant of the Jacobian of the coordinate change may grow with $d$. In contrast, the formulae we derive require no such transformation, but rather, only knowledge of the derivatives of $z$ and the original boundary at the instanton. 

It is also notable that our results are flexible, encompassing multiple asymptotic regimes. As discussed above, the remainder bound~\eqref{intro:expand2} can be made small in multiple ways, depending on the interaction between $C(\z)$, $d$, and $\la$, which ultimately translates into interaction between $z$, the boundary $\psi$, $d$, and $\la$. In the special case of Gaussian probabilities, we obtain the even more transparent bound~\eqref{intro:gauss}, which is small under simple explicit conditions on $d$, $\la$, and the second, third, and fourth derivatives of $\psi$.

Finally, we discuss the significance of our approximation of $\pi\vert_D$ by $\hat\pi$. Samples from $\hat\pi$ are a cheap but provably accurate proxy for samples from $\pi$ conditional on the rare event $D$, and can be used to explore the different configurations (points $u$) which are likely to be observed under $\pi\vert_D$, beyond the instanton $u^*$ itself. Furthermore, i.i.d. samples $U_i\sim\hat\pi$ can be used to approximate expectations $\E_{U\sim\pi}[g(U)\mid U\in D]$, as in~\eqref{intro:piD}. Indeed, if $g$ is non-smooth (such as an indicator of a subset of $D$ near the instanton), then our integral expansion of $\int_D ge^{-\la z}$ does not apply, but we can instead use the approximation $\int gd\pi\vert_D\approx \int gd\hat\pi \approx \frac1N\sum_{i=1}^Ng(U_i)$. If $g$ is a bounded function, then the error in the first approximation is controlled by the TV distance between $\pi\vert_D$ and $\hat\pi$, which we have bounded, as described in point 6 above. The error in the second approximation can be made vanishingly small by taking $N$ sufficiently large. 

There is a third use for the samples $U_i$. Suppose the asymptotic approximation to $\pi(D)$ or $\int gd\pi\vert_D$ (based on integral expansions) is insufficiently accurate. For example, in the regime of our first result, this could happen if $d^2/\la$ is not sufficiently small. Although this can be remedied by expanding further, until the error scales as a sufficiently high power of $d^2/\la$, another approach could be even more computationally convenient. Namely, samples $U_i$ can be used in an importance sampling procedure to yield an unbiased estimator of the true quantity of interest. 

It is well known that the performance of the importance sampling estimator depends on the proximity of the proposal distribution (from which samples are generated) to the target $\pi\vert_D$. Our $\hat\pi$ has the advantage over other proposals in the literature, e.g.~\cite{uribe2021cross, el2021improvement,tong2023large}, in that it exploits asymptotic properties of $\pi\vert_D$ to obtain a canonical approximation.

\subsection{Related Work} Several works use large deviation theory (LDT), a closely related asymptotic framework, to analyze rare events. The LDT-based approach is particularly common in the setting of probability distributions over infinite-dimensional path space, e.g. random realizations of the solution to an SDE. See~\cite{grafke2019numerical} for an overview in this setting. LDT has been applied in finite-dimensional settings in~\cite{tong2023large, tong2021extreme, dematteis2019extreme}. The works dealing with infinite dimensions are not directly comparable to ours, while the works in finite dimensions do not quantity  the approximation error's dependence on the dimension.

\subsection*{Organization} The rest of the paper is organized as follows. In Section~\ref{sec:not} we introduce some notation and describe the problem setting. In Section~\ref{sec:gen}, we study asymptotic expansions of high-dimensional Laplace-type integrals with minimum on the boundary. In Section~\ref{sec:gauss}, we apply the results from Section~\ref{sec:gen} to derive asymptotic approximations to rare probabilities under the standard Gaussian density in high dimensions. In Section~\ref{sec:sample}, we present our results on the approximation to a density conditional on a rare event. Section~\ref{sec:proofs} outlines the proofs of the main results, and technical details are deferred to the Appendix.

\section{Notation, assumptions, and problem setting}\label{sec:not}

\subsection{Notation} In this section, we introduce some frequently used notation; additional notation will be introduced as needed.

The point $0_k$ denotes the origin in $\br^k$. The notation $a\les b$ means $a\leq Cb$ for an absolute constant $C$, and the notation $a\les_M b$ means $a\leq C(M)b$, where $M\mapsto C(M)$ is a monotone increasing function from $(0,\infty)$ to itself. For a vector $x\in\br^d$, and a symmetric positive definite matrix $H\in\br^{d\times d}$, let
\be\label{xH}
\|x\|_H = \sqrt{x^\top H x} = \|H^{1/2}x\|,
\ee where $\|\cdot\|$ indicates the standard Euclidean norm in $\br^d$. For a vector $u=(x,y)\in\br^{d+1}$, where $x\in\br^d$, let
\be\label{uH1}
\|u\|_{H,1}^2= \|x\|_H^2 + y^2 = x^\top Hx +y^2.
\ee
For a function $f\in C^k(\br^d)$, let $\nabla^kf(x)$ be the tensor of $k$th order derivatives,
$$
(\nabla^kf(x))_{i_1\dots i_k} = \pa_{x_{i_1}}\dots\pa_{x_{i_k}}f(x).$$ The inner product of $\nabla^kf(x)$ with a $k$-fold outer product of vectors $u_1,\dots,u_k$ is defined as
$$\langle \nabla^kf(x), u_1\otimes\dots\otimes u_k\rangle = \sum_{i_1,\dots,i_k=1}^d(\nabla^kf(x))_{i_1\dots i_k} (u_1)_{i_1}(u_2)_{i_2}\dots (u_k)_{i_k}.$$ 
We define the $H$-weighted operator norm of $\nabla^kf(x)$ to be
\be\label{TA}\|\nabla^kf(x)\|_H:=\sup_{\|u_1\|_H=\dots=\|u_k\|_H=1}\langle \nabla^kf(x), u_1\otimes \dots\otimes u_k\rangle=\sup_{\|u\|_H=1}\langle \nabla^kf(x), u^{\otimes k}\rangle.\ee The second equality is by~\cite[Theorem 2.1]{symmtens} and holds because $\nabla^kf(x)$ is a \emph{symmetric} tensor. 
\begin{example}We have
\bs\label{T-form}
\|\nabla f(x)\|_H &=\|H^{-1/2}\nabla f(x)\|,\\
\|\nabla^2f(x)\|_H &=\|H^{-1/2}\nabla^2 f(x)H^{-1/2}\|,
\es where the righthand sides are the standard Euclidean norm of a vector in $\br^d$ (first line) and the standard matrix operator norm (second line).
\end{example}
We note an important distinction between~\eqref{xH} and the first line of~\eqref{T-form}. Both $x$ and $\nabla f(x)$ are vectors in $\br^d$, but we view the latter as a 1-linear form acting on $\br^d$. Thus for points in $\br^d$ we have $\|\cdot\|_H =\|H^{1/2}\cdot\|$, whereas for 1-linear forms on $\br^d$, we have $\|\cdot\|_H =\|H^{-1/2}\cdot\|$. There will be no ambiguity because all 1-linear forms in the present work take the form of gradients. 

Finally, $\lfloor t \rfloor$ and $\lceil t \rceil$ denote the floor and ceiling functions, respectively:
\be
\lfloor t \rfloor:=\max\{n\in\BZ:n\le t\},\ \lceil t \rceil:=\min\{n\in\BZ:n\ge t\},\ t\in\br.
\ee

\subsection{Assumptions and problem setting}

Let $D\subset \br^{d+1}$ be a closed domain with a sufficiently smooth boundary. In this paper, we derive the asymptotics of integrals
\be\label{intD}
\int_{D}g(u)e^{-\la z(u)}du,\qquad d,\la\gg1,
\ee 
such that the minimizer of $z$ over $D$ lies on the boundary $\pa D$, and is not a local minimizer. 

By a translation and rotation, we may assume without loss of generality that the \emph{instanton}, i.e. the minimizer of $z$ over $D$, is the origin $0_{d+1}\in\pa D$, and that the gradient $\nabla z(0_{d+1})$ is parallel to $(0_d, 1)$. In fact, by redefining $\la\to c\la$ and $z\to z/c$, we further assume that $\nabla z(0_{d+1})=(0_d, 1)$. Due to the special role of the $(d+1)$st coordinate, we distinguish it with the letter $t$, writing points $u\in \br^{d+1}$ as $u=(x,t)$, with $x\in\br^d$ and $t\in\br$. We summarize the above discussion, and introduce some conditions on the set $D$, in Assumption~\ref{assume1} below. First, we introduce some notation.
\begin{definition}\label{def:set}
Given an open set $\Omega\subset\br^d$ and a function $\psi:\Omega\to\br$, define $\Theta:\Omega\times\br^+\to\Omega\times\br$ by
\be\label{Thetadef}\Theta(x,y)=(x,\psi(x)+y).\ee For a matrix $H\succ0$, let
\bs\label{UH}
\us(r)&=\{x\in\br^d\,:\,\|x\|_H\leq r\},\\
\U(r)&=D\cap\{u\in\br^{d+1}\,:\,\|u\|_{H,1}\leq r\},
\es where $D$ is the rare event set of interest. Here, $\|\cdot\|_H$ and $\|\cdot\|_{H,1}$ are as in~\eqref{xH} and~\eqref{uH1}, respectively.
\end{definition}
\begin{assumption}\label{assume1}Suppose
\begin{enumerate}[label=\theassumption.\arabic*., ref=\theassumption.\arabic*]
\item $z\in C^2(D)$, $\arg\min_Dz(u)=0_{d+1}\in\pa D$, and $\nabla z(0_{d+1}) = (0_d, 1).$ \label{a11}
\item There exists an open neighborhood $\Omega$ of $0_d$ and a function $\psi:C^2(\Omega)\to\br$ such that $\psi(0_d)=0$, $\nabla\psi(0_d)=0_d$, and $\pa D$ is locally parameterized by $\psi$: $\pa D\cap (\Omega\times\br) = \{(x,\psi(x))\,:\,x\in\Omega\}$.\label{a12}
\item $H:=\nabla_x^2z(0_{d+1})+\nabla^2\psi(0_d)\succ0$. \label{a13}
\item  \label{a14}There exists $\rho_0>0$ such that $\us(\rho_0)\subseteq\Omega$ and
$$
\Theta(\us(\rho_0)\times(0,\rho_0))\subseteq D\cap (\us(\rho_0)\times\br) \subseteq \Theta(\us(\rho_0)\times(0,\infty)).
$$
\item By shrinking $\rho_0$ if necessary,  assume $\pa_tz(u)>0$ for all $u\in \U(\rho_0)$.\label{a15}
\end{enumerate}
\end{assumption}
Since $D$ is closed, by $z\in C^2(D)$ we mean that there exist an open set $D_1\supset D$ and a function $z_1\in C^2(D_1)$ such that $z_1\equiv z$ on $D$. Hence, all the derivatives of $z$ at the boundary $\pa D$ can be computed using $z_1$ instead.
The condition~\ref{a14} states that the portion of $D$ in the cylinder $\{(x,t)\,:x\in\us(\rho_0)\}$ lies entirely above the boundary $(x,\psi(x))$, and it must at least contain a strip of thickness $\rho_0>0$ above the boundary; see the curved blue rectangle in the top portion of Figure~\ref{fig:glob}. Since $z\in C^2(D)$ and $\pa_tz(0_{d+1})=1$ is positive, we can always find $\rho_0$ such that $\pa_tz(u)>0$ for all $u\in \U(\rho_0)$. That $\nabla\psi(0_d)=0_d$ is not an assumption at all, but follows from the fact that $0_d$ is the minimizer of $x\mapsto z(x,\psi(x))$ and $\nabla_xz(0_d)=0_d$. The matrix $H$ is the Hessian of $x\mapsto z(x,\psi(x))$ at $x=0_d$, so $H\succ0$ conveys that $0_{d+1}$ is the strict local minimizer of $z\vert_{\pa D}$.

\begin{remark}[Geometric interpretation of $H\succ0$]The boundary $(x,\psi(x))$ of $D$ is tangent to the level set $\{z(u)=z(0_{d+1})\}$ at $u=0_{d+1}$. Thus if this level set is locally parameterized as $(x,\psi_z(x))$, then $\nabla\psi_z(0_d)=0_d$. We then have $\psi(x)\approx\frac12x^\top\nabla^2\psi(0_d)x$ and $\psi_z(x)\approx \frac12x^\top\nabla^2\psi_z(0_d)x$ near the origin. Now, it is straightforward to show that $H$ from condition~\ref{a13} can also be expressed as $H=\nabla^2\psi(0_d)-\nabla^2\psi_z(0_d)$. Thus $H\succ0$ conveys the fact that points $u\in\pa D$ near the origin must lie \emph{above} the level set $\{z(u)=z(0_{d+1})\}$ and therefore $z(u)>z(0_{d+1})$, since $z$ increases as we move upward. This is precisely what is needed to ensure $z\vert_{\pa D}$ has a strict local minimizer at the origin. 
\end{remark}
\begin{remark}[General case] In Appendix~\ref{app:gen}, we consider the more general case $\int_{\bar D}e^{-\bar\la\bar z(u)}\dd u$, where $u^*=\arg\min_{u\in \bar D}z(u)\in\pa \bar D$ is not necessarily the origin, and $\nabla z(u^*)$ is not necessarily aligned with the $t$-axis, nor has unit norm. We assume we have a description of the set $\bar D$ as $\bar D=\{u\in\br^{d+1}\,:F(u)\geq F_0\}$, as is standard in reliability engineering~\cite{choi2007reliability}. By a rotation and translation, we show how to bring the integral into the form $\int_De^{-\la z(u)}\dd u$, where $z$ and $D$ satisfy Assumption~\ref{a11}, and we give the formula for the matrix $H$ in terms of the derivatives of $F$ and $\bar z$.\end{remark}
\begin{figure}[h]
{\centerline{
{\epsfig{file={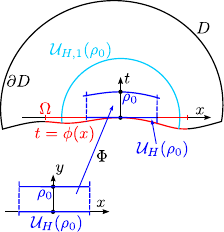}, width=6cm}}
}}
\caption{Illustration of various sets and the map $\Phi$ used in Assumption~\ref{assume1}.}
\label{fig:glob}
\end{figure}
%
Next, we assume $z$ grows linearly away from the origin, and that the linear growth starts sufficiently close to the origin.
\begin{assumption}[Global growth condition]\label{assume:global2}
There exists $\rho_1\in(0,\rho_0)$ and $s>0$ such that
\be
z(u)-z(0_{d+1})\geq s\|u\|_{H,1}\qquad\forall u\in D\setminus\U(\rho_1).
\ee  
\end{assumption}
The condition $\rho_1<\rho_0$ ensures that $\pa_tz(u)>0$ for points $u$ on the boundary $\{\|u\|_{H,1}=\rho_1\}$, beyond which the linear growth starts. Next, we define functions $\z$ and $\g$, which are compositions of $z$ and $g$ with the map $\Theta$, respectively.
\begin{definition}\label{def:wq}
Let $\rho_0$ be as in Assumption~\ref{assume1} and define $\z,\g:\us(\rho_0)\times[0,\rho_0)\to\br$ by
\bs\label{zg}
\z(x,y)&=z(x,\psi(x)+y),\\
\g(x,y)&=g(x,\psi(x)+y).
\es
\end{definition} 
Here, recall that $g$ is the observable from~\eqref{intD}, which we assume to be defined on $D$. The coefficients in the asymptotic expansion (AE) of the integral~\eqref{intD} will be defined in terms of coefficients of standard Laplace integral expansions. To that end, we have the following definition.
\begin{definition}\label{def:lap}
Recall $\Omega$ from Assumption~\ref{assume1}, and let $\z$ be as in Definition~\ref{def:wq}. For a function $h$, let $\nu_j(h)$ be the $j$th order term in the standard Laplace expansion of $\int_\Omega h(x)e^{-\la\z(x,0)}\dd x$. In other words,
\be\label{intlap}
\int_{\Omega}h(x)e^{-\la\z(x,0)}\dd x = e^{-\la\z(0_{d+1})}\frac{(2\pi/\la)^{d/2}}{\sqrt{\det H}}\left(\nu_0(h)+\nu_1(h)\la^{-1}+\nu_2(h)\la^{-2}+\dots\right)
\ee
For example, $\nu_0(h)=h(0_d)$. See~\cite{A24} and Appendix~\ref{app:lap} for more details.
\end{definition} Note that~\eqref{intlap} is indeed a classical Laplace-type integral, since the function $x\mapsto\z(x,0)=z(x,\psi(x))$ in the exponent has a single, strict global minimizer $x=0_d$ in the region $\Omega$ of integration. 

The coefficients in the AE of the integral~\eqref{intD} are defined as follows:
\begin{align}
 a_m &= d^{-2m} \sum_{k=1}^{m+1}\nu_{m+1-k}(\g_k),\qquad m=1,2,3,\dots,\label{amdef}\\
\g_k(x)&=\frac{1}{\pa_y\z(x,0)}(D_{\z}^{k-1}\g)(x,0),\quad k=1,2,3,\dots\,.
\label{gkdef1}
\end{align}
Here $D_{\z}$ is the operator defined by
\be\label{Dzdef}
(D_{\z}f)(x,y) =\pa_y(f(x,y)/\pa_y\z(x,y)),
\ee 
where $f$ is a function defined on $\us(\rho_0)\times[0 , \rho_0]$ which is differentiable with respect to $y$. According to the clarification following Assumption~\ref{assume1}, the derivatives of $\z$ at the boundary $y=0$ are computed by replacing $z$ with $z_1$ in \eqref{zg}. Derivatives of $\g$ at the boundary are handled similarly by using an extension of $g$.

To give an example of the $q_k$, we have
\be\label{g12def}
\g_1(x) = \frac{\g(x,0)}{\pa_y\z(x,0)},\qquad \g_2(x) =  \frac{\pa_y\g(x,0)}{\left(\pa_y\z(x,0)\right)^2}- \frac{\g(x,0)\pa_y^2\z(x,0)}{\left(\pa_y\z(x,0)\right)^3}
\ee 

Our remainder bound will be expressed in terms of the following quantities.
\begin{definition}[Derivative bounds on $\z$, $\g$, and $\psi$]\label{def:omega}For all $0\leq r,r'\leq\rho_0$, let
\bs\label{deldef}
\g_{0,j}(r,r') &=\sup_{u\in \us(r)\times[0,r']}|\pa_y^j\g(u)|,\\
\g_{\ell,j}(r,r') &= \sup_{u\in \us(r)\times[0,r']}\|\nabla_x^\ell\pa_y^j\g(u)\|_H, \quad\ell\geq1,\\
\omega_{0,j}(r,r') &= \sup_{u\in \us(r)\times[0,r']}|\pa_y^j\z(u)|,\\
\omega_{\ell,j}(r,r') &= \sup_{u\in \us(r)\times[0,r']}\|\nabla_x^\ell\pa_y^j\z(u)\|_H,\quad\ell\geq1,\\
\delta_0(r)&=\sup_{x\in\us(r)}|\psi(x)|,\\
 \delta_\ell(r)&=\sup_{x\in\us(r)}\|\nabla^\ell\psi(x)\|_H,\quad \ell\geq1.
\es
\end{definition}
Finally, we introduce one more quantity, which controls the behavior of $z$ in the region $\U(\rho_1)$. Recall that we have linear growth in the complement of this region, by Assumption~\ref{assume:global2}.
\be\label{C1def}
\cmin = \min\left(\inf_{u\in \U(\rho_1)}\pa_tz(u),\; \;\inf_{x\in\us(\rho_1)}(z(x,\psi(x))-z(0_{d+1}))/\tfrac12\|x\|_H^2\right).
\ee
Note that $\cmin>0$ since $\pa_tz$ and $(z(x,\psi(x))-z(0_{d+1}))/(\|x\|_H^2/2)$ are both positive and continuous in the regions over which we take their infimum. Continuity follows by the fact that $z\in C^2(D)$, and $H$ is the Hessian of $x\mapsto z(x,\psi(x))$ at $x=0_d$. Positivity follows because $0_{d+1}$ is the global minimizer of $z$ over $D$, and because we assumed $\rho_1<\rho_0$ in Assumption~\ref{assume:global2}, and $\pa_tz(u)>0$ for all $u\in \U(\rho_1)$, by Assumption~\ref{assume1}. Note that the closer $\rho_1$ is to zero, the closer $\cmin$ is to 1 (recall $\pa_tz(0_{d+1})=1$). 

\section{Main integral expansions}\label{sec:gen}

In the below proposition and throughout the paper,
\be
\e=\sqrt{d/\la}.
\ee

  \begin{proposition}\label{prop:main}
Let $N\geq1$, $\sup_{u\in D}|g(u)|\leq1$, and suppose Assumptions~\ref{assume1},~\ref{assume:global2} are satisfied. Suppose there is $R \geq 12/\cmin + 2(1+2N)(\log\la)/d$ and constants $\CR$, $C_{\g,\ell}$, $\ell=0,\dots,2N$, $C_{\z,\ell}$, $\ell=0,\dots,2N+2$ 
such that $\g\in C^{2N}(\us(R\e)\times[0,(R\e)^2))$, $\z\in C^{2N+2}(\us(R\e)\times[0,(R\e)^2))$, and
\begingroup
\addtolength{\jot}{0.3em}
\begin{align}
&\e\leq s,\qquad R\e\leq\min(1,\rho_0,2\rho_1, 2/\delta_2(\rho_1)),\label{A}\\
&R\e\leq1/(2C_{\z,0}+2C_{\z,1}),\qquad R^4d^2/\la\leq \CR,\label{BC}\\
\begin{split}\label{forq}
&\max_{j=0,\dots,N}\g_{0,j}(R\e,(R\e)^2)\leq C_{\g,0}, \\
&\max_{j=0,\dots,\lfloor N-\frac\ell2\rfloor}\g_{\ell,j}(R\e,0)\leq C_{\g,\ell} \,d^{\lceil\frac\ell2\rceil},\quad\ell=1,\dots,2N,
\end{split}\\
\begin{split}\label{forw}
&\omega_{0,2}(R\e,R\e)\vee \max_{j=3,\dots,N+1}\omega_{0,j}(R\e,(R\e)^2)\leq C_{\z,0},\\
&\max_{j=j_\ell,\dots,\lfloor N+1-\frac\ell2\rfloor}\omega_{\ell,j}(R\e,0)\leq C_{\z,\ell} \,d^{(\lceil\frac\ell2\rceil-2)_+},\quad\ell=1,\dots,2N+2.
\end{split}
\end{align}
\endgroup
Here, $j_1=j_2=1$, $j_\ell=0$ for $\ell\geq3$, and $(\cdot)_+=\max(\cdot, 0)$.  Then
\bs\label{expansion}
\int_{D}g(u)e^{-\la z(u)}\dd u =  \frac{(2\pi/\la)^{d/2}e^{-\la z(0_{d+1})}}{\la(\det H)^{1/2}}\left(g(0_{d+1})+\sum_{m=1}^{N-1}a_m\left(\frac{d^2}{\la}\right)^{m}+ \Rem_N\right),
\es where
\begin{align}
\max_{m=1,\dots,N-1}|a_m|&\leq C_{\g,\z,N},\qquad |\Rem_N|\leq C_{\g,\z,N,R} (d^2/\la)^{N}.\label{RemLmainbd}
\end{align}
The constant $C_{\g,\z,N}$ depends on $N$, $C_{\g,\ell}$, $\ell=0,\dots,2N$, and $C_{\z,\ell}$, $\ell=0,\dots,2N+2$. The constant $C_{\g,\z,N,R}$ depends on these same constants and also on $\CR$.
 \end{proposition}
 \begin{remark}[Asymptotic formulation]\label{rk:asym}
This is a nonasymptotic statement, but can be understood in the following asymptotic set-up. Suppose we have a sequence $\la\to\infty$, dimensions $d=d_\la$, functions $g=g_\la$, $z=z_\la$, and regions $D=D_\la\subset\br^{d_\la}$, such that Assumptions~\ref{assume1},~\ref{assume:global2} are satisfied for each $(\la,d_\la)$, and that the constants $\rho_0,\rho_1, \cmin,s,\delta_2(\rho_1)$ can be fixed independently of $\la$. Furthermore, suppose~\eqref{forq},~\eqref{forw} can be satisfied for some constants $C_{\g,\ell},C_{\z,\ell}$ which are also independent of $\la$. Finally, suppose $d_\la/\la\to0$ and $d_\la^2/\la$ remains bounded by a constant as $\la\to\infty$, and take $R=12/\cmin + 2(1+2N)(\log\la)/d$. Then the conditions~\eqref{A},~\eqref{BC} will be satisfied for all $\la$ large enough, with $\CR=C(\cmin, N)$. Thus the conclusion~\eqref{expansion}-\eqref{RemLmainbd} holds for all $\la$ large enough. Note that, while~\eqref{expansion}-\eqref{RemLmainbd} is valid if $d_\la^2/\la <C$ for any arbitrary $C$, the expansion is useful when $d_\la^2/\la \ll 1$.
\end{remark}
\begin{remark}[Tight remainder bound]
Taking $N=2$ gives $\Rem_1 = a_1d^2/\la + \Rem_2$, with $|a_1|\leq C$ and $|\Rem_2|\les (d^2/\la)^2$. At the end of Section~\ref{sec:gauss}, we give an example where $a_1$ is not just bounded above by a constant but also bounded \emph{below} by a constant. This allows us to conclude that $\Rem_1\asymp d^2/\la$. Thus our remainder bound is tight. \end{remark}
\begin{remark}[The conditions~\eqref{A},~\eqref{BC}]
The conditions~\eqref{A},~\eqref{BC} are needed for a variety of reasons. To give one example, $s\geq\e$ gives a lower bound on the linear growth of $z$ at infinity, which is needed to control a tail integral. The full rationale for all these conditions can be found in the proof of Lemma~\ref{lma:mu} in Appendix~\ref{app:secgenproofs}; see also the discussion following the statement of Lemma~\ref{lma:mu}.
\end{remark}
  \begin{remark}[Base case $N=1$] The conditions~\eqref{forq} and~\eqref{forw} are nested as $N$ increases: the conditions for order $N'$ imply the conditions for order $N<N'$. Thus in particular, we always require the conditions from the $N=1$ case. It is helpful to write these out explicitly:
  \be\label{derivcondL2}\begin{gathered}
\g_{0,1}(R\e,(R\e)^2)\leq C, \quad (\g_{1,0}+\g_{2,0})(R\e,0)\les d,\\
\omega_{0,2}(R\e,R\e) +(\omega_{1,1}+\omega_{2,1}+\omega_{3,0}+\omega_{4,0})(R\e,0)\leq C,
 \end{gathered}\ee where we have used the same letter $C$ to denote all the different constants, for simplicity. 
 \end{remark}
%
\begin{remark}[Origin and interpretation of~\eqref{forq} and~\eqref{forw}]
The basic proof strategy is to first integrate $g(x,t)\exp(-\la (z(x,t)-z(x,\psi(x)))$ in the $t$ direction and apply Watson's lemma, since the exponent is approximately linear in $t$. This leads to an AE of the form 
$$
\sum_{k\ge1}\la^{-k}\int_\Omega\g_k(x)e^{-\la z(x,\psi(x))}\dd x, 
$$
where the functions $\g_k$ are as in~\eqref{gkdef1}. As discussed below Definition~\ref{def:lap}, each integral in this sum is a standard Laplace-type integral, with global minimizer of the exponent lying within the domain of integration. We then apply Corollary 2.17 of~\cite{A24}, which imposes conditions on the derivatives of $\g_k$ and $\z(x,0)=z(x,\psi(x))$ to deduce an AE of $\int\g_k(x)e^{-\la z(x,\psi(x))}\dd x$ in powers of $d^2/\la$. We then deduce conditions on the derivatives of $\g$ and $\z$ which imply the conditions on $\g_k$, and combine them with the conditions on $\z$. This is how~\eqref{forq} and~\eqref{forw} come about.


The first line of~\eqref{forq} and the first line of~\eqref{forw} impose bounds on pure $y$-partial derivatives in a neighborhood within the interior of $D$, whereas the second lines of~\eqref{forq} and~\eqref{forw} impose bounds on mixed $x,y$ partial derivatives containing at least one $x$ derivative, but only on the boundary of $D$. The region $\us(R\e)\times[0,(R\e)^2]$ is relevant because $\Theta(\us(R\e)\times[0,(R\e)^2])$ gives the dominant contribution to the integral $\int_Dge^{-\la z}$, whereas the integral in the complement of this region contributes only exponentially small terms. Note that there is one exception to the required pure-$y$ derivative bounds: namely, we require a bound on $\omega_{0,2}$ over the larger strip $\us(R\e)\times [0,R\e]$. We use this to show $\pa_tz(u)\geq1/2$ in $\Theta(\us(R\e)\times[(R\e)^2,R\e])$, which guarantees that the integral over this region is exponentially small.

Note that the conditions~\eqref{forq},~\eqref{forw} indirectly restrict the derivatives of the three functions $g$, $z$, and $\psi$, since $\g$ and $\z$ are compositions of $g$ and $z$ with $\Theta:(x,y)\mapsto(x,\psi(x)+y)$, respectively. If the boundary is flat, i.e. $\psi\equiv0$, then $\g=g$ and $\z=z$, so that~\eqref{forq},~\eqref{forw} constitute explicit bounds on the derivatives of $g$ and $z$. Another special case is considered in Section~\ref{sec:gauss} on Gaussian probabilities; there, $g\equiv1$ and $z$ is quadratic. The derivative bounds~\eqref{forq},~\eqref{forw} then reduce to bounds on $\psi$ derivatives alone; see~\eqref{deltas:gauss}.
\end{remark}

\begin{remark}[Relaxing~\eqref{forq},~\eqref{forw}]\label{rk:gen}
As we have just discussed, the conditions~\eqref{forq},~\eqref{forw} on $\g,\z$ derive from conditions on $\g_k,\z$ which, by Corollary 2.17 of~\cite{A24}, give an expansion of $\int\g_k(x)e^{-\la z(x,\psi(x))}\dd x$ in powers of $d^2/\la$. The conditions on $\g_k,\z$ are stated in~\eqref{w-deriv-cond},~\eqref{qk-deriv-cond}. Yet Corollary 2.17 actually contains additional flexibility. Specifically, let $\tau>0$ be some other possibly large parameter. If instead of~\eqref{w-deriv-cond},~\eqref{qk-deriv-cond}, it holds 
$$
\|\nabla_x^\ell\z(x,0)\|_H\les \tau^{\ell-2}d^{\lceil\ell/2\rceil-2},\qquad \|\nabla_x^\ell\g_k(x)\|_H\les \tau^\ell d^{\lceil\ell/2\rceil}$$ for an appropriate range of $\ell$ and $k$, then $\int\g_k(x)e^{-\la z(x,\psi(x))}\dd x$ admits an AE in powers of $\tau^2d^2/\la$, provided $\tau^2d^2/\la\ll1$. For example if $\tau=\sqrt d$ and $d^3\ll 1$, we obtain an AE in powers of $d^3/\la$. Thus by requiring a more stringent condition on $\la$, we can relax the condition on the growth of the derivatives of $\z$ and $\g_k$. In turn, this translates to weaker conditions on the derivatives of $\z$ and $\g$. We leave the formulation of these weaker conditions to future work.

A second avenue for weakening~\eqref{forq},~\eqref{forw} is discussed in Remark~\ref{rk:roomd}.
\end{remark}

\begin{remark}[Growth of derivative tensor norms with $d$]\label{rk:d1}Note that we allow for derivative tensor norms to grow as powers of $d$, where the power can depend on the order of the $x$-partial derivative but not on the order of the $y$-partial derivative. This is reasonable since $y$ is scalar and $x$ is $d$-dimensional. Indeed, using the number of entries as a rough proxy for order of magnitude of a tensor norm, note that the number of entries of $\nabla_x^\ell\pa_y^j\z$ is $d^\ell$, which depends on $\ell$ but not $j$. 
\end{remark}


 \begin{proposition}[Convex case]\label{main:cvx}Let $N\geq1$ and $\sup_{u\in D}|g(u)|\leq1$. Suppose Assumption~\ref{assume1} is satisfied and $D$ is star-shaped about $0_{d+1}$. Also, suppose the restriction of $z$ to any line segment in $D$ emanating from $0_{d+1}$ is convex. Let $R=24+2(1+2N)(\log\la)/d$ and suppose $\g\in C^{2N}(\us(R\e)\times[0,(R\e)^2))$, $\z\in C^{2N+2}(\us(R\e)\times[0,(R\e)^2))$. Suppose there are constants $\CR$, $C_{\g,\ell}$, $\ell=0,\dots,2N$ and $C_{\z,\ell}$, $\ell=0,\dots,2N+2$ such that
 \begin{align}
 R\e&\leq\min(1,\rho_0),\qquad R\e\delta_2(R\e/2)\leq 2,\label{D}\\
 R\e &\leq 3/C_{\z,3},\label{E}
 \end{align}
and~\eqref{BC},~\eqref{forq},~\eqref{forw}, are satisfied.
 Then the conclusion of Proposition~\ref{prop:main} follows.
 \end{proposition} 
 
 \begin{remark}The difference between these assumptions and those of Proposition~\ref{prop:main} is that here, we have dropped Assumption~\ref{assume:global2} and replaced~\eqref{A} by~\eqref{D}. We have also added the condition~\eqref{E}. The reason we do not need Assumption~\ref{assume:global2} is that convex functions $z$ automatically grow linearly, starting at any point. Correspondingly, the bounds from~\eqref{A} involving $\rho_1,s$ are not needed. The condition~\eqref{E} and the new bounds in~\eqref{D} relative to~\eqref{A}, are needed to ensure the slope of linear growth is sufficiently large.
 \end{remark}


In the rest of the paper, we focus on the case $g\equiv1$. We now give another more explicit bound in the case $N=1$, which makes clear the dependence of the remainder on the derivatives of $\z$. Recall Definition~\ref{def:omega} of $\omega_{\ell,j}(\cdot,\cdot)$. We introduce the shorthand
$$
\omega_{\ell,j}=\omega_{\ell,j}(0,0).
$$
\begin{proposition}\label{prop:explicitL2}
Let $M\geq1$. Suppose Assumptions~\ref{assume1},~\ref{assume:global2} are satisfied and there is some $R \geq 12/\cmin+2(1+2M)(\log\la)/d$ and a constant $\CL$ such that $\z\in C^4(\us(R\e)\times[0,(R\e)^2])$,~\eqref{A} is satisfied, and 
\bs\label{Rwc34eps}
&R\e(\omega_{1,1}+\omega_{0,2}(R\e,R\e)) + (R\e)^2\omega_{2,1}(R\e,0)\leq1/2,\\
&\left(\omega_{3,0}^2 +\omega_{4,0}(R\e,0)\right)\frac{d^2}{\la}  \leq \CL.
\es
Then
\begingroup
\addtolength{\jot}{0.7em}
\be\label{expansion-Rem2-g1}
\int_{D}e^{-\la z(u)}\dd u =  \frac{(2\pi/\la)^{d/2}e^{-\la z(0_{d+1})}}{\la\sqrt{\det H}}\left(1+ \Rem_1\right),
\ee \endgroup where 
\bs\label{bd-Rem2-g1}
|\Rem_1| \les_{R,\CL} &\left(\omega_{3,0}^2+\omega_{4,0}(R\e,0)\right)\frac{d^2}{\la} + \left(\omega_{1,1}^2+\omega_{2,1}(R\e,0)\right)\frac{d}{\la}+ \omega_{0,2}(R\e,(R\e)^2)\frac{1}{\la}  + \frac{1}{\la^M}.
\es
\end{proposition}
\begin{proposition}[Convex version]\label{L2:cvx}
Suppose Assumption~\ref{assume1} is satisfied, and $D$ is star-shaped about $0_{d+1}$. Also, suppose the restriction of $z$ to any line segment in $D$ emanating from $0_{d+1}$ is convex. Let $M\ge1$, $R= 24+2(1+2M)(\log\la)/d$, suppose $\z\in C^4(\us(R\e)\times[0,(R\e)^2])$, \eqref{D} holds, and \eqref{Rwc34eps} holds for some $C_{34}$. Finally, suppose $R\e\omega_{3,0}(R\e,0)\leq 3$. Then the conclusion of Proposition~\ref{prop:explicitL2} follows.\end{proposition}

As is seen from the remainder bound \eqref{bd-Rem2-g1}, the expansion \eqref{expansion-Rem2-g1} is useful when $\Delta:=\omega_{3,0}^2 +\omega_{4,0}(R\e,0)d^2/\la$ is small. The inequality in the second line of \eqref{Rwc34eps} is then automatically satisfied if $\CL\sim1$. We nevertheless impose this inequality as a requirement because it allows us to simplify an intermediate bound on $\Rem_1$. The intermediate bound depends on functions of $\Delta$ which do not vanish if $\Delta$ is small (e.g. $\exp(\Delta)$). Thus they contribute only constant factors depending on $\CL$ and can therefore be suppressed in the final bound~\eqref{bd-Rem2-g1}.


\begin{remark}In Propositions~\ref{prop:explicitL2} and~\ref{L2:cvx}, $M$ is a free parameter which can be chosen to ensure $\la^{-M}$ is negligible compared to the rest of the remainder bound~\eqref{bd-Rem2-g1}.\end{remark}
\begin{remark}The suppressed constant in~\eqref{bd-Rem2-g1} depends on $R$, which in turn depends on $\cmin^{-1}$. It may be possible that in some cases, $\cmin$ is very small (if the linear growth in Assumption~\ref{assume:global2} starts very far from the origin), and it may even depend on $d$. A convenient feature of Proposition~\ref{L2:cvx} in the convex case is that $R$ no longer depends on $\cmin^{-1}$, and is bounded as long as $M$ and $\log\la/d$ are bounded.
\end{remark}
\begin{remark}It is helpful to interpret Proposition~\ref{prop:explicitL2} in the following asymptotic set-up, which differs slightly from the asymptotic set-up discussed in Remark~\ref{rk:asym}. Suppose we have a sequence of dimensions $d\to\infty$, a corresponding sequence of domains $D=D_d\subseteq\br^d$ and functions $z=z_d$ on $\br^d$ such that Assumptions~\ref{assume1},~\ref{assume:global2} are satisfied. We assume $d,D,z$ are independent of $\la$, and therefore the quantities $\rho_0,\rho_1, \cmin,s,\delta_2(\rho_1)$ all depend on $d,D,z$ but not on $\la$. Now, fix any two arbitrary constants $M$ and $\CL$. Then \emph{by taking $\la$ large enough}, and setting $R=12/\cmin+2(1+2M)(\log\la)/d$, we can always ensure~\eqref{A} and~\eqref{Rwc34eps} are satisfied. Proposition~\ref{L2:cvx} can be interpreted similarly. The take-away is that Propositions~\ref{prop:explicitL2} and~\ref{L2:cvx} hold for all $\la$ large enough relative to $d, D,z$. Remark~\ref{rk:brd} discusses precisely how large $\la$ needs to be in various asymptotic regimes.
\end{remark}
\begin{remark}\label{rk:brd}Propositions~\ref{prop:explicitL2} and~\ref{L2:cvx} are more broadly applicable than Propositions~\ref{prop:main} and~\ref{main:cvx} in the following two respects. First, suppose $\omega_{3,0},\omega_{4,0}(R\e,0)$ actually \emph{decay} with $d$. This can happen e.g. if the boundary $\psi$ is close to flat and $z$ is quadratic. Then to satisfy the second line of~\eqref{Rwc34eps}, and in order for the remainder bound~\eqref{bd-Rem2-g1} to be small, we do \emph{not} need $\la$ to be much larger than $d^2$. For example if $\omega_{3,0}\asymp d^{-1/2}$ and $\omega_{4,0}(R\e,0)\asymp d^{-1}$, then it suffices for $d\ll \la$. This is in contrast to Propositions~\ref{prop:main} and~\ref{main:cvx}, for which $d^2\ll\la$ is required for the AE to be meaningful. Second, suppose instead that $\omega_{3,0},\omega_{4,0}(R\e,0)$ \emph{grow} as a power of $d$. This setting is not allowed by Propositions~\ref{prop:main},~\ref{main:cvx}, as can be seen from the second line of~\eqref{derivcondL2}. However, it is allowed by Propositions~\ref{prop:explicitL2} and~\ref{L2:cvx}. Indeed, if for example $\omega_{3,0}\asymp d^{1/2}$ and $\omega_{4,0}(R\e,0)\asymp d$, then we can still satisfy the second line of~\eqref{Rwc34eps}, and ensure that the remainder bound~\eqref{bd-Rem2-g1} is small, by taking $\la\gg d^3$.
\end{remark}
\begin{remark}[Dominant error contribution]If $\omega_{1,1},\omega_{2,1}(R\e,0), \omega_{0,2}(R\e,(R\e)^2)$ can all be bounded above by a constant independent of $d,\la$, and $\omega_{3,0},\omega_{4,0}(R\e,0)$ can be bounded below by a constant independent of $d,\la$, then the dominant contribution to the error bound $\Rem_1$ is
$$
|\Rem_1|\les \left(\|\nabla_x^3\z(0_{d+1})\|_{H}^2+\sup_{x\in\us(R\e)}\|\nabla_x^4\z(x,0)\|_{H}\right)\frac{d^2}{\la}.
$$
\end{remark}
\begin{remark}[Rare event probabilities]Often one is interested in probabilities of rare events $D$ under a probability density $\pi(u)\propto e^{-\la z(u)}$:
$$\pi(D)=\int\ind_Dd\pi=\frac{\int_De^{-\la z(u)}\dd u}{\int_{\br^{d+1}}e^{-\la z(u)}\dd u}.$$ To approximate such probabilities, the asymptotic expansion of the numerator from this section can be combined with explicit formulas for the denominator or with an asymptotic expansion of the denominator. For example, in Section~\ref{sec:gauss} below, we consider Gaussian densities $\pi$, for which normalizing constants are available in closed form. And since the denominator is a Laplace-type integral, an asymptotic expansion is available as long as $z$ satisfies certain conditions, e.g., has a unique strict global minimizer. In particular, the work~\cite{A24} provides such an expansion in the high-dimensional regime.
\end{remark}
\begin{remark}\label{rk:regime}Since Propositions~\ref{prop:main},~\ref{main:cvx},~\ref{prop:explicitL2}, and~\ref{L2:cvx} are non-asymptotic, they can be applied to a number of different asymptotic regimes. For example, another relevant asymptotic regime, studied e.g. in~\cite{tong2023large} and~\cite{tong2021extreme}, is the case where the integrand is independent of $\la$ but the set $D=D_\la$ is increasingly extreme. Thus we are interested in integrals of the form $\int_{D_\la}e^{-\bar z(u)}\dd u$. ``Extremity" of $D_\la$ is characterized by the fact that $\la:=\|\nabla \bar z(\theta_\la)\|\gg1$, where $\theta_\la =\arg\min_{u\in D_\la}\bar z(u)\in\pa D_\la$. Such an integral can easily be brought into the form studied in this section, by defining $z_\la(u)=\la^{-1}(\bar z(u+\theta_\la)-\bar z(\theta_\la))$ (and rotating coordinates if necessary), so that 
$$
\int_{D_\la}e^{-\bar z(u)}\dd u=e^{-\bar z(\theta_\la)}\int_{D_\la-\theta_\la}e^{-\la z_\la(u)}\dd u.
$$ 
Note that the domain $D=D_\la-\theta_\la$ and the function $z=z_\la$ both depend on $\la$. Since our results are non-asymptotic, they can be applied to a wide range of $\la$-dependent $z$ and $D$ without any issue. 

There can be other ways to bring an integral $\int_{D_\la}e^{-\bar z(u)}\dd u$ into the form $\int_{D}e^{-\la z(u)}\dd u$. For example, in Section~\ref{sec:gauss} on standard Gaussian probabilities, we take advantage of the specific form of the Gaussian density to rewrite the integral as $\int_{D}e^{-\la z(u)}\dd u$ in a different way than above, for a function $z$ that does not depend on $\la$.  
 \end{remark}

\section{High-dimensional standard Gaussian probabilities}\label{sec:gauss}
Here, we apply the results of the preceding section to derive the asymptotics of rare probabilities $\mathbb P(\mathcal N(0_{d+1}, I_{d+1})\in  D_\la)$. Since the level sets of the standard Gaussian density are spheres about zero, the instanton is the point in $D_\la$ closest to the origin. This closest distance should be large for $D_\la$ to be a rare event. We take it to be $\sqrt\la$ for some $\la\gg1$. Furthermore, by the rotation invariance of the standard Gaussian, we assume without loss of generality that the instanton is the point $(0_d,\sqrt\la)$. Finally, it is natural to assume the boundary $\pa D_\la$ also scales with $\la$. We encode the appropriate scaling of the boundary and the instanton's distance from zero using a template set $D$:
\be\label{D-Dla}D_\la= \sqrt\la D+(0_d,\sqrt\la).\ee 
Formally, this setting differs from that studied in Section~\ref{sec:gen}, in which we studied integrals against increasingly concentrating (unnormalized) densities $e^{-\la z(u)}$, $\la\gg1$ over fixed sets $D$. In contrast, here we study a fixed density (the standard Gaussian density), over an increasingly extreme event $D_\la$, $\la\gg1$. However, the above structure of $D_\la$ easily allows us to bring the latter set-up into the former through a simple rescaling. 
\begin{remark}
In the results in this section, it is most natural to consider sets $D$ which do not explicitly depend on $\la$ (except possibly through $d$, which may grow large with $\la$). This implies that the scaling of $D_\la$ with $\la$ is encoded purely through the template $D_\la= \sqrt\la D+(0_d,\sqrt\la)$. However, as we will see in Example~\ref{ex:quad2}, our theory can accommodate the case where $D$ does depend on $\la$. 
\end{remark}
\begin{remark}
 Note that tail probabilities for non-standard Gaussian distributions can always be expressed as standard Gaussian tail probabilities via an affine transformation of the event of interest. 
\end{remark}
We now specify our assumptions on the set $D$.
\begin{assumption}\label{assume:D}Suppose $D\subset\br^{d+1}$ is star-shaped about $0_{d+1}\in \pa D$, with
\be\label{inst-la}0_{d+1}=\arg\min_{u}\|u-(0_d,1)\|.\ee Suppose $D$ satisfies Assumption~\ref{a12} with boundary $\psi$, and $H:=I_d+\nabla^2\psi(0_d)\succ0$. Finally, suppose Assumption~\ref{a14} is satisfied with some $\rho_0$, and by shrinking $\rho_0$ if necessary, that $\delta_2(\rho_0)\rho_0^2<2$. 
\end{assumption}Note that~\eqref{inst-la} is equivalent to $(0_d,\sqrt\la)$ being the point in $D_\la$ closest to the origin. Assumption~\ref{assume:D} is essentially equivalent to Assumption~\ref{assume1} with 
\be\label{z-gauss} z(x,t)=\|x\|^2/2 + (t+1)^2/2.\ee 
Indeed, we have omitted Assumption~\ref{a11} above, since it is immediate that $0_{d+1}$ minimizes $z$ over $D$ and that $\nabla z(0_{d+1})=(0_d,1)$. Furthermore, Assumption~\ref{a13} reduces precisely to $H=I_d+\nabla^2\psi(0_d)\succ0$. Finally, we show in Appendix~\ref{app:gauss:2} that the condition $\rho_0^2\delta_2(\rho_0)<2$ ensures $\pa_tz(u)>0$ for all $u\in\U(\rho_0)$, so that Assumption~\ref{a15} is also satisfied. 

The reason the function $z$ from~\eqref{z-gauss} arises is that $e^{-\la z}$ is proportional to the density of $\mathcal N((0_{d},-1),\la^{-1}I_{d+1})$, so that
\bs\label{gaussprob-prelim}
\mathbb P\left(\mathcal N\left(0_{d+1}, \,I_{d+1}\right)\in  D_\la\right) &=\mathbb P\left(\mathcal N((0_{d},-1),\la^{-1}I_{d+1})\in D\right)\\
&=(\la/2\pi)^{(d+1)/2}\int_{D}e^{-\la z(u)}\dd u.
\es
Since $z$ is convex, we can now apply Propositions~\ref{main:cvx} and~\ref{L2:cvx} to expand the Gaussian probability~\eqref{gaussprob-prelim}.

\begin{proposition}\label{prop:gauss:1}Suppose $D$ satisfies Assumption~\ref{assume:D} and let $D_\la$ be as in~\eqref{D-Dla}. Let $R = 24 + 2(1+2N)(\log\la)/d$. Suppose $\psi\in C^{2N+2}(\us(R\e))$ and there exist constants $\CR, C_2,\dots, C_{2N+2}$ such that
\begin{align}
 \delta_\ell(R\e)\leq C_\ell d^{(\lceil\ell/2\rceil-2)_+},\quad\ell=2,\dots,2N+2,\label{deltas:gauss}\\
d^2/\la\leq C_0, \qquad R\e \leq \min(1/3, 2/(4C_2+C_3), \rho_0),\label{cond:gauss}\end{align}
where $\rho_0$ is as in Assumption~\ref{assume:D} and the $\delta_\ell$ are as in \eqref{deldef}, with $H=I_d+\nabla^2\psi(0)$. Then
\bs\label{expansion-gauss}
\mathbb P\left(\mathcal N\left(0_{d+1}, \,I_{d+1}\right)\in  D_\la\right) = \frac{e^{-\la/2}\la^{-1/2}}{\sqrt{2\pi\det H}}\left(1+\sum_{m=1}^{N-1}a_m\left(\frac{d^2}{\la}\right)^{m} + \Rem_N\right),\es where 
\bs
\max_{m=1,\dots,N-1}|a_m|&\leq C(C_2,\dots,C_{2N+2}, N),\\
|\Rem_N|&\leq C(C_0, C_2,\dots,C_{2N+2}, N)\bigg(\frac{d^2}{\la}\bigg)^{N}.\es Here, $a_m$ is as in~\eqref{amdef}, with the $\g_k$ as in~\eqref{gkdef1}, where $\g(x)\equiv1$ and $\z(x,0)=\|x\|^2/2 +(\psi(x)+1)^2/2$. In particular, 
\bs\label{a1gengauss}
d^2a_1 = & - 1  -\frac12\mathrm{Tr}(\nabla^2\psi_H(0_d)) +\frac{1}{12}\|\nabla^3\psi_H(0_d)\|^2_{F} + \frac18\|\langle\nabla^3\psi_H(0_d),I_d\rangle\|^2 \\
&-\frac18\langle\nabla^4\psi_H(0_d), I_d\otimes I_d\rangle-\frac18 \mathrm{Tr}^2(\nabla^2\psi_H(0_d)) -\frac14\mathrm{Tr}\left(\nabla^2\psi_H(0_d)^2\right),
\es where $\psi_H(x)=\psi(H^{-1/2}x)$.
\end{proposition} 
The next proposition is the Gaussian analogue of Proposition~\ref{L2:cvx}.
\begin{proposition}\label{prop:gauss:3}Suppose $D$ satisfies Assumption~\ref{assume:D} and let $D_\la$ be as in~\eqref{D-Dla}. Let $R=24+6(\log\la)/d$, fix an arbitrary constant $c_{\psi}$, and suppose
\begin{align}\label{rho1-gauss3}
R\e\max(1,\delta_2(R\e),\delta_3(R\e))&\leq \min\left(1/3,\rho_0\right),\\
\label{eq:assume:gauss2}
(\delta_2^2+\delta_3^2+\delta_4( R\e))\frac {d^2}{\la} &\leq c_{\psi},
\end{align}
where $\rho_0$ is as in Assumption~\ref{assume:D}. Then
\bs\label{expansion-gauss2}
\mathbb P\left(\mathcal N\left(0_{d+1}, \,I_{d+1}\right)\in  D_\la\right) &= \frac{e^{-\la/2}\left(1+\Rem_1\right)}{\sqrt{2\pi\la\det H}},\\
|\Rem_1| &\leq C(c_\psi,\tfrac{\log\la}{d}) \left(\left(\delta_2^2+\delta_3^2 +\delta_4( R\e)\right)\frac{d^2}{\la} + \frac1\la\right).
\es \end{proposition}
\begin{remark}\label{gauss-psi-d-la}
In contrast to Proposition~\ref{prop:gauss:1}, here we do not impose explicit constraints on the magnitude of the derivative norms of $\psi$. In particular, the conditions~\eqref{deltas:gauss} may be violated, as long as $\la$ is large enough to ensure~\eqref{rho1-gauss3} and~\eqref{eq:assume:gauss2} are satisfied. Conversely, if the derivative norms of $\psi$ are small, then a bound on $d^2/\la$ (as in~\eqref{cond:gauss}) may not be needed. This is exactly parallel to the observations from Remark~\ref{rk:brd}. See the third case in Example~\ref{ex:quad2} for more on the setting where the derivative norms of $\psi$ are small.
\end{remark}
\begin{example}[Pure quartic boundary]\label{ex:quart}
We apply Proposition~\ref{prop:gauss:1} to the case 
\be\label{psi-S}\psi(x)=\frac{1}{4!}\langle S, x^{\otimes4}\rangle\ee for a symmetric tensor $S$. Then $\nabla^2\psi(0_d)$ is the zero matrix, so $H=I_d$. We show in Appendix~\ref{app:gauss} that if Assumptions~\ref{a12},~\ref{a14} are satisfied for some $\rho_0\leq\frac13(1\vee\|S\|)^{-1}$, and 
\be\label{Re-quart}
R\e =24\sqrt{d/\la} + 2(1+2N)\frac{\log\la}{\sqrt{d\la}}\leq \rho_0, \qquad \frac{d^2}{\la}\leq C_0,
\ee
then
\be\label{P-quart}\mathbb P\left(\mathcal N\left(0_{d+1}, \,I_{d+1}\right)\in  D_\la\right) = \frac{e^{-\la/2}}{\sqrt{2\pi\la}}\left(1-(1+\tfrac18\langle S, I_d\otimes I_d\rangle)\la^{-1}+\sum_{m=2}^{N-1}a_m\left(\frac{d^2}{\la}\right)^{m} + \Rem_N\right)\ee for all $N\geq2$,
where 
\be\label{amquart}
\max_{m=1,\dots,N-1}|a_m|\leq C(N,\|S\|),\qquad|\Rem_N|\leq C(N, C_0, \|S\|)\bigg(\frac{d^2}{\la}\bigg)^{N}.\ee
It is interesting to note that the leading order term $\frac{e^{-\la/2}}{\sqrt{2\pi\la}}$ does not contain $d$ at all. It may therefore seem, misleadingly, that a bound on $d$ is not required for the approximation $\mathbb P(\mathcal N(0_{d+1}, \,I_{d+1})\in  D_\la) \approx \frac{e^{-\la/2}}{\sqrt{2\pi\la}}$ to be asymptotically accurate. But in fact, the first remainder term $(1+\langle S, I_d\otimes I_d\rangle/8)/\la$ can grow with $d$. For example, we show in Appendix~\ref{app:a2:gauss}, that in the case $\psi(x)=\|x\|^4/24$, which can be written as~\eqref{psi-S} for $S=\mathrm{Sym}(I_d\otimes I_d)$, it holds $\langle S, I_d\otimes I_d\rangle = \frac13d^2 + \frac23d$. Thus
\be\label{quart-tight}
\mathbb P\left(\mathcal N\left(0_{d+1}, \,I_{d+1}\right)\in  D_\la\right) = \frac{e^{-\la/2}}{\sqrt{2\pi\la}}\left(1 -\frac{d^2+2d+24}{24\la} + \mathcal O((d^2/\la)^2)\right).
\ee Therefore, the $d$-independent leading order term $\frac{e^{-\la/2}}{\sqrt{2\pi\la}}$ is truly leading order \emph{only if} $d^2\ll \la$.
\end{example}

\begin{example}[Quadratic boundaries]\label{ex:quad2}
 In this example, we apply Proposition~\ref{prop:gauss:3} to the case $\psi(x) = \frac12x^\top Bx$. Then $H=I_d+B$. We have $\delta_\ell(\cdot)\equiv0$ for all $\ell\geq3$, and $\delta_2(\cdot)\equiv\delta_2=\|B\|_{I_d+B}$. We suppose Assumption~\ref{assume:D} is satisfied, and to ensure $\delta_2(\rho_0)\rho_0^2<2$, we suppose $\rho_0\leq 1/\sqrt{\delta_2}$. We also set $c_\psi=1$ in~\eqref{eq:assume:gauss2}. Proposition~\ref{prop:gauss:3} now gives that if
\be\label{gauss-quad}
R\e=24\sqrt{d/\la} + 6\frac{\log\la}{\sqrt{d\la}} \leq \frac{\frac13\wedge\rho_0}{1\vee\delta_2},\qquad (\delta_2d)^2/\la \leq 1,
\ee then
\bs\label{gauss-quad-rem}
\mathbb P\left(\mathcal N\left(0_{d+1}, \,I_{d+1}\right)\in  D_\la\right) &=  \frac{e^{-\la/2}\left(1+\Rem_1\right)}{\sqrt{2\pi\la\det(I_d+B)}},\qquad |\Rem_1|\leq C((\log\la)/d)\frac{((\delta_2d)\vee1)^2}{\la}.
\es
We now consider three particular quadratic functions $\psi$, and in each case, we assume the boundary parameterization of $D$ holds globally. In other words, $D=\{(x,t)\in\br^{d+1}\,:\,t\geq\frac12x^\top Bx\}$. Thus Assumption~\ref{a14} is satisfied for $\rho_0 =1/\sqrt{\delta_2}$ (which guarantees $\delta_2(\rho_0)\rho_0^2<2$). The set $D_\la$ is given by $D_\la =\{(x,t)\in\br^{d+1}\,:t\geq \sqrt\la + \frac{1}{2\sqrt\la}x^\top Bx\}$, so that
$$\mathbb P\left(\mathcal N\left(0_{d+1}, \,I_{d+1}\right)\in  D_\la\right)= \mathbb P\left(T\geq  \sqrt\la + \frac{1}{2\sqrt\la}X^\top BX\right),$$ where $X\sim\mathcal N(0_d, I_d)$ and $T\sim\mathcal N(0,1)$ is independent of $X$.

First, suppose $B = \sqrt\la I_d$, i.e. $\psi(x)=\sqrt\la\|x\|^2/2$, in which case the set $D$ actually depends on $\la$ itself. This can be easily accommodated by the nonasymptotic Proposition~\ref{prop:gauss:3}. Then $\delta_2=\|\sqrt\la I_d\|_{(\sqrt\la+1)I_d} = \sqrt\la/(1+\sqrt\la) \asymp 1$, so that~\eqref{gauss-quad} is satisfied if $d^2/\la$ is a sufficiently small absolute constant. From~\eqref{gauss-quad-rem}, we obtain
\be\label{rem1-1}
\mathbb P\left(T\geq\sqrt\la + \frac{\|X\|^2}{2}\right) =  \frac{e^{-\la/2}\left(1+\Rem_1\right)}{\sqrt{2\pi\la(1+\sqrt\la)^d}},\qquad |\Rem_1|\leq C((\log\la)/d)\frac{d^2}{\la}.
\ee 
Next, consider $B=I_d$. We then have $\delta_2=\| I_d\|_{2I_d} =1/2$. Thus~\eqref{gauss-quad} is again satisfied if $d^2/\la$ is a sufficiently small absolute constant, and~\eqref{gauss-quad-rem} reduces to
\be\label{rem1-2}
\mathbb P\left(T\geq\sqrt\la + \frac{\|X\|^2}{2\sqrt\la}\right) =  \frac{e^{-\la/2}\left(1+\Rem_1\right)}{\sqrt{2\pi\la\cdot 2^d}},\qquad |\Rem_1|\leq C((\log\la)/d)\frac{d^2}{\la}.
\ee 
Note that the leading order term is significantly larger in~\eqref{rem1-2} than in~\eqref{rem1-1}, but in both cases, the leading order term dominates the remainder under the same condition $d^2\ll \la$ (provided $(\log\la)/d\les1$). 

Finally, consider $B=d^{-1}I_d$. We then have $\delta_2=\|d^{-1}I_d\|_{(1+d^{-1})I_d}\asymp d^{-1}$. Then~\eqref{gauss-quad} is satisfied if $d/\la$ is less than a sufficiently small absolute constant, and~\eqref{gauss-quad-rem} reduces to
\be\label{rem1-3}
\mathbb P\left(T\geq\sqrt\la + \frac{\|X\|^2}{2d\sqrt\la}\right) =  \frac{e^{-\la/2}\left(1+\Rem_1\right)}{\sqrt{2\pi\la(1+1/d)^d}},\qquad |\Rem_1|\leq C((\log\la)/d)\frac{1}{\la}.
\ee 
The leading order term is larger here than in~\eqref{rem1-1} and~\eqref{rem1-2}, and is only weakly dependent on $d$. Furthermore, if $(\log\la)/d\lesssim C$, then we obtain $|\Rem_1|\les 1/\la$ which does not depend on $d$ at all, unlike in the other two cases. This is natural, since the probability in ~\eqref{rem1-3} is closest to the extreme case of a flat boundary, $\mathbb P(T\geq\sqrt\la)$. The accuracy of an asymptotic approximation to this one-dimensional integral should not depend on $d$. Note that we derived ~\eqref{rem1-3} under the assumption that $d/\la$ is sufficiently small. Since the remainder bound itself is small even if $d/\la$ is not small, we conjecture that the restriction on $d/\la$ is an artifact of the proof and could be dropped using a different proof technique. 
\end{example}
We now show that in~\eqref{rem1-2}, we actually have $\Rem_1\asymp d^2/\la$. This shows that $d^2/\la$ is a sharp upper bound.
\begin{proposition}\label{prop:lb}
Let $\Rem_1$ be as in~\eqref{rem1-2}. There exist absolute constants $c\in(0,1)$ and $0<C<C'$ such that if
\be\label{la23}
d^2/\la \leq c, \qquad 24\sqrt{d/\la}+10\frac{\log\la}{\sqrt{d\la}}\leq 1/3,
\ee then
\be
C\frac{d^2}{\la} \leq |\Rem_1| \leq C'\frac{d^2}{\la}.
\ee
\end{proposition}
See Appendix~\ref{app:gauss} for the proof, which is a straightforward application of Proposition~\ref{prop:gauss:1} to order $N=1$ and $N=2$. This gives that $\Rem_1 = a_1d^2/\la + \Rem_2$, where $|\Rem_2|\les (d^2/\la)^2$, $|\Rem_1|\les d^2/\la$. Furthermore, using~\eqref{a1gengauss} we can show that $a_1$  is upper- and lower-bounded by absolute constants to conclude that $|\Rem_1|\asymp d^2/\la$. 
\begin{remark}Note that~\eqref{quart-tight} provides another example, with a quartic boundary, in which $|\Rem_1|\asymp d^2/\la$.
\end{remark}

\section{Rare event sampling}\label{sec:sample}
In this section, we derive an asymptotic approximation $\hat\pi$ to a probability density $\pi\vert_D$ given by the restriction of a density $\pi\propto e^{-\la z}$ on $\br^{d+1}$ to a rare event $D$:
\be
\pi\vert_D(u) =\frac{ \ind\{u\in D\}\exp(-\la z(u))}{\int_D\exp(-\la z(u'))\dd u'}.
\ee 
We consider the same set-up as in Section~\ref{sec:gen}, as described in Assumption~\ref{assume1}. In particular, recall from this assumption that $H=\nabla_x^2z(0_{d+1})+\nabla^2\psi(0_d)$, where $(x,\psi(x))$ parameterizes the boundary $\pa D$ near the instanton $0_{d+1}=\arg\min_{u\in D}z(u)$. Consider the following product measure on $\br^{d+1}$ and coordinate transformation $\hat\Theta:\br^{d+1}\to\br^{d+1}$:
\bs\label{pi-H}
\pi_H(x,y)&=\frac{\la\sqrt{\det H}}{(2\pi/\la)^{d/2}} \exp\left(-\frac\la2x^\top Hx - \la y\right)\ind(y\geq0),\\
\hat\Theta(x,y) &= \left(x,y+\tfrac12x^\top\nabla^2\psi(0_d)x\right).
\es
The following density is our approximation to $\pi\vert_D$:
\be\label{pi-hat}\hat\pi:=\hat\Theta\#\pi_H,\ee the pushforward of $\pi_H$ under the coordinate transformation $\hat\Theta$. Note that we can sample from $\hat\pi$ as follows:
\begin{enumerate}
\item Draw $X \sim \mathcal N(0_d, (\la H)^{-1})$.
\item Draw $Y\sim\mathrm{Exp}(\la)$ independently of $X$.
\item Return $\hat\Theta(X,Y)$.
\end{enumerate} 
Below, recall that $\omega_{k,\ell}$ is shorthand for $\omega_{k,\ell}(0,0)$, and similarly for $\delta_k$.
\begin{proposition}\label{prop:TV}
Assume the conditions of either Proposition~\ref{prop:explicitL2} or Proposition~\ref{L2:cvx}, and suppose $\psi\in C^4(\us(R\e))$. If $R\e\delta_3(R\e)\leq3$ and $(\delta_3^2+\delta_4(R\e))d^2/\la\leq c_{\delta}$, then
\bs\label{eq:TV}
\mathrm{TV}(\pi\vert_D, \hat\pi)\les_{R,\CL,c_\delta} &\,(\omega_{3,0}+\delta_3)\frac{d}{\sqrt\la}+\frac{\omega_{1,1}}{\sqrt\la}+\left(\omega_{1,1}^2+\omega_{2,1}(R\e,0)\right)\frac{d}{\la}\\
& + \omega_{0,2}(R\e,(R\e)^2)\frac1\la+\left(\omega_{4,0}(R\e,0)+\delta_4(R\e)\right)\frac{d^2}{\la} + \la^{-M}.\es
\end{proposition}
\begin{remark}[Region where $\hat\pi$ concentrates]\label{rk:conc}
A desirable feature of samples from a density approximating $\pi\vert_D$ is that they lie in the rare event set $D$ with high probability. Samples from $\hat\pi$ satisfy this property. Specifically, we show in the proof of Lemma~\ref{lma:tailAB} (see Remark~\ref{rk:bB}) that 
$$
\hat\pi(D) \geq 1-2e^{-(R-1)^2d/2},
$$  which is close to 1 whenever $R^2d\gg1$.  
\end{remark}
\begin{remark}[Uses for $\hat\pi$]
Sampling from $\hat\pi$ is much cheaper than sampling from $\pi\vert_D$, and can be used for the following purposes.
\begin{enumerate}
\item Exploration. It may be of interest to observe configurations (points $u$) likely to be generated under $\pi\vert_D$, beyond the instanton itself.  Samples of $\hat\pi$ can be used as proxies for such configurations.
\item Observable expectations. We may be interested in expectations of a functional $f$ over the rare event, e.g. damage caused by a natural disaster. Such an expectation is given by $\int fd\pi\vert_D= \int_D f(u)e^{-\la z(u)}\dd u/ \int_D e^{-\la z(u)}\dd u$. If $f$ is sufficiently smooth in a neighborhood of the instanton, then the asymptotic approximations from Section~\ref{sec:gen} can be used for the numerator and denominator. However, if $f$ is discontinuous near the instanton, then the theory of Section~\ref{sec:gen} does not apply. An alternative approach to estimate the expectation is based on sampling $\hat\pi$, and it does not require $f$ to be continuous. Namely, we have $\int fd\pi\vert_D\approx \int fd\hat\pi \approx \frac1N\sum_{i=1}^Nf(X_i)$, with $X_i\stackrel{\mathrm{i.i.d.}}{\sim}\hat\pi$. Proposition~\ref{prop:TV} can be used to control the error in the first approximation if $f$ is bounded, in which case $|\int fd\pi\vert_D- \int fd\hat\pi|\leq\|f\|_\infty\mathrm{TV}(\pi\vert_D, \hat\pi).$ The second approximation yields an unbiased estimator of $\int fd\hat\pi$, with variance scaling with $N$ as $1/N$. 
\item Rare-event probabilities. Samples from $\hat\pi$ can be used to construct an estimator of $\pi(D)$:
\be\label{imp}\pi(D)\approx p:=\sum_{i=1}^N\ind(X_i\in D)w(X_i),\qquad X_i\stackrel{\mathrm{i.i.d.}}{\sim}\hat\pi,\ee where common choices of the weights are $w(X_i)=\frac1N\frac{\pi(X_i)}{\hat\pi(X_i)}$ and $w(X_i)=\frac{(\pi/\hat\pi)(X_i)}{\sum_{j:\,X_j\in D}(\pi/\hat\pi)(X_j)}$~\cite{owen2013monte, agapiou2017importance}. The estimator~\eqref{imp} is an alternative to the asymptotic approximations from Section~\ref{sec:gen}, and may be more computationally convenient in certain cases. We leave the theoretical and numerical analysis of this estimator to future work. We note that $\hat\pi$ is a promising choice of proposal distribution compared to other proposals in the literature due to $\hat\pi$ being a canonical approximation to $\pi$.

As discussed in the introduction, existing literature does not exploit asymptotic properties of $\pi\vert_D$ in their construction of the proposal distributions to the extent that we do here. 

For example, the proposal distribution in~\cite{tong2023large} is chosen from the family of Gaussians over $\br^{d+1}$. But the Gaussian family is not well-suited for modeling a distribution which concentrates at a point on the boundary of a domain and which, at the same time, assigns zero mass below the boundary. 
\end{enumerate}
\end{remark}
\begin{remark}[Pointwise approximation to $e^{-\la z}$]Note that $\pi\propto\ind_De^{-\la z}$ depends on the set $D$ only via the normalizing constant, and via the fact that outside $D$, the density is zero. Otherwise, $\pi$ depends solely on the function $z$. It is reasonable to expect that any good approximation $\hat\pi\approx\pi$ shares this property: it should depend on $D$ only through the normalizing constant and through the set where $\hat\pi$ is nonzero, but it should otherwise only depend on $z$. In other words, we essentially expect the density of $\hat\pi$ to be a pointwise approximation to the function $e^{-\la z}$, up to normalizing constant. Our $\hat\pi$ from~\eqref{pi-hat} indeed has this property, since for $(x,t)$ in its support, we have
\bs\label{hatThetadens}
\hat\pi(x,t) &\propto \exp\left(-\la x^\top Hx/2 -\la (t-x^\top\nabla^2\psi(0_d)x/2)\right) \\
&= \exp\left(-\la \left[\frac12x^\top \nabla_x^2z(0_{d+1})x +\pa_tz(0_{d+1})t\right]\right),
\es recalling that $H=\nabla_x^2z(0_{d+1})+\nabla^2\psi(0_d)$ and $\pa_tz(0_{d+1})=1$. Thus $\hat\pi$ \emph{simply replaces $z(u)$ in the exponent by its Taylor approximation}. In this way, $\hat\pi$ can be viewed as a canonical approximation.

Note that the Taylor approximation only keeps one of the three quadratic terms, namely, $x^\top \nabla_x^2z(0_{d+1})x/2$, and drops the other two: the cross term $t\nabla_x\pa_tz(0_{d+1})^\top x$ and the pure $t$ quadratic $\pa_t^2z(0_{d+1})t^2/2$. To see why this is justified, recall from Remark~\ref{rk:conc} that $\hat\pi$ concentrates in the set where $\|x\|_H\leq R\e$ and $t-\psi(x)\in [0,(R\e)^2]$. Note that such $t$ are of order $(R\e)^2$ provided $\|\nabla^2\psi(0)\|_H$ is bounded by a constant, since $\psi(x)\approx\frac12x^\top\nabla^2\psi(0_d)x$. Thus, roughly speaking, $\hat\pi$ concentrates in a region where $\|x\|_H\leq R\e$ and $|t|\leq (R\e)^2$. But in this region, the quadratic cross term scales as $(R\e)^3$ and the pure $t$ quadratic scales as $(R\e)^4$. They are therefore negligible compared to the first quadratic term $x^\top \nabla_x^2z(0_{d+1})x/2$ and the linear term $\pa_tz(0_{d+1})t$, both of which are of order $(R\e)^2$. 
\end{remark}
\begin{remark}[Case of general instanton $u^*$ and direction $\nabla z(u^*)$]In Appendix~\ref{app:gen:sample}, we give the procedure for sampling from a Laplace-type density $\pi(u)\propto e^{-\la z(u)}$ conditional on a rare event $D$, in the more general case where the instanton $u^*=\arg\min_{u\in D}z(u)$ is not necessarily the origin and $\nabla z(u^*)$ is not necessarily parallel to the $t$-axis. 
\end{remark}

\begin{remark}A result similar to Proposition~\ref{prop:TV} was obtained in~\cite[Theorem 6]{lapinski2019multivariate} in the case of fixed $d$ and a flat boundary.
\end{remark}
We now consider the Gaussian setting from Section~\ref{sec:gauss}. Thus we are interested in approximately sampling from the restriction of the standard Gaussian density to a set $D_\la=\sqrt\la D + (0_d,\sqrt\la)$, where $\pa D$ is locally parameterized by $\psi$. 
\begin{corollary}[Gaussian case]\label{corr:gTV}Suppose the assumptions of Proposition~\ref{prop:gauss:3} are satisfied, and let $\pi_\la$ be the restriction of $\mathcal N(0_{d+1}, I_{d+1})$ to $D_\la$. Then
\bs\label{eq:gauss:TV}
\mathrm{TV}(\pi_\la,\hat\pi_\la)\leq C(c_\delta,(\log\la)/d)\left( \delta_3\frac{d}{\sqrt\la} +(\delta_2^2+\delta_4(R\e))\frac{d^2}{\la}+\frac1\la\right),
\es Here, $\hat\pi_\la =\hat\Theta_\la\#\pi_H$, where $\hat\Theta_\la(x,y)=\sqrt\la\hat\Theta(x,y)+(0_d,\sqrt\la)$ and $\pi_H,\hat\Theta$ are as in~\eqref{pi-H}, with $H,\psi$ as in Assumption~\ref{assume:D}. \end{corollary}
\begin{remark}
Just as the set-up of Section~\ref{sec:gen} differs from the set-up of Section~\ref{sec:gauss}, the set-up of Proposition~\ref{prop:TV} differs from that of Corollary~\ref{corr:gTV}: the latter involves a fixed density restricted to an increasingly extreme event. Correspondingly, the scaling of $\hat\pi_\la$ is different from the scaling of $\hat\pi$. Specifically, to sample from $\hat\pi_\la$, we draw $(X,Y)$ as described in (1),(2) above Proposition~\ref{prop:TV}, but then return $\hat\Theta_\la(X,Y)$ rather than $\hat\Theta(X,Y)$.
\end{remark}

\section{Proof outlines}\label{sec:proofs}
\subsection{Outline for proof of Propositions~\ref{prop:main},~\ref{main:cvx},~\ref{prop:explicitL2},~\ref{L2:cvx}}\label{app:outline:int}
Recall $\e=\sqrt{d/\la}$, and the definition~\eqref{UH} of $\us$ and $\U$. We break up the domain $D$ into three overlapping regions $D_1,D_2,D_3$, where
\bs\label{D123}
D_1&=\Theta(\us(R\e)\times[0, (R\e)^2]),\\
D_2&=\Theta(\us(R\e)\times[(R\e)^2, R\e]),\\
D_3 &=D\setminus \U(R\e/2),
\es and $R$ will be chosen throughout the course of the proof. See Figure~\ref{fig:local}. In the next four lemmas, we implicitly assume Assumption~\ref{assume1} is satisfied, and that
\be R\e\leq\rho_0\wedge1.\ee
Recall from Assumption~\ref{a14} that $\us(\rho_0)\subseteq\Omega$, which is the region where $\psi$ exists. Since $R\e\leq\rho_0$, we conclude $\psi(x)$ is well-defined for $x\in\us(R\e)$. This is needed for $D_1,D_2$ above to be well-defined.  

\begin{figure}[h]
{\centerline{
{\epsfig{file={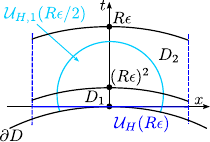}, width=6cm}}
}}
\caption{Illustration of various sets used in the proofs in Section~\ref{sec:proofs}. The set $D_1$ is bounded by $\pa D$ and the surface parallel to $\pa D$ passing through the point $(0_d,(R\e)^2)$.  The set $D_2$ is bounded by the surfaces parallel to $\pa D$ and passing through the points $(0_d,(R\e)^2)$ and $(0_d,R\e)$, respectively. The set $D_3$ is the part of $D$ located outside the cyan circle, i.e. $D_3=D\setminus \U(R\e/2)$. The set $D_3$ intersects $D_1$ and $D_2$.}
\label{fig:local}
\end{figure}

 \begin{lemma}\label{lma:D123}Suppose $R\e\delta_2(R\e/2)\leq 4$. Then $D= D_1\cup D_2\cup D_3$. 
\end{lemma}
\begin{proof}
We first show $D_1\cup D_2=\Theta(\us(R\e)\times[0, R\e])\subseteq D$. Indeed, since $R\e\leq \rho_0$, we have $D_1\cup D_2\subseteq\Theta(\us(\rho_0)\times[0,\rho_0])\subseteq D$, by Assumption~\ref{a14}. It remains to show $\U(R\e/2)\subseteq \Theta(\us(R\e)\times[0, R\e])$, i.e. that if $(x,t)\in D$ and $\|(x,t)\|_{H,1}\leq R\e/2$, then $\|x\|_H\leq R\e$ and $0\leq t-\psi(x)\leq R\e $. That $\|x\|_H\leq R\e $ is obvious by the definition~\eqref{uH1} of $\|\cdot\|_{H,1}$. That $t-\psi(x)\geq 0$ holds by Assumption~\ref{a14}. It remains to show $t-\psi(x)\leq R\e $. But $|t|\leq R\e/2$ and $x\in\us(R\e/2)$, so $$t-\psi(x) \leq R\e/2 + \sup_{x\in\us(R\e/2)}|\psi(x)| \leq \frac12R\e  +\frac12\delta_2(R\e/2)(R\e/2)^2\leq R\e.$$ Here, we used a Taylor expansion and the fact that $R\e\delta_2(R\e/2)\leq 4$.
\end{proof}
Define 
\bs\label{Z-rho-def}
\mathcal Z = \frac{(2\pi/\la)^{d/2}}{e^{\la z(0)}\sqrt{\det H}},\qquad\mu = \frac{1}{\mathcal Z}\int_{\us(R\e)}e^{-\la z(x,\psi(x))}\dd x=\frac{1}{\mathcal Z}\int_{\us(R\e)}e^{-\la\z(x,0)}\dd x.
\es
We now state three lemmas which handle the three regions.

\begin{lemma}[Decomposition of integral over $D_1$]\label{lma:D1}
Let $N\ge 1$ and suppose $\pa_tz(u)\geq1/2$ for all $u\in D_1$. Then
\be\label{D1intdecomp}
\frac{1}{\mathcal Z}\int_{D_1}g(u)e^{-\la z(u)}\dd u = \sum_{k=1}^{N}\la^{-k}\frac{1}{\mathcal Z}\int_{\us(R\e )}\g_k(x)e^{-\la \z(x,0)}\dd x +\Rem_{N+1}^{\Wat},
\ee
where
\be
|\Rem_{N+1}^{\Wat}|\leq 2\mu\left(\|D_{\z}^{N}\g\|_\infty\la^{-N-1}+e^{-R^2d/2}\max_{1\leq k\leq N}\|D_{\z}^{k-1}\g\|_\infty\la^{-k}\right). 
\ee Here, $\|f\|_\infty = \sup_{(x,y)\in \us(R\e)\times[0, (R\e)^2]}|f(x,y)|$.
\end{lemma}
\begin{proof}Using the $\z,\g$ parameterization, we have
\be\label{Iloc}
\frac{1}{\mathcal Z}\int_{D_1}g(u)e^{-\la z(u)}\dd u = \frac{1}{\mathcal Z}\int_{\us(R\e )}\int_{0}^{(R\e)^2} \g(x,y)e^{-\la \z(x,y)}\dd y\dd x.
\ee Here, note that the Jacobian of the transformation $\Theta(x,y)=(x, \psi(x)+y)$ has determinant 1. We now apply Watson's lemma, formulated in Lemma~\ref{lma:gen-wat}. The lemma applies because $\pa_y\z(x,y)=\pa_tz(x,\psi(x)+y)\geq1/2$ for all $(x,y)\in\us(R\e)\times[0,(R\e)^2]$. We obtain
\be\label{lma:1:ir}
\int_0^{(R\e)^2} \g(x,y)e^{-\la \z(x,y)}\dd y =e^{-\la \z(x,0)}\left(\sum_{k=1}^{N}\frac{(D_{\z}^{k-1}\g)(x,0)}{\pa_y\z(x,0)}\la^{-k}+ \Rem_{N+1}^{\Wat}(x)\right),\ee where
\bs\label{wat rem L}
\sup_{x\in\us(R\e)}|\Rem_{N+1}^{\Wat}(x)|\leq 2\|D_{\z}^{N}\g\|_\infty\la^{-N-1}+2e^{-R^2d/2}\max_{1\leq k\leq N}\|D_{\z}^{k-1}\g\|_\infty\la^{-k}.
\es Finally, we integrate~\eqref{lma:1:ir} over $x\in\us(R\e)$ and normalize by $\mathcal Z$. We use that $$\frac{1}{\mathcal Z}\int_{\us(R\e)}|\Rem_{N+1}^{\Wat}(x)|e^{-\la\z(x,0)}\dd x\leq \mu\sup_{x\in\us(R\e)}|\Rem_{N+1}^{\Wat}(x)|$$ and the bound~\eqref{wat rem L} to conclude.
\end{proof}
\begin{lemma}[Bound on integral over $D_2$]\label{lma:D2}
Suppose $\pa_tz(u)\geq1/2$ for all $u\in D_2$. Then
\be
\frac{1}{\mathcal Z}\int_{D_2}e^{-\la z(u)}\dd u \leq \frac{2\mu}{ \la }e^{-Rd/2}.
\ee
\end{lemma}
\begin{proof}
We have $z(x,t)\geq z(x,\psi(x)) + \frac12(t-\psi(x))$ for all $(x,t)\in D_2$, and therefore
\bs
\frac{1}{\mathcal Z}\int_{D_2}e^{-\la z(u)}\dd u& \leq \frac{1}{\mathcal Z}\int_{\us(R\e)}e^{-\la z(x,\psi(x))}\int_{\psi(x)+ (R\e)^2}^\infty e^{-\la(t-\psi(x))/2}\dd t\dd x\\
&= \mu\int_{(R\e)^2}^\infty e^{-\la y/2}\dd y = \frac{2\mu}{\la}e^{-R^2d/2},
\es
 as desired. Here, recall the definition of $\mu$ from~\eqref{Z-rho-def}.
\end{proof} 
\begin{lemma}[Bound on integral over $D_3$]\label{lma:D3}
Suppose $z(u)-z(0_{d+1})\geq \e\|u\|_{H,1}$ for all $u\in D_3$. Then 
\be\label{D3bd}
\frac{1}{\mathcal Z}\int_{D_3}e^{-\la z(u)}\dd u \les \frac{d}{\sqrt\la}e^{d-Rd/4}.
\ee 
\end{lemma} See Appendix~\ref{app:secgenproofs} for the proof. 

Lemmas~\ref{lma:D1}-\ref{lma:D3} demonstrate that the integral over $D_1$ provides the dominant contribution to the Laplace integral in \eqref{intD}, while the contributions from $D_2$ and $D_3$ are exponentially small. The contribution from $D_2$ is negligible due to the rapid exponential decay in the $y$ direction near the instanton. The contribution from $D_3$ is suppressed by at least linear growth of $z(u)$ with $\|u\|_{H,1}$ as $u$ moves away from the instanton $0_{d+1}$.

We now use the above lemmas to prove a preliminary decomposition and remainder bound. To that end, for each $k=1,\dots,N$, define $\Rem_{N+1-k}^\Lap(\g_k)$ implicitly through the following equation.
\be\label{lap-gz-L}
\frac{1}{\mathcal Z}\int_{\us(R\e)}\g_k(x)e^{-\la \z(x,0)}\dd x = \sum_{\ell=0}^{N-k}\nu_{\ell}(\g_k)\la^{-\ell}+\Rem_{N+1-k}^{\mathrm{Lap}}(\g_k).
\ee 
The coefficients $\nu_\ell(\g_k)$ are the standard coefficients of the ($\mathcal Z$-normalized) Laplace integral on the left in \eqref{lap-gz-L}.
See Appendix~\ref{app:lap} for more details.
\begin{corollary}\label{corr:genprelim}Suppose $|g(x)|\leq 1$ for all $x\in D$, and Assumption~\ref{assume1} holds. Let $N\geq1,M\geq1$, and
\be\label{RdM}
R \geq  4\left(1+\frac{\log d}{d}+(M+1/2)\frac{\log\la}{d}\right).\ee 
Suppose $R\e \leq\rho_0\wedge1$, $R\e\delta_2(R\e/2)\leq 4$,  $\pa_tz(u)\geq1/2$ for all $u\in D_1\cup D_2=\Theta(\us(R\e)\times[0,R\e])$, and $z(u)-z(0_{d+1})\geq \e\|u\|_{H,1}$ for all $u\in D_3=D\setminus \U(R\e/2)$. Then
\bs\label{eq:genprelim}
\frac{1}{\mathcal Z}\int_{D} g(u)e^{-\la z(u)}\dd u &=\frac{1}{\la}\left(\sum_{m=1}^{N}c_m\la^{-(m-1)} +\Rem_{N}\right),\\
 c_m &= \sum_{k=1}^m\nu_{m-k}(\g_k),\quad m=1,\dots,L-1.
\es The remainder $\Rem_N$ is bounded as follows, with $\|f\|_\infty=\sup_{x\in \us(R\e),y\in[0 ,(R\e)^2]}|f(x,y)|$:
\bs\label{cm-remL-tail}
|\Rem_{N}| \les\; &\sum_{k=1}^{N}|\Rem_{N+1-k}^\Lap(\g_k)|\la^{-(k-1)} + \mu\|D_{\z}^N\g\|_\infty\la^{-N}\\
&+(1\vee\mu)\big(1+\max_{1\leq k\leq N}\|D_{\z}^{k-1}\g\|_\infty\la^{-(k-1/2)}\big)\la^{-M}.
\es Here, $\Rem_{N+1-k}^\Lap(\g_k)$ is as in~\eqref{lap-gz-L}.
\end{corollary} See Appendix~\ref{app:secgenproofs} for the proof. To prove Propositions~\ref{prop:main},~\ref{main:cvx},~\ref{prop:explicitL2}, and~\ref{L2:cvx}, it remains to bound terms of the form $\Rem_k^\Lap$ and $\|D_{\z}^k\g\|_\infty$ and $\mu$. But first, we verify the conditions of the above corollary under the conditions of the four propositions, and show $\mu$ is bounded by a constant.
\begin{lemma}\label{lma:mu}The conditions of Corollary~\ref{corr:genprelim} are satisfied under the conditions of Propositions~\ref{prop:main}, \ref{main:cvx},~\ref{prop:explicitL2}, and~\ref{L2:cvx}, respectively, with $M=N$ in the first two of these propositions. Under the assumptions of Proposition~\ref{prop:main} and~\ref{main:cvx}, we have $\mu\leq C(\CR, C_{\z,3},C_{\z,4})$. Under the assumptions of Proposition~\ref{prop:explicitL2} and~\ref{L2:cvx}, we have $\mu\leq C(R, \CL)$. \end{lemma} 
See Appendix~\ref{app:secgenproofs} for the proof. Here, we briefly discuss the proof of the first statement. Most of the conditions of Corollary~\ref{corr:genprelim} are automatically satisfied as long as $R\e$ is small enough. For example, since $\pa_tz(0_{d+1})=1$, the assumption that $\pa_tz(u)\geq1/2$ in $\Theta (\us(R\e)\times[0,R\e])$ is indeed satisfied for $R\e$ sufficiently small. The only nontrivial assumption is that $z(u)-z(0_{d+1})\geq \e\|u\|_{H,1}$ for all $u\in D\setminus \U(R\e/2)$. We must show that this holds either under Assumption~\ref{assume:global2} (as in Propositions~\ref{prop:main},~\ref{prop:explicitL2}) or under convexity of $z$ (as in Propositions~\ref{main:cvx},~\ref{L2:cvx}). In the latter case, this is straightforward, since convexity guarantees linear growth starting at any point. In the former case, Assumption~\ref{assume:global2} only guarantees linear growth for all $\|u\|_{H,1}\geq\rho_1$ (and the coefficient is indeed at least $\e$ starting at this point, since we assume $s\geq\e$ in~\eqref{A}). We can then prove that we also have linear growth in the local region $R\e/2\leq\|u\|_{H,1}\leq \rho_1$ using a Taylor expansion, since $z$ is approximately quadratic in $x$ and linear in $t$. Recall from~\eqref{C1def} that $\cmin$ controls the behavior of $z(u)$ for $\|u\|_{H,1}\leq \rho_1$. Ensuring that the coefficient of linear growth in $R\e/2\leq\|u\|_{H,1}\leq \rho_1$ is at least $\e$ explains why we impose the condition $R\geq12/\cmin$ in Propositions~\ref{prop:main},~\ref{prop:explicitL2}. 

We now turn to the proof of Propositions~\ref{prop:main} and~\ref{main:cvx}. We start with an important preliminary lemma, which is proved in Appendix~\ref{app:deriv}. 
\begin{lemma}\label{lma:deriv-change}The derivative bounds~\eqref{forq},~\eqref{forw} of Proposition~\ref{prop:main} imply 
\begin{align}
\sup_{x\in\us(R\e)}\|\nabla_x^\ell\z(x,0)\|_H&\leq C_{\z}d^{\lceil\ell/2\rceil-2},\quad\ell=3,\dots,2N+2,\label{w-deriv-cond}\\
\sup_{x\in\us(R\e)}\|\nabla_x^\ell\g_k(x)\|_H&\leq C_{\z,\g}d^{\lceil\ell/2\rceil},\quad \ell=0,\dots,2(N+1-k),\;k=1,\dots,N,\label{qk-deriv-cond}\\
\sup_{x\in \us(R\e),y\in[0, (R\e)^2]}|D_{\z}^{k-1}\g(x,y)|&\leq C_{\z,\g},\quad k=1,\dots,N+1.\label{qk-0-cond}
\end{align} The constant $C_{\z}$ depends on $C_{\z,\ell}$, $\ell=0,\dots,2N+2$ from~\eqref{forw}. The constant $C_{\z,\g}$ is similar and also depends on $C_{\g,\ell}$, $\ell=0,\dots,2N$ from~\eqref{forq}.
\end{lemma}
We now apply Theorem~\ref{thm:lap:orig} with $L=N+1-k$, $\g=\g_k$, and $\z=\z(\cdot,0)$. The conditions of the theorem are satisfied thanks to Lemma~\ref{lma:deriv-change}.
 \begin{corollary}\label{corr:nu-ell-qk}Let $R \geq 12/\cmin(\rho_1) + 2(1+2N)(\log\la)/d$ as in Proposition~\ref{prop:main},  and let $\Rem_{N+1-k}^{\mathrm{Lap}}(\g_k)$ be as in~\eqref{lap-gz-L}. Under the conditions of either Proposition~\ref{prop:main} or Proposition~\ref{main:cvx}, it holds
\bs\label{nu-ell-rem}
|\nu_\ell(\g_k)|&\leq C_{\z,\g}d^{2\ell},\quad\ell=0,1,\dots,N-k,\\
\Rem_{N+1-k}^\Lap(\g_k)&\leq  C_{\z,\g,R}\left\{(d^2/\la)^{N+1-k}+ e^{-(R-1)^2d/4}\right\},\quad k=1,\dots, N.
\es Here, $C_{\z,\g}$ depends on the same constants as in Lemma~\ref{lma:deriv-change}, and $C_{\z,\g,R}$ also depends on $\CR$.
 \end{corollary}
The proof is immediate from Theorem~\ref{thm:lap:orig}. Substituting the bounds from Lemma~\ref{lma:deriv-change} and Corollary~\ref{corr:nu-ell-qk} into the preliminary bound~\eqref{cm-remL-tail} finishes the proof of Propositions~\ref{prop:main} and~\ref{main:cvx}; see Appendix~\ref{app:secgenproofs} for the details.

Next, we turn to the proof of Proposition~\ref{prop:explicitL2}.
\begin{proof}[Proof of Proposition~\ref{prop:explicitL2}] Lemma~\ref{lma:mu} shows the assumptions of Corollary~\ref{corr:genprelim} are satisfied. We apply the corollary with $g\equiv1$ and $N=1$. Note that $g\equiv1$ implies $\g\equiv1$ and thus $\g_1(x)=1/\pa_y\z(x,0)$, recalling~\eqref{g12def}. Since $\pa_y\z(0_{d+1})=1$, we have $\g_1(0)=1$. Thus $c_1=\nu_0(\g_1)=\g_1(0)=1$. We obtain
\be\label{1ZDg}
\frac{1}{\mathcal Z}\int_{D} e^{-\la z(u)}\dd u =\frac{1}{\la}\left(\g_1(0)+\Rem_1\right) = \frac1\la(1+\Rem_1),
\ee where
\bs\label{Rem2lapg1}
|\Rem_1|&\les |\Rem_1^\Lap(\g_1)|+\mu\|D_{\z}\g\|_\infty\la^{-1}+(\mu\vee1)(1+\|\g\|_\infty\la^{-1/2})\la^{-M}\\
&\les_{R,\CL} |\Rem_1^\Lap(1/\pa_y\z(\cdot,0))|+\|\pa_y^2\z/(\pa_y\z)^2\|_\infty\la^{-1}+\la^{-M}\\
&\les_{R,\CL} |\Rem_1^\Lap(1/\pa_y\z(\cdot,0))|+\omega_{0,2}(R\e,(R\e)^2)\la^{-1}+\la^{-M}.
\es To get the second line, we used that $\|\g\|_\infty=1$, that $\g_1(x)=1/\pa_y\z(x,0)$, and that $D_{\z}\g = \pa_y(1/\pa_y\z) =-\pa_y^2\z/(\pa_y\z)^2$. We also used Lemma~\ref{lma:mu} to absorb $\mu$ into the constant. To get the third line, we used that $\pa_y\z\geq1/2$ in $\Theta (\us(R\e)\times[0,R\e])$ (since the assumptions of Corollary~\ref{corr:genprelim} are verified) and therefore $\|\pa_y^2\z/(\pa_y\z)^2\|_\infty\les\omega_{0,2}(R\e, (R\e)^2)$. To conclude, we apply Lemma~\ref{lma:D1L2} below.\end{proof}
\begin{proof}[Proof of Proposition~\ref{L2:cvx}]
Lemma~\ref{lma:mu} shows the assumptions of Corollary~\ref{corr:genprelim} are satisfied. We now proceed exactly as in the above proof of Proposition~\ref{prop:explicitL2}.
\end{proof}

\begin{lemma}\label{lma:D1L2}
Under the assumptions of Proposition~\ref{prop:explicitL2}, we have 
\bs
|\Rem_1^\Lap(1/\pa_y\z(\cdot,0))| \les_{R,\CL} &\left(\omega_{3,0}^2+\omega_{4,0}(R\e,0)\right)\frac{d^2}{\la} + \left(\omega_{1,1}^2+\omega_{2,1}(R\e,0)\right)\frac{d}{\la} + \la^{-M}.
\es 
\end{lemma}
See Appendix~\ref{app:secgenproofs} for the proof, which is a direct application of Theorem~\ref{thm:lapexplicit} with $\g=1/\pa_y\z(\cdot,0)$ and $\z=\z(\cdot,0)$. 

\subsection{Outline for proof of Proposition~\ref{prop:TV}}\label{app:outline:TV}In this section, we assume without loss of generality that $z(0_{d+1})=0$, and we use the notation $$\us:=\us(R\e),\qquad I:= [0,(R\e)^2].$$ Throughout the section, we implicitly assume the conditions of Proposition~\ref{prop:TV} hold. Omitted proofs can be found in Appendix~\ref{app:sec:sample}. 

Recall from Section~\ref{app:outline:int} that the crux of the approach to the expansion of $\int_D ge^{-\la z}$ is to first integrate in the $t$ direction. The result is a sum of standard Laplace-type integrals $\int_{\us}\g_k(x) e^{-\la\z(x,0)}\dd x$ with minimum in the interior of the region of integration; see the righthand side of~\eqref{D1intdecomp}. We then apply results of~\cite{A24} on the expansion of high-dimensional Laplace integrals with interior minimum.

The proof of Proposition~\ref{prop:TV} is analogous to this approach. The error arising from the $t$ direction is new, while the error arising in the $x$ direction bears some resemblance to that arising in~\cite{katskew} on the accuracy of the Laplace approximation (LA) to an unrestricted probability density $\mu(x)\propto\exp(-\la \z(x,0))$, for which the global minimizer $x=0_d$ of $\z(x,0)$ lies in the interior of the density's support. Thus, some of the techniques of that work can be leveraged to handle the error in the $x$ direction. To explain this in more detail, we zoom in on one of the two key sources of the TV error, namely, the integral $\frac1B\int_{D_1}|a(u)-b(u)|\dd u$ in~\eqref{sum} below. As shown in the proof of Lemma~\ref{lma:BD}, it can be written as follows:
\be\label{ab}\begin{gathered}
\frac1B\int_{D_1}|a(u)-b(u)|\dd u=\E\left[\left|e^{-\la(q(X)+r(X,Y))} - e^{-\la\Delta(X)}\right|\ind_{\us}(X)\ind_{I}(Y)\right],\\
q(x)=\z(x,0)-\frac12x^\top Hx, \qquad r(x,y)=\z(x,y)-\z(x,0)-y,\\
X\sim\mathcal N(0_d,(\la H)^{-1}),\;Y\sim\mathrm{Exp}(\la).
\end{gathered}
\ee
Here, $X$ and $Y$ are independent, and $\us=\us(R\e)$, $I=[0,(R\e)^2]$. We must show that both $e^{-\la(q(X)+r(X,Y))}$ and $e^{-\la\Delta(X)}$ are close to 1 in expectation. The random variable $\exp(-\la q(X))$ arises in~\cite{katskew} when studying the LA error. Indeed, note that $e^{-\la q(x)}$ is the ratio of $\exp(-\la \z(x,0))$ and $\exp(-\la x^\top Hx/2)$. This is precisely the ratio of an unnormalized density and its LA, obtained by replacing the exponent of the density by its quadratic Taylor approximation at the global minimizer. The term $e^{-\la r(X,Y)}$ is attributable to the Taylor approximation in the $y$ direction, and must be handled using new techniques. The term $e^{-\la\Delta(X)}$, stemming from the quadratic boundary approximation, also does not arise in the LA proof. However, since $\Delta(x)$ is itself the error from a quadratic Taylor approximation, we can control $e^{-\la\Delta(X)}$ as though it came from a LA, using the techniques of~\cite{katskew}.

Another difference with the LA proof of~\cite{katskew} is that an additional integral must be controlled, stemming from the differing supports of $\pi$ and $\hat\pi$. This is the term $\frac1B\int_{D_1\setminus \hat D}a(u)\dd u$ in~\eqref{sum} below. We show in the proof of Lemma~\ref{lma:BD} that it can also be written in terms of the random variables $e^{-\la q(X)}$ and $e^{-\la\Delta(X)}$, and tackled using the techniques of~\cite{katskew}.\\

We now present a series of lemmas which completes the proof of Proposition~\ref{prop:TV}. We remark that, while the proof is inspired by that of~\cite{katskew} on the LA (as discussed above), we will instead rely on results from~\cite{A24} in some of these lemmas. This is because~\cite{A24} distills the results of~\cite{katskew} in a convenient way. First, using simple manipulations, we break up the TV distance between $\pi\vert_D$ and $\hat\pi$ into the following component bounds. 
\begin{lemma}\label{lma:AB}
Let $\hat\psi(x)=\frac12x^\top\nabla^2\psi(0_d)x$ and
\be\begin{gathered}
A = \int_De^{-\la z(u)}\dd u,\quad B =  \frac{(2\pi/\la)^{d/2}}{\la\sqrt{\det H}},\\
a(u) = e^{-\la z(u)}, \qquad b(x,t)=e^{-\la x^\top Hx/2 - \la (t-\hat\psi(x))}.
\end{gathered}\ee 
Then 
\bs\label{sum}
\mathrm{TV}(\pi\vert_D,\hat\pi) \leq &|1-A/B| + \frac1B\int_{D_2\cup D_3}a(u)\dd u + \frac1B\int_{\hat D\setminus D_1}b(u)\dd u \\
&+\frac1B\int_{D_1\setminus \hat D}a(u)\dd u + \frac1B\int_{D_1}|a(u)-b(u)|\dd u,
\es where $D_1,D_2,D_3$ are as in~\eqref{D123} and $\hat D = \{(x,t)\in\br^{d+1}\,:\,t\geq\hat\psi(x)\}$.
\end{lemma}
The three summands on the righthand side in the first line are straightforward to handle. The first term, $1-A/B$, is precisely $\Rem_1$ from Proposition~\ref{prop:explicitL2} (recall that we have assumed $z(0_{d+1})=0$). Thus we have
\bs\label{1AE} |1-A/B| \les_{R,\CL} &\left(\omega_{3,0}^2+\omega_{4,0}(R\e,0)\right)\frac{d^2}{\la} \\
&+ \left(\omega_{1,1}^2+ \omega_{2,1}(R\e,0)\right)\frac{d}{\la}+ \omega_{0,2}(R\e,(R\e)^2)\frac{1}{\la}  + \la^{-M}.
\es
The second summand in~\eqref{sum} is a tail integral and can be bounded using Lemmas~\ref{lma:D2} and~\ref{lma:D3}. For the third summand, note that $b(u)/B$ is precisely the density of $\hat\pi$ (inside its support), which is a pushforward of the density of a pair of independent Gaussian and exponential densities, for which tail bounds are readily available. Thus the third summand can also be easily bounded. Specifically, we have the following lemma.  
 \begin{lemma}\label{lma:tailAB} We have
\begin{align}
\frac{1}{B}\int_{D_2\cup D_3}a(u)du&\les_{R, \CL}\,\la^{-M},\label{agbd} \\
 \frac1B\int_{\hat D\setminus D_1}b(u)\dd u &\les\la^{-M}.\label{agbd2}
\end{align}
\end{lemma}
Next, we simplify the final two summands in~\eqref{sum} as follows, using simple manipulations.
\begin{lemma}\label{lma:BD}
Let $X\sim\mathcal N(0_d, (\la H)^{-1})$ and $Y\sim\mathrm{Exp}(\la)$ be independent. Let $q(x)= \z(x,0)-\frac12x^\top Hx$,  $\Delta(x)=\psi(x)-\hat\psi(x)$, and $r(x,y)=\z(x,y) - \z(x,0)-y$. Finally, let $\Delta, q, r, \ind_{\us\times I}$ be shorthand for $\Delta(X),q(X),r(X,Y), \ind_{\us\times I}(X,Y)$, respectively. Then
\bs\label{eq:bd}
Q&:=\frac1B\int_{D_1\setminus \hat D_1}a(u)\dd u + \frac1B\int_{D_1}|a(u)-b(u)|\dd u\\
&\leq 2\max_{|t_1|,|t_2|\leq 1}\E\left[\la(|\Delta|+|q| +|r|)e^{\la(t_1\Delta +t_2q+|r|)}\ind_{\us\times I}\right].
\es
\end{lemma}
Next, we bound the function $r$ over $\us\times I$.
\begin{lemma}\label{lma:rxy}Let $r(x,y)$ be as in Lemma~\ref{lma:BD}. For all $x\in\us,\,y\in I$, it holds
\bs
|r(x,y)|\leq \left\{|x^\top\nabla_x\pa_y\z(0_{d+1})| +\omega_{2,1}(R\e,0)\|x\|_H^2\right\}y + \omega_{0,2}(R\e, (R\e)^2)y^2\leq y/2.
\es
\end{lemma}
\begin{proof}We have 
\bsn
r(x,y)=& \z(x,y)-\z(x,0)-\pa_y\z(x,0)y \\
&+ (\pa_y\z(x,0)-\pa_y\z(0_{d+1}) - x^\top \nabla_x\pa_y\z(0_{d+1}))y + x^\top \nabla_x\pa_y\z(0_{d+1})y,\esn and we conclude the first inequality by a Taylor remainder bound. To prove the second inequality from the first, note that for all $x\in\us,\,y\in I$, it holds
\bsn
\big\{|x^\top\nabla_x\pa_y\z(0_{d+1})| &+\omega_{2,1}(R\e,0)\|x\|_H^2\big\}y + \omega_{0,2}(R\e, (R\e)^2)y^2\\
&\leq \left[R\e\omega_{1,1}+\omega_{2,1}(R\e,0)(R\e)^2 +  \omega_{0,2}(R\e, (R\e)^2)R\e\right]y \leq y/2.
\esn by the first line of~\eqref{Rwc34eps}.
\end{proof}
Recall that $Y\sim\mathrm{Exp}(\la)$, which has density $\la e^{-\la y}$. The fact that $e^{\la|r(x,y)|}\leq e^{\la y/2}$ means the term $e^{\la|r|}$ in~\eqref{eq:bd} acts as a change of measure.  Thus the upper bound on $Q$ in~\eqref{eq:bd} can be written as an expectation with respect to $X\sim\mathcal N(0_d, (\la H)^{-1})$ and $\bar Y\sim\mathrm{Exp}(\la/2)$, without the term $e^{\la |r|}$. We then apply Cauchy-Schwarz. This yields the following result.
\begin{corollary}
Let $X\sim\mathcal N(0_d, (\la H)^{-1})$ and $\bar Y\sim\mathrm{Exp}(\la/2)$ be independent, and $Q$ be as in~\eqref{eq:bd}. Then
\bs\label{Qbd}
Q\les &\,\la\left(\E[\Delta(X)^2\ind_{\us}(X)]^{1/2} + \E[ q(X)^2\ind_{\us}(X)]^{1/2}+\E[ r(X,\bar Y)^2\ind_{\us\times I}(X,\bar Y)]^{1/2}\right)\\
&\times \max_{|t_1|,|t_2|\leq2}\E\left[e^{\la(t_1\Delta+t_2q)(X)}\ind_{\us}(X)\right].
\es
\end{corollary}
\begin{proof}
For any function $f(x,y)\geq 0$, we have 
\bsn
\E[f(x,y)&e^{\la|r(X,Y)|}\ind_{\us\times I}(X,Y)]\\
& \leq (2\pi/\la)^{-d/2}\sqrt{\det H}\int f(x,y)e^{\la y/2}(e^{-x^\top Hx/2}\la e^{-\la y})\ind_{\us\times I}(x,y)\dd x \dd y\\
&=2\E[f(X, \bar Y)\ind_{\us\times I}(X, \bar Y)].\esn We apply this to the expectation in the second line of~\eqref{eq:bd}, and conclude by Cauchy-Schwarz.
\end{proof}

Next, we use the first bound on $r$ in Lemma~\ref{lma:rxy} to bound $\la\E[ r(X,\bar Y)^2\ind_{\us\times I}(X,\bar Y)]^{1/2}$. This is a straightforward computation, since the upper bound on $r$ is a polynomial, whose expectation under Gaussian $X$ and exponential $\bar Y$ can be computed explicitly.
\begin{lemma}\label{lma:Er2}
Let $X\sim\mathcal N(0_d, (\la H)^{-1})$ and $\bar Y\sim\mathrm{Exp}(\la/2)$ be independent and $r$ be as in Lemma~\ref{lma:rxy}. Then
\be\label{eq:er2}
\la\E[ r(X,\bar Y)^2\ind_{\us\times I}(X,\bar Y)]^{1/2} \les \omega_{1,1}\frac{1}{\sqrt\la}+ \omega_{2,1}(R\e,0)\frac d\la + \omega_{0,2}(R\e, (R\e)^2)\frac1\la.
\ee
\end{lemma}
Next, we bound the remaining expectations in~\eqref{Qbd}, involving the functions $\Delta$ and $q$ defined in Lemma~\ref{lma:BD}. Note that both of these functions are given by Taylor remainders after a second-order expansion. Indeed, recall $\z(\cdot,0)$ is minimized at $x=0_d$, and we have assumed $\z(0_{d+1})=0$. Also, $H=\nabla_x^2\z(0_{d+1})$. Thus $\frac12x^\top Hx$ is the second-order Taylor expansion of $\z(x,0)$. Similarly, recall $\psi(0_d)=0$ and $\nabla\psi(0_d) =0_d$, so $\hat\psi(x)=\frac12x^\top\nabla^2\psi(0_d)x$ is the second-order Taylor expansion of $\psi(x)$. The following lemma is based on the more general Lemma~\ref{lma:Ef2}, which bounds expectations involving such Taylor remainders. 
\begin{lemma}\label{lma:Eb2}Let $X\sim\mathcal N(0_d, (\la H)^{-1})$ and $q, \Delta$ be as in Lemma~\ref{lma:BD}. We have
\begin{align}\label{labet}
\la\E[q(X)^2\ind_{\us}(X)]^{1/2}&\les\omega_{3,0}\frac{d}{\sqrt\la} + \omega_{4,0}(R\e,0)\frac{d^2}{\la},\\
\la\E[\Delta(X)^2\ind_{\us}(X)]^{1/2} &\les \delta_3\frac{d}{\sqrt\la} + \delta_4(R\e)\frac{d^2}{\la}\label{laDel},\\
\E \left[e^{\la(t_1q+t_2\Delta)(X)}\ind_{\us}(X)\right] &\leq C(R,\CL,c_\delta, |t_1|, |t_2|).\label{elat}
\end{align}
\end{lemma}
\begin{proof}We apply~\eqref{Er2} of Lemma~\ref{lma:Ef2} with $f(x)=\z(x,0)$ and $f(x)=\psi(x)$ to obtain the bounds~\eqref{labet} and~\eqref{laDel}, respectively. In the first case we have $f_3 = \omega_{3,0}$, $f_4=\omega_{4,0}(R\e,0)$,  which satisfy~\eqref{f34} with $\Cf = \CL$, thanks to the second line of~\eqref{Rwc34eps}. In the second case we have $f_3=\delta_3$ and $f_4=\delta_4(R\e)$, which satisfy~\eqref{f34} with $\Cf = c_{\delta}$, thanks to the assumption in Proposition~\ref{prop:TV}.

To prove~\eqref{elat}, we apply~\eqref{Erexp} of Lemma~\ref{lma:Ef2}, with $f(x)=t_1\z(x,0) +t_2\psi(x)$, since then $t_1q+t_2\Delta$ corresponds to $h$ in the lemma statement. Note that we can take $f_3=t_1\omega_{3,0}+t_2\delta_3$, and $f_4=t_1\omega_{4,0}(R\e,0)+t_2\delta_4(R\e)$. Then~\eqref{f34} is satisfied with $\Cf=C(|t_1|, |t_2|, \CL, c_\delta)$, thanks to~\eqref{Rwc34eps} and the assumption $(\delta_3^2+\delta_4(R\e))d^2/\la\leq c_\delta$. Thus we obtain~\eqref{elat}.
\end{proof}
 
 \begin{proof}[Proof of Proposition~\ref{prop:TV}]
We substitute~\eqref{eq:er2},~\eqref{labet},~\eqref{laDel},~\eqref{elat} into~\eqref{Qbd} to get
\bs\label{Qbd2}
Q\les_{R,\CL,c_{\delta}}& \left(\omega_{3,0}+\delta_3\right)\frac{d}{\sqrt\la} + \left(\omega_{4,0}(R\e,0)+\delta_4(R\e)\right)\frac{d^2}{\la} \\
&+ \omega_{1,1}\frac{1}{\sqrt\la}+ \omega_{2,1}(R\e,0)\frac d\la + \omega_{0,2}(R\e, (R\e)^2)\frac1\la
\es
Next, recall $Q$ is defined as in~\eqref{eq:bd}, which in turn is given by the last two summands in~\eqref{sum}. Thus substituting~\eqref{Qbd2} in the second line of~\eqref{sum}, and substituting~\eqref{1AE},~\eqref{agbd},~\eqref{agbd2} in the first line of~\eqref{sum} now gives
\bs
\mathrm{TV}(\pi\vert_D,\hat\pi)\les_{R,\CL,c_{\delta}}& \left(\omega_{3,0}+\delta_3\right)\frac{d}{\sqrt\la} + \left(\omega_{3,0}^2+\omega_{4,0}(R\e,0)+\delta_4(R\e)\right)\frac{d^2}{\la} \\
&+ \omega_{1,1}\frac{1}{\sqrt\la}+ \left(\omega_{2,1}(R\e,0)+\omega_{1,1}^2\right)\frac d\la + \omega_{0,2}(R\e, (R\e)^2)\frac1\la+ \la^{-M}.
\es Finally, we use $\omega_{3,0}^2d^2/\la \les_{\CL}\omega_{3,0}d/\sqrt\la$ to conclude.
 \end{proof}
\appendix
\section{Omitted proofs from Section~\ref{app:outline:int}}\label{app:secgenproofs}
\begin{proof}[Proof of Lemma~\ref{lma:D3}]
Recalling from~\eqref{Z-rho-def} the definition of $\mathcal Z$, we have
\bs
I_3:=\frac{1}{\mathcal Z}\int_{D_3}e^{-\la z(u)}\dd u &= \left(\frac{\la}{2\pi}\right)^{d/2}\sqrt{\det H}\int_{D_3}e^{-\la (z(u)-z(0))}\dd u\\
&\leq \left(\frac{\la}{2\pi}\right)^{d/2}\sqrt{\det H}\int_{\|u\|_{H,1}\geq \frac R2\e }e^{-\la\e \|u\|_{H,1}}\dd u\\
&= \left(\frac{\la}{2\pi}\right)^{d/2}\int_{\|u\|\geq \frac R2\e }e^{-\sqrt{\la d}\|u\|}\dd u.
\es
We now switch to polar coordinates and then change variables. Let $S_d$ be the surface area of the $(d+1)$-dimensional sphere. Then
\bs
I_3&\leq \left(\frac{\la}{2\pi}\right)^{d/2}S_d\int_{\frac R2\sqrt{d/\la} }^\infty s^d e^{-\sqrt{\la d} s}\dd s\\
&= \left(\frac{\la}{2\pi}\right)^{d/2}(\sqrt{\la d})^{-(d+1)}S_d\int_{R d/2}^\infty s^de^{-s}\dd s\\
&= \sqrt{\frac{2\pi}{\la}}d^{-(d+1)/2}\frac{S_{d}}{(2\pi)^{(d+1)/2}} \int_{R d/2}^\infty s^de^{-s}\dd s.
\es
We use that $\frac{S_{d}}{(2\pi)^{(d+1)/2}} \leq (e/(d+1))^{\frac {d+1}{2}-1}$ and the inequality
$$
\int_{a}^\infty s^{b-1}e^{-s}\dd s\leq   e^{-a/2}(2b)^b 
$$ 
(see~\cite{A24}) to get
\bs
I_3 &\leq \sqrt{\frac{2\pi}{\la}}d^{-(d+1)/2}\frac{e^{(d+1)/2 -1}}{(d+1)^{(d+1)/2-1}}(2d+2)^{d+1}e^{-R d/4}\\
&= \frac1e\sqrt{\frac{2\pi}{\la}}\left(1+\frac1d\right)^{\frac{d+1}{2}}(d+1)\left(2\sqrt e\right)^{d+1}e^{-R d/4}\les \frac{d}{\sqrt\la}e^{d(1-R/4)}.\es
\end{proof}
\begin{proof}[Proof of Corollary~\ref{corr:genprelim}]Note that the assumptions of the above four lemmas are satisfied. Substituting~\eqref{lap-gz-L} into~\eqref{D1intdecomp} gives
\be\label{prelimdecomp}
\frac{1}{\mathcal Z}\int_{D_1} g(u)e^{-\la z(u)}\dd u =\sum_{m=1}^{N}c_m\la^{-m} + \left(\sum_{k=1}^{N}\Rem_{N+1-k}^\Lap(\g_k)\la^{-k} + \Rem_{N+1}^\Wat\right),
\ee where
\be
c_m = \sum_{k=1}^m\nu_{m-k}(\g_k),\quad m=1,\dots,N.
\ee Using~\eqref{prelimdecomp}, we have
\bsn
\frac{1}{\mathcal Z}\int_{D} &g(u)e^{-\la z(u)}\dd u = \frac{1}{\mathcal Z}\int_{D_1} g(u)e^{-\la z(u)}\dd u + \frac{1}{\mathcal Z}\int_{D\setminus D_1} g(u)e^{-\la z(u)}\dd u\\
&=\sum_{m=1}^{N}c_m\la^{-m} + \left(\sum_{k=1}^{N}\Rem_{N+1-k}^\Lap(\g_k)\la^{-k} + \Rem_{N+1}^\Wat+ \frac{1}{\mathcal Z}\int_{D\setminus D_1} g(u)e^{-\la z(u)}\dd u\right).
\esn The expression in parentheses is then by definition $\frac{1}{\la}\Rem_{N}$. Next, note that $D\setminus D_1\subset D_2\cup D_3$ by Lemma~\ref{lma:D123}. Using Lemmas~\ref{lma:D1}-\ref{lma:D3}, and the fact that $|g(u)|\leq1$ on $D$, gives
\bs\label{cm-remL-tail-0}
\frac{1}{\la}|\Rem_{N}|\les &\sum_{k=1}^{N}\Rem_{N+1-k}^\Lap(\g_k)\la^{-k} + \mu\left(\|D_{\z}^{N}\g\|_\infty\la^{-N-1}+e^{-Rd/2}\max_{1\leq k\leq N}\|D_{\z}^{k-1}\g\|_\infty\la^{-k}\right)\\
&+\frac{\mu}{\la} e^{- R^2d/2}+\frac{d}{\sqrt\la}e^{d-R d/4}\\
\les &\sum_{k=1}^{N}\Rem_{N+1-k}^\Lap(\g_k)\la^{-k} + \mu\|D_{\z}^{N}\g\|_\infty\la^{-N-1}\\
&+\frac{\mu\vee1}{\sqrt\la}\left(1+\max_{1\leq k\leq N}\|D_{\z}^{k-1}\g\|_\infty\la^{-(k-1/2)}\right)de^{d-Rd/4}.
\es 
Next, $R\geq 4(1+\frac{\log d}{d}+(M+1/2)\frac{\log\la}{d})$ implies $de^{d-Rd/4} \leq \la^{-M-1/2}$. Substituting this bound into the last line of~\eqref{cm-remL-tail-0} and then multiplying both sides by $\la$ concludes the proof.
\end{proof}

Next, we turn to the proof of Lemma~\ref{lma:mu}. We first focus on the subproblem of showing
\be\label{zeps}
z(u)-z(0_{d+1})\geq \e\|u\|_{H,1}\qquad\forall u\in D\setminus\U(R\e/2).
\ee under the conditions of Propositions~\ref{prop:main},~\ref{main:cvx},~\ref{prop:explicitL2}, and~\ref{L2:cvx}. We start with the following auxiliary result, which shows that we have linear growth in any bounded annulus. To that end, we first generalize the definition~\eqref{C1def} of $\cmin$: for any $\rho\leq\rho_0$, let
\be\label{C1def2}
\cmin(\rho) = \min\left(\inf_{u\in \U(\rho)}\pa_tz(u),\; \;\inf_{x\in\us(\rho)}(z(x,\psi(x))-z(0_{d+1}))/\tfrac12\|x\|_H^2\right).
\ee
\begin{lemma}[Linear growth in bounded annulus]\label{lma:extendlin}Suppose Assumption~\ref{assume1} holds and fix any $\rho,\rho'<\rho_0$. Then
\be\label{zu1}z(u)-z(0_{d+1}) \geq  \frac{\cmin(\rho)}{6}\min(\delta_2(\rho)^{-1},1,\rho')\|u\|_{H,1}\qquad\forall  \rho'\leq\|u\|_{H,1}\leq\rho,\,u\in D.\ee 
\end{lemma}
Note that Assumption~\ref{assume:global2} can be combined with~\eqref{zu1}. Taking $\rho'=R\e/2$ and $\rho=\rho_1$, we  conclude that 
\be\label{zu1-Re}z(u)-z(0_{d+1}) \geq  \min\left(s,\tfrac16\cmin(\rho_1)\min(\delta_2(\rho_1)^{-1},1,\tfrac12R\e)\right)\|u\|_{H,1}\quad\forall u\in D\setminus \U(R\e/2).\ee We then need only show that this coefficient is bounded below by $\e$.
\begin{proof}[Proof of Lemma~\ref{lma:extendlin}]
Fix $\rho\leq\rho'\leq\rho_0$, and let $\cmin$ be shorthand for $\cmin(\rho')$. Also, in this proof only, let $\delta_2$ be shorthand for $\delta_2(\rho')$ rather than shorthand for $\delta_2(0)$ as usual. Let $u=(x,t)\in D$ be such that $\rho \leq \|u\|_{H,1}\leq \rho'$, provided such a $u$ exists. Then 
\bs
z(x,t) - z(0_{d+1})&= z(x,t)-z(x,\psi(x)) + z(x,\psi(x))-z(0_{d+1}) \\
&\geq \cmin (t-\psi(x)) + \frac{\cmin }{2}\|x\|_H^2,
\es by the definition~\eqref{C1def2} of $\cmin$. Next, a Taylor expansion gives
\be\label{Rdpsi}|\psi(x)| \leq \frac{\delta_2}{2}\|x\|_H^2,\qquad \forall \|x\|_H\leq\rho'.\ee Fix any $\sigma$ such that $(\delta_2\vee1)\sigma \leq 1 $. Since $t-\psi(x)\geq 0$, we have
\bs\label{Rdpsi2}
\cmin (t-\psi(x))&\geq \cmin\sigma (t-\psi(x))\geq \cmin\sigma (|t| -  |\psi(x)|) \geq \cmin\sigma \left(|t| - \frac{\delta_2}{2} \|x\|_H^2\right).
\es  Thus
\bs
z(x,t) - z(0_{d+1})&\geq \cmin\sigma |t| + \frac{\cmin}{2}\left(1 -\delta_2\sigma \right)\|x\|_H^2.\es Now, consider two cases: (1) $\|x\|_H\leq \rho/2 $, and (2) $\|x\|_H\geq\rho/2$. In the first case, since $\|x\|_H^2+t^2\geq \rho^2 $, we must have 
$$t^2\geq\frac34\rho^2 \geq 3\|x\|_H^2.$$ Therefore, $t^2\geq \frac34t^2+\frac34\|x\|_H^2 = \frac34\|(x,t)\|_{H,1}^2$. Thus
\bsn
z(x,t) - z(0_{d+1})&\geq \cmin\sigma |t| \geq  \cmin\frac{\sqrt3}{2}\sigma \|(x,t)\|_{H,1}.
\esn
 In the second case,
\bsn
z(x,t) - z(0_{d+1})&\geq  \cmin\sigma|t| + \frac{\cmin}{2}\left(1 -\delta_2\sigma\right)\|x\|_H^2\\
&\geq \cmin\sigma|t| + \cmin\left(1 -\delta_2\sigma \right)\frac{\rho}{4} \|x\|_H\\
&\geq\cmin\min\left(\sigma, (1 -\delta_2\sigma)\frac{\rho}{4}\right) \|(x,t)\|_{H,1}.\esn
Combining the above two displays and writing $u$ in place of $(x,t)$ gives
\bs\label{zxtd}
z(u)-z(0_{d+1})&\geq \cmin\min\left(\frac{\sqrt3}{2}\sigma, \left(1 -\delta_2\sigma \right)\frac{\rho}{4}\right) \|u\|_{H,1}.
\es
We now set $\sigma= \frac13(1\vee\delta_2)^{-1}$, which satisfies the above requirement that $(\delta_2\vee1)\sigma \leq 1$. With this $\sigma$, we get from~\eqref{zxtd} that
\be\label{zxtd1}
z(u)-z(0_{d+1})\geq \frac{\cmin}{6}\min((\delta_2\vee1)^{-1},\rho)\|u\|_{H,1},\qquad\forall u\in D,\;\rho \leq \|u\|_{H,1}\leq\rho',
\ee  as desired.
\end{proof}
If $z$ is convex, then we can show~\eqref{zeps} more directly. The key point is that the function $\rho\mapsto \inf_{\|u\|_{H,1}\geq\rho}(z(u)-z(0_{d+1}))/\|u\|_{H,1}$ is an increasing function when $z$ is convex, so it suffices to lower bound $(z(u)-z(0_{d+1}))/\|u\|_{H,1}$ over all $\|u\|_{H,1}=R\e/2$.
\begin{lemma}[Verifying~\eqref{zeps} under convexity]\label{lma:cvx}
Suppose Assumption~\ref{assume1} holds and $D$ is star-shaped about $0_{d+1}$, i.e. for all $u\in D$, the line segment between $0_{d+1}$ and $u$ is contained within $D$. Also, suppose the restriction of $z$ to any such line segment is convex. Suppose 
\be\label{eqcvx}\e\leq\frac1{12},\quad R\e\leq\rho_0,\quad \e\delta_2(R\e/2)\leq1/12,\quad \cmin(R\e/2)\geq \frac12,\quad R\geq24.\ee Then~\eqref{zeps} holds.
\end{lemma} Since $\e$ is assumed small, we can control $\cmin(R\e/2)$ through bounds on the derivatives of $z$ near zero.
\begin{proof}[Proof of Lemma~\ref{lma:cvx}]Since $R\e/2\leq R\e\leq\rho_0$ we can apply~\eqref{zu1} from Lemma~\ref{lma:extendlin} with $\rho=\rho'=R\e/2$, to get $$
z(u)-z(0_{d+1}) \geq  \frac{\cmin(R\e/2)}{6}\min(\delta_2(R\e/2)^{-1},1,R\e/2)\|u\|_{H,1}\qquad\forall \|u\|_{H,1}=R\e/2,\,u\in D.
$$ The bounds~\eqref{eqcvx} now imply the above coefficient of linear growth is at least $\e$, so that $z(u)-z(0_{d+1})\geq \e\|u\|_{H,1}$ for all $\|u\|_{H,1}=R\e/2$ such that $u\in D$. 
%
Next, fix $u\in D$ such that $\|u\|_{H,1}\geq R\e/2$, and let $u_0$ be the point on the line segment from $0_{d+1}$ to $u$ such that $\|u_0\|_{H,1}=R\e/2$. Then $u_0\in D$ since $D$ is star-shaped about $0_{d+1}$. By convexity of $z$ on the line segment from $0_{d+1}$ to $u$, we have
\be
z(u_0) \leq \left(1-\frac{\rho_1}{\|u\|_{H,1}}\right)z(0_{d+1}) + \frac{\rho_1}{\|u\|_{H,1}}z(u) = z(0_{d+1}) + \frac{\rho_1}{\|u\|_{H,1}}(z(u)-z(0_{d+1})).
\ee
Rearranging, we get 
$$\frac{z(u)-z(0_{d+1})}{\|u\|_{H,1}} \geq \frac{z(u_0)-z(0_{d+1})}{\rho_1} = \frac{z(u_0)-z(0_{d+1})}{\|u_0\|_{H,1}}\geq \e,$$ as desired.\end{proof}
\begin{proof}[Proof of Lemma~\ref{lma:mu}]
Each of the propositions assumes either~\eqref{A} or~\eqref{D}, and both imply $R\e\leq1\wedge\rho_0$ and $R\e\delta_2(R\e/2)\leq 4$. In each proposition, we have at least that $R\geq 12+2(1+2M)(\log\la)/d$ (with $M=N$ in Propositions~\ref{prop:main} and~\ref{main:cvx}), recalling that $\cmin(\rho_1)\leq1$, as noted below the statement of Lemma~\ref{lma:extendlin}. Next, recall $\pa_tz(0_{d+1})=1$, and note that $\pa_y\z(u)=\pa_tz(\Theta(u))$. Thus $\pa_tz(u)\geq1/2$ for all $u\in\Theta (\us(R\e)\times[0,R\e])$ if and only if $\pa_y\z(u)\geq1/2$ for all $u\in \us(R\e)\times[0,R\e]$. Furthermore, a Taylor expansion gives 
\bs\label{wlb}
|\pa_y\z(x,y)-1|&=|\pa_y\z(x,y)-\pa_y\z(0_{d+1})|\\
&\leq |\pa_y\z(x,y)-\pa_y\z(x,0)| + |\pa_y\z(x,0)-\pa_y\z(0_{d+1})| \\
&\leq R\e(\omega_{0,2}(R\e,R\e)+\omega_{1,1}(R\e,0))
\es
Under~\eqref{BC} and~\eqref{forw}, assumed in Propositions~\ref{prop:main} and~\ref{main:cvx}, we have the further bound
\be\label{wlb2}
Re(\omega_{0,2}(R\e,R\e)+\omega_{1,1}(R\e,0)) \leq R\e(C_{\z,0}+C_{\z,1})\leq1/2.\ee Under~\eqref{Rwc34eps}, assumed in Propositions~\ref{prop:explicitL2} and~\ref{L2:cvx}, we have by a Taylor expansion,
\be\label{wlb3}
R\e(\omega_{0,2}(R\e,R\e)+\omega_{1,1}(R\e,0))\leq R\e(\omega_{0,2}(R\e,R\e)+\omega_{1,1})+(R\e)^2\omega_{2,1}(R\e,0) \leq 1/2.
\ee
Combining~\eqref{wlb} and~\eqref{wlb2}, or~\eqref{wlb} and~\eqref{wlb3}, we conclude $\pa_y\z(x,y)\geq 1/2$ for all $(x,y)\in \us(R\e)\times[0,R\e]$ under the assumptions of any one of the four propositions. Next, we show $z(u)-z(0)\geq\e\|u\|_{H,1}$ for all $u\in D_3$. Under the assumptions of Proposition~\ref{prop:main} or~\ref{prop:explicitL2}, the lower bound~\eqref{zu1-Re} holds. We therefore have that the coefficient of linear growth in $D_3$ is at least 
$$\min(s, \frac{\cmin(\rho_1)}{6}\min(\delta_2(\rho_1)^{-1},1,R\e/2)),$$ so it remains to bound this quantity below by $\e$. We assumed in~\eqref{A} that $s\geq\e$. Next, note that 
$$
 \frac{\cmin(\rho_1)}{6}\min(\delta_2(\rho_1)^{-1},1)\geq \frac{2}{R}\min(\delta_2(\rho_1)^{-1},1) \geq\e
$$
by~\eqref{A} and the lower bound on $R$. Finally, $ \frac{\cmin(\rho_1)}{6}(R\e/2)\geq\e$ by the lower bound on $R$. Next, consider the setting of Proposition~\ref{main:cvx} or~\ref{L2:cvx}. To show $z(u)-z(0)\geq\e\|u\|_{H,1}$ for all $u\in D_3$, we verify the conditions of Lemma~\ref{lma:cvx}. Since $R\e\leq1$ and $R\geq24$, we indeed have $\e\leq1/12$. Since $R\e\delta_2(R\e/2)\leq2$ and $R\geq24$, we indeed have $\e\delta_2(R\e/2)\leq1/12$. It remains to show $\cmin(R\e/2)\geq1/2$. To check this, recall from the proof of Lemma~\ref{lma:D123} that $R\e\delta_2(R\e/2)\leq4$ implies $\U(R\e/2)\subseteq \Theta (\us(R\e)\times[0,R\e])$. Therefore since we just showed $\pa_tz(u)\geq1/2$ for all $u\in\Theta (\us(R\e)\times[0,R\e])$, we conclude that $\pa_tz(u)\geq1/2$ for all $u\in\U(R\e/2)$. Thus the first infimum in the definition~\eqref{C1def2} of $\cmin(R\e/2)$ is bounded below by $1/2$. Next, a Taylor expansion gives that for all $x\in\us(R\e/2)$ we have
$$
\left|\frac{z(x,\psi(x))-z(0_{d+1})}{\|x\|_H^2/2}-1\right| =\left|\frac{\z(x,0)-\z(0_{d+1})}{\|x\|_H^2/2}-1\right| \leq \omega_{3,0}(R\e/2, 0)\frac{R\e}{6}.
$$ 
But now in the setting of Proposition~\ref{main:cvx} we have $\omega_{3,0}(R\e/2, 0)\frac{R\e}{6}\leq C_{\z,3}\frac{R\e}{6}\leq 1/2$ by~\eqref{forw} and~\eqref{E}. In the setting of Proposition~\ref{L2:cvx} we directly assumed $\omega_{3,0}(R\e,0)R\e/6\leq 1/2$. Either way, we conclude the  second infimum in the definition~\eqref{C1def} of $\cmin(\cdot)$ is bounded below by $1/2$, so $\cmin(R\e/2)\geq1/2$.

This concludes the proof that the conditions of Corollary~\ref{corr:genprelim} are satisfied under the conditions of any one of the four propositions. Next, we bound $\mu$. To do so, we apply Corollary~\ref{corr:norm-const}. We check the condition~\eqref{Rc34eps} for the function $\z(x)=z(x,0)$. In the context of Propositions~\ref{prop:main} and~\ref{main:cvx}, we use the definition of $C_{\z,3},C_{\z,4}$ from~\eqref{forw}, and the fact that $R^4d^2/\la\leq \CR$, to get
\bs\label{endupwithC34-2}
\left(R^2\sup_{x\in\us(R\e)}\|\nabla_x^3\z(x,0)\|_H\frac{d}{\sqrt\la}\right)^2 + \sup_{x\in\us(R\e)}\|\nabla_x^4\z(x,0)\|_H\frac{d^2}{\la}\leq \CR C_{\z,3}^2+C_{\z,4}=:C_{34}.
\es Thus the corollary applies, and we conclude that $\mu\leq C(\CR,C_{\z,3},C_{\z,4})$. The bound on $\mu$ in the setting of Proposition~\ref{prop:explicitL2} and~\ref{L2:cvx} is proved similarly. 
\end{proof}
\begin{proof}[Proof of Proposition~\ref{prop:main}]
Lemma~\ref{lma:mu} shows the assumptions of Corollary~\ref{corr:genprelim}, with $M=N$, are satisfied. Now, we rewrite the expansion~\eqref{eq:genprelim} as follows:
\bs
\frac{1}{\mathcal Z}\int_{D} g(u)e^{-\la z(u)}\dd u =&\frac{1}{\la}\left(\sum_{m=1}^{N}c_m\la^{-(m-1)} +\Rem_{N}\right) \\
= &\frac{1}{\la}\left( \sum_{m=1}^{N}[d^{-(2m-2)}c_m]\left(\frac{d^2}{\la}\right)^{m-1} +\Rem_{N}\right)\\
= &\frac{1}{\la}\left( \sum_{m=0}^{L-2}a_m\left(\frac{d^2}{\la}\right)^{m} +\Rem_{N}\right),
\es where 
\be
a_m = d^{-2m}c_{m+1} = d^{-2m} \sum_{k=1}^{m+1}\nu_{m+1-k}(\g_k),\qquad m=0,\dots,N-1.
\ee
This confirms the formula~\eqref{amdef} for the $a_m$. Corollary~\ref{corr:nu-ell-qk} now gives
\be
|a_m| \leq \sum_{k=1}^{m+1}|\nu_{m+1-k}(\g_k)| \leq C_{\z,\g}d^{2m}, \quad m=0,\dots,N-1,\ee proving the bound on $a_m$ in~\eqref{RemLmainbd}. To bound $\Rem_{N}$, we use~\eqref{cm-remL-tail} with $M=N$. By Lemma~\ref{lma:mu}, we have 
\be\label{rhobd-2}
\mu\leq C(\CR,C_{\z,3},C_{\z,4}).\ee 
Substituting the bounds~\eqref{rhobd-2},~\eqref{qk-0-cond} and~\eqref{nu-ell-rem} on $\mu$, $\|D_{\z}^{k-1}\g\|_\infty$, $\Rem_{N+1-k}^\Lap(q_k)$ into~\eqref{cm-remL-tail} with $M=N$, gives
\bs\label{RemLwatbd}
|\Rem_{N}| \leq &C_{\z,\g,R}\left\{\sum_{k=1}^{N}\left((d^2/\la)^{N+1-k}+e^{-(R-1)^2d/4}\right)\la^{-(k-1)} + \la^{-N}\right\}\\
\leq &C_{\z,\g,R}\left\{\left(\frac{d^2}{\la}\right)^{N} + e^{-(R-1)^2d/4}\right\},
\es where $C_{\z,\g,R}$ depends on the $C_{\z,\ell}$'s, the $C_{\g,\ell}$'s, and $\CR$. Finally, we use that $R\geq 4(2+(N+1/2)(\log\la)/d)$ to get that $e^{-(R-1)^2d/4}\leq \la^{-N}$. This concludes the proof.
\end{proof}
\begin{remark}[Room for growth with $d$]\label{rk:roomd}Neglecting constants and exponentially small terms,~\eqref{nu-ell-rem} shows that $|\Rem_{N+1-k}^\Lap(q_k)|\les (d^2/\la)^{N+1-k}$. 
Plugging this into~\eqref{cm-remL-tail} gives $|\sum_k\la^{-k}\Rem_{N+1-k}^\Lap(q_k)|\les\sum_k\la^{-(k-1)}(d^2/\la)^{N+1-k} \les (d^2/\la)^{N}$. But notice that there is slack in this bound, since only the summand corresponding to $k=1$ is of size $(d^2/\la)^N$, and all other summands are smaller. Specifically, if instead of~\eqref{nu-ell-rem} it held \be\label{lrb}|\Rem_{N+1-k}^\Lap(q_k)|\les d^{2k-2}(d^2/\la)^{N+1-k},\ee then this would lead to the same total bound on $|\sum_k\la^{-k}\Rem_{N+1-k}^\Lap(\g_k)|$. Similarly, note that if instead of the bound~\eqref{qk-0-cond} on $\|D_{\z}^{k-1}\g\|_\infty$ it held
\be\label{lrb2}\|D_{\z}^{k-1}\g\|_\infty \les d^{2k-2},\qquad\forall k=1,\dots,N+1,\ee then the contribution to $|\Rem_N|$ from $\|D_{\z}^N\g\|_\infty\la^{-N}$ in~\eqref{cm-remL-tail} would be $(d^2/\la)^N$ and the contribution from $\big(1+\max_{1\leq k\leq N}\|D_{\z}^{k-1}\g\|_\infty\la^{-(k-1/2)}\big)\la^{-N}$ would still be $\la^{-N}$.

Since $|\Rem_{N+1-k}^\Lap(\g_k)|$ is linear in $\g_k$, one can obtain the larger remainder bound~\eqref{lrb} by loosening each of the bounds~\eqref{qk-deriv-cond} on the derivatives of $\g_k$ by a factor $d^{2k-2}$. Recalling that $\g_k=D_{\z}^{k-1}\g/\pa_y\z$, the weaker versions of~\eqref{qk-deriv-cond} take the following form:
\be\sup_{x\in\us(R\e)}\|\nabla_x^\ell(D_{\z}^{k-1}\g/\pa_y\z)\|_H\les d^{2k-2}d^{\lceil\ell/2\rceil},\quad \ell=0,\dots,2(N+1-k),\;k=1,\dots,N.\label{qk-deriv-cond2}\ee Thus to summarize, the bounds~\eqref{lrb2} and~\eqref{qk-deriv-cond2}, which are looser than~\eqref{qk-0-cond} and~\eqref{qk-deriv-cond} by a factor of $d^{2k-2}$, would lead to the same overall bound $|\Rem_N|\les (d^2/\la)^N$. 

We now ask: what looser conditions on $\g$ and $\z$ (as compared to the original~\eqref{forq},~\eqref{forw}) imply the looser conditions~\eqref{lrb2} and~\eqref{qk-deriv-cond2}? Lemma~\ref{lma:Dwk} below shows that $D_{\z}^{k-1}\g$ is a sum of products of terms $\pa_y^j\g$, $\pa_y^{j'}\z$, and $1/\pa_y\z$, with $j,j'\leq k$. Thus to fully take advantage of~\eqref{lrb2}, $\pa_y^j\g$, $\pa_y^{j'}\z$ should be allowed to grow as a power of $d$ which depends on $j$ and $j'$. Similarly, to fully take advantage of~\eqref{qk-deriv-cond2}, the growth rate of $\nabla_x^\ell\pa_y^j\g$ and $\nabla_x^\ell\pa_y^{j'}\z$ with $d$ should be allowed to depend on $j$ and $j'$. 

This implies that the way in which the original conditions~\eqref{forq},~\eqref{forw} may be relaxed is by allowing for increasing orders of $y$ derivatives to grow as increasing powers of $d$. Such a situation does not seem typical because $y$ is univariate. As discussed in Remark~\ref{rk:d1}, it is reasonable to expect e.g. $\|\nabla_x^\ell\z\|_H$ to grow as a larger power of $d$ than does $\|\nabla_x^{\ell'}\z\|_H$ for $\ell>\ell'$ simply because $\nabla_x^\ell\z$ is a tensor containing a larger number of elements: $d^\ell$ compared to $d^{\ell'}$. However, $\nabla_x^\ell\nabla_y^j\z$ and $\nabla_x^\ell\nabla_y^{j'}\z$ are both tensors of the same size.

Nevertheless, there may be other reasons why the operator norm of $\nabla_x^\ell\nabla_y^j\z$ may grow as a $j$-dependent power of $d$. We leave the specification of the relaxed conditions on $\g$ and $\z$ to future work.
\end{remark}

\begin{proof}[Proof of Proposition~\ref{main:cvx}]
Lemma~\ref{lma:mu} shows the assumptions of Corollary~\ref{corr:genprelim}, with $M=N$, are satisfied. We now proceed exactly as in the above proof of Proposition~\ref{prop:main}.
\end{proof}
\begin{proof}[Proof of Lemma~\ref{lma:D1L2}]
In the proof, we use the following notation:
\bs\label{s1234-def}
s_1&=\omega_{1,1}=\|\nabla_x\pa_y\z(0,0)\|_{\HH},\qquad s_2=\omega_{2,1}(R\e,0)=\sup_{x\in\us(R\e)}\|\nabla_x^2\pa_y\z(x,0)\|_{\HH},\\
s_3&=\omega_{3,0}=\|\nabla_x^3\z(0,0)\|_{\HH},\qquad s_4=\omega_{4,0}(R\e,0)=\sup_{x\in\us(R\e)}\|\nabla_x^4\z(x,0)\|_{\HH}.
\es Note that 
\bs\label{sups12}
\sup_{x\in\us(R\e)}\|\nabla_x\pa_y\z(x,0)\|_{\HH} &\leq s_1 + R\e s_2,\\
\sup_{x\in\us(R\e)}\|\nabla_x^3\z(x,0)\|_{\HH}& \leq s_3 + R\e s_4.\es Also, note that the derivative bounds~\eqref{Rwc34eps} imply
\be\label{s1234-bd}
\frac{1}{\sqrt\la}s_1 \leq \frac12,\quad \frac d\la s_2\leq \frac12,\quad s_3\frac{d}{\sqrt{\la}} \leq \sqrt{\CL},\quad s_4\frac{d^2}{\la} \leq \CL.
\ee
We now apply Theorem~\ref{thm:lapexplicit} with $\g=1/\pa_y\z(\cdot,0)$ and $\z=\z(\cdot,0)$. Note that $|\g(0_{d+1})|=1$, and $c_{0,\g}(R)=\sup_{x\in\us(R\e)}|1/\pa_y\z(x,0)|\leq2$, since the assumptions of Corollary~\ref{corr:genprelim} are satisfied. We obtain
\bs\label{rem1lapg1}
|\Rem_1^\Lap&(1/\pa_y\z(\cdot,0))|\\
 \les_{R,\CL} &\bigg\{d^{-1}\sup_{\us(R\e)}\|\nabla_x^2(1/\pa_y\z(\cdot,0))\|_{\HH} + d^{-1}\sup_{\us(R\e)}\|\nabla_x(1/\pa_y\z(\cdot,0))\|_{\HH}\big(s_3+\frac{d}{\sqrt\la} s_4\big)\\
 &+s_3^2+s_4\bigg\}\frac{d^2}{\la} + \left(1+ \frac{d}{\sqrt{\la}}|\nabla_x(1/\pa_y\z(0_{d+1}))\|_{\HH}\right)e^{-(R-1)^2d/4}.
\es The suppressed constant depends on $R$ and $\CL$ because~\eqref{Rc34eps} is satisfied for some $\Clap=C(R, \CL)$. Now, the first two $x$ derivatives of $1/\pa_y\z$ are given by
\bs\label{nablaxx}
\nabla_x(1/\pa_y\z) &= -\frac{\nabla_x\pa_y\z}{(\pa_y\z)^2},\\
\nabla_x^2(1/\pa_y\z) &=  -\frac{\nabla_x^2\pa_y\z}{(\pa_y\z)^2} + 2\frac{(\nabla_x\pa_y\z)^{\otimes2}}{(\pa_y\z)^3}.
\es Again using that $\pa_y\z(x,0)>1/2$ for $x\in\us(R\e)$, the first line of~\eqref{nablaxx} gives
\be\label{rem1lapg1-1}
\|\nabla_x(1/\pa_y\z(0,0))\|_{\HH}\les \|\nabla_x\pa_y\z(0,0)\|_H= s_1.
\ee
Similarly,
\be\label{rem1lapg1-2}
\sup_{x\in\us(R\e)}\|\nabla_x(1/\pa_y\z(x,0))\|_{\HH} \les \sup_{x\in\us(R\e)}\|\nabla_x\pa_y\z(x,0)\|_{\HH} \les_R\, s_1 + \e \,s_2,\ee
and
\bs\label{rem1lapg1-3}
\sup_{x\in\us(R\e)}\|\nabla_x^2&(1/\pa_y\z(x,0))\|_{\HH} \les \sup_{x\in\us(R\e)}\|\nabla_x^2\pa_y\z(x,0)\|_{\HH}+\sup_{x\in\us(R\e)}\|\nabla_x\pa_y\z(x,0)\|_{\HH}^2\\
\les_R&\;s_2 + (s_1 + \e \,s_2)^2 \les\; (1+\e^2s_2)s_2 + s_1^2 \les_{R} s_2+s_1^2.
\es 
Let
\be
T:= d^{-1}(s_2+s_1^2)+ d^{-1} (s_1+\e \,s_2)\big(s_3+\frac{d}{\sqrt\la} s_4\big)+s_3^2+s_4.\ee Substituting the bounds~\eqref{rem1lapg1-1},~\eqref{rem1lapg1-2},~\eqref{rem1lapg1-3} into~\eqref{rem1lapg1} gives
\bs\label{Rem1Lkappa}
|\Rem_1^\Lap(\g_1)|&\les_{R} \; Td^2/\la + (1+s_1d/\sqrt\la)de^{-(R-1)^2d/4}\\
&\les_{R}Td^2/\la + de^{-(R-1)^2d/4}\\
&\les_{R}Td^2/\la + \la^{-M}.
\es 
The last line follows using $R\geq4(1+\log d/d + (M+1/2)(\log\la)/d)$, which gives $ de^{-(R-1)^2d/4}\leq de^{-Rd/4}\leq\la^{-M}$. We now simplify $T$. Rearranging terms and applying the bounds~\eqref{s1234-bd}, we have
\bs\label{Tbd}
T &= d^{-1}s_2\left(1+\e (s_3+(d/\sqrt{\la})s_4)\right) + \left(d^{-1}s_1^2+d^{-1}s_1s_3+s_3^2\right) + s_4(1+s_1/\sqrt{\la})\\
&\les_{\CL} d^{-1}s_2 +  \left(d^{-1}s_1^2+d^{-1}s_1s_3+s_3^2\right) + s_4.
\es For the term in parentheses in the middle, note that $d^{-1}s_1s_3 \les (d^{-1}s_1)^2+s_3^2 \leq d^{-1}s_1^2+s_3^2$. Substituting this into~\eqref{Tbd}, we obtain
\be\label{Tbd2}
T\les_{\CL} d^{-1}(s_1^2+s_2) + s_3^2+s_4.
\ee Substituting~\eqref{Tbd2} into~\eqref{Rem1Lkappa} concludes the proof.
\end{proof}

\section{Proofs from Section~\ref{sec:gauss}}\label{app:gauss}
\subsection{Connection to Section~\ref{sec:gen}}\label{app:gauss:2}
Recall that $z(x,t)=\frac12\|x\|^2+\frac12(t+1)^2$ and $D$ satisfies Assumption~\ref{assume:D}. We show that $z,D$ satisfy Assumption~\ref{assume1}. As discussed below Assumption~\ref{assume:D}, it remains only to show $\pa_tz(u)>0$ for all $u\in\U(\rho_0)$. Note that $\pa_tz(x,t)  = t+1$. Next, $(x,t)\in\U(\rho_0)$ implies $x\in\us(\rho_0)\subseteq\Omega$, and therefore $t\geq\psi(x)\geq -\max_{x\in\us(\rho_0)}|\psi(x)| \geq -\delta_2(\rho_0)\rho_0^2/2$, by a Taylor expansion. Thus $\pa_tz(x,t)  \geq 1 - \delta_2(\rho_0)\rho_0^2/2>0$, since we assumed $\delta_2(\rho_0)\rho_0^2<2$ in Assumption~\ref{assume:D}.

\subsection{Elementary calculations}\label{app:gauss:elem}
We first take note of a few derivatives of $\z$. We have
\bs\label{w-quadz-genpsi}
\z(x,y) &= z(x,\psi(x)+y) = \frac12\|x\|^2 + \frac12(\psi(x)+y+1)^2.
\es Therefore
\bs\label{gen-psi-derivs}
\pa_y\z(x,y) &= \psi(x)+y+1,\\
\nabla_x^\ell\pa_y\z(x,y)&=\nabla^\ell\psi(x),\qquad\ell\geq1,\\
\omega_{\ell,1}(\rho,0)&=\delta_\ell(\rho),\qquad\ell\geq1.
\es Next, note that $\nabla_x^\ell\z(x,y)=\frac12\nabla_x^\ell(\psi(x)+y+1)^2$ for all $\ell\geq3$. Taking $y=0$ and omitting the $x$ argument, we have
\bs
\frac12\nabla_x(\psi+1)^2&=(\psi+1)\nabla\psi,\\
\frac12\nabla_x^2(\psi+1)^2&=(\psi+1)\nabla^2\psi +\nabla\psi\nabla\psi^\top,\\
\nabla_x^3\z(\cdot,0)=\frac12\nabla_x^3(\psi+1)^2&=(\psi+1)\nabla^3\psi +3\mathrm{Sym}\left(\nabla\psi\otimes\nabla^2\psi \right),\\
\nabla_x^4\z(\cdot,0)=\frac12\nabla_x^3(\psi+1)^2&=(\psi+1)\nabla^4\psi +\mathrm{Sym}\left(3\nabla^2\psi\otimes\nabla^2\psi +4\nabla\psi\otimes\nabla^3\psi\right).
\es Here, for a $d\times d\times d$ tensor $T$, the tensor $\mathrm{Sym}(T)$ is its symmetrization, given by 
\be\label{sym}\mathrm{Sym}(T)_{ijk} = \frac16(T_{ijk}+T_{ikj}+T_{jik}+T_{jki}+T_{kij}+T_{kji}).\ee The definition of $\mathrm{Sym}(T)$ for a $d\times d\times d\times d$ tensor $T$ is analogous. It is then straightforward to show that in general,
\be
\|\nabla_x^\ell\z(x,0)\|_H \les \|\nabla^\ell\psi(x)\|_H + \sum_{k=0}^\ell\|\nabla\psi^k(x)\|_H\|\nabla^{\ell-k}\psi(x)\|_H,\qquad\ell\geq3.
\ee Thus
\bs\label{w-deltas-gauss}
\omega_{3,0}(\rho,0)&\leq \delta_3(\rho)(1+\delta_0(\rho))+3(\delta_1\delta_2)(\rho),\\
\omega_{\ell,0}(\rho,0)&\les \delta_\ell(\rho)(1+\delta_0(\rho))+\sum_{k=1}^{\ell-1}(\delta_k\delta_{\ell-k})(\rho),\qquad\ell\geq4.
\es Note also that, since $\psi(0)=0$, $\nabla\psi(0)=0$, we have
\be\label{del12tay}
\delta_0(\rho)\leq \frac{\rho^2}{2}\delta_2(\rho),\qquad \delta_1(\rho)\leq \rho\delta_2(\rho).
\ee
\\

\subsection{Proposition proofs and omitted details in Example~\ref{ex:quart}}\label{app:gauss:pf}
\begin{proof}[Proof of Proposition~\ref{prop:gauss:1}]
We check the conditions of Proposition~\ref{main:cvx}, with $g\equiv1$, $z$ as in~\eqref{z-gauss}, and $D$ as in~\eqref{D-Dla}.  We showed in Section~\ref{app:gauss:2} that Assumption~\ref{assume1} is satisfied, and we supposed in Assumption~\ref{assume:D} that $D$ is star-shaped. The function $z(x,t)=\|x\|^2/2+(t+1)^2/2$ is obviously convex. 

The conditions~\eqref{D} are satisfied thanks to~\eqref{cond:gauss} and the fact that $\delta_2(R\e)\leq C_2$, by~\eqref{deltas:gauss}. The condition $R^4d^2/\la \leq \CR$ from~\eqref{BC} is satisfied with $\CR=C(C_0, N)$, using~\eqref{cond:gauss} and the definition of $R$. Next we check~\eqref{forq},~\eqref{forw}, which entails computing the constants $C_{\z,\ell}$; we will then return to check the first bound in~\eqref{BC} and the bound~\eqref{E}, which depend on the values of $C_{\z,0},C_{\z,1},C_{\z,3}$.

The conditions~\eqref{forq} are automatically satisfied since $\g\equiv1$. To check~\eqref{forw}, note using~\eqref{del12tay}, the fact that $R\e\delta_2(R\e)\leq R\e C_2\leq 1/2$, and the fact that $R\e\leq 1/3$, that
\be\label{de12tay}
\delta_0(R\e) \leq \frac{(R\e)^2}{2}\delta_2(R\e) \leq 1/12, \qquad \delta_1(R\e)\leq R\e\delta_2(R\e) \leq 1/2.
\ee Now, recall from~\eqref{gen-psi-derivs} that $\pa_y\z(x,y)=\psi(x)+1+y$. Therefore, $\pa_y^2\z\equiv1$ and $\pa_y^j\z\equiv0$ for all $j\geq3$. Thus $\omega_{0,2}(R\e,R\e)=1$ and $\omega_{0,j}(R\e,(R\e)^2)=0$ for $j\geq3$, so the first line of~\eqref{forw} is satisfied with $C_{\z,0}=1$. Next, consider $\ell=1$, $j=1,\dots, N$. We have $\nabla_x\pa_y\z(x,y)=\nabla\psi(x)$ and therefore $\nabla_x\pa_y^j\z\equiv0$ for all $j\geq2$. Since $\delta_1(R\e)\leq1/2$ by~\eqref{de12tay}, we conclude that for $\ell=1$, the required bound in the second line of~\eqref{forw} is satisfied with $C_{\z,1}=1/2$. Next, consider $\ell=2$, $j=1,\dots, N$. Since $\delta_2(R\e)\leq C_2$, by the same reasoning as for $\ell=1$ we get that for $\ell=1$, the required bound is satisfied with $C_{\z,2}=C_2$. Next, let $\ell=3$, and consider $j=0,\dots, \lfloor N+1-\ell/2\rfloor$. For $j=0$ we have using~\eqref{w-deltas-gauss} and~\eqref{de12tay} that
$$
\omega_{3,0}(R\e,0)\leq \delta_3(R\e)(1+\delta_0(R\e)) + 3(\delta_1\delta_2)(R\e) \leq \frac{13}{12}C_3 +\frac32C_2\leq\frac32(C_3+C_2)
$$ For $j=1$ we have using~\eqref{gen-psi-derivs} that $\omega_{3,1}(R\e,0)= \delta_3(R\e)\leq C_3$. For $j\geq2$ we have $\nabla_x^3\pa_y^j\z\equiv0$, so $\omega_{3,j}(R\e,0)=0$. Thus for $\ell=3$, the required bound is satisfied with $C_{\z,3}=\frac32(C_3+C_2)$. Finally, consider $\ell\geq 4$ and $j=0,\dots, \lfloor N+1-\ell/2\rfloor$. Using~\eqref{w-deltas-gauss} and~\eqref{de12tay} we have
\be\label{Cwell}
\omega_{\ell,0}(R\e,0)\les \delta_\ell(R\e)(1+\delta_0(R\e))+\sum_{k=1}^{\ell-1}(\delta_k\delta_{\ell-k})(R\e)\leq C(C_2,\dots, C_\ell).
\ee For $j=1$ we have using~\eqref{gen-psi-derivs} that $\omega_{\ell,1}(R\e,0)= \delta_\ell(R\e)\leq C_\ell$. For $j\geq2$ we have $\omega_{\ell,j}(R\e,0)=0$. Thus we can take $C_{\z,\ell}=C(C_2,\dots, C_\ell)$. Thus all the conditions of~\eqref{forw} are verified. We return to the first condition in~\eqref{BC} and the bound~\eqref{E}.  Substituting $C_{\z,0}=1,C_{\z,1}=1/2$, we have $1/(2C_{\z,0}+2C_{\z,1}) = 1/3\geq R\e$, by~\eqref{cond:gauss}. Thus~\eqref{BC} is verified. Also, we have $3/C_{\z,3} = 2/(C_2+C_3) \geq 2/(4C_2+C_3)\geq R\e$, as desired.

We now use~\eqref{gaussprob-prelim} and then~\eqref{expansion}. We obtain
\bs
\mathbb P\left(\mathcal N\left(0_{d+1}, \,I_{d+1}\right)\in  D_\la\right) &=(\la/2\pi)^{(d+1)/2}\int_{D}e^{-\la z(u)}\dd u \\
&= (\la/2\pi)^{(d+1)/2}\frac{(2\pi/\la)^{d/2}e^{-\la z(0)}}{\la(\det H)^{1/2}}\left(1+\sum_{m=1}^{N-1}a_m\left(\frac{d^2}{\la}\right)^{m}+ \Rem_N\right)\\
&=\frac{e^{-\la/2}}{\sqrt{2\pi\la\det H}}\left(1+\sum_{m=1}^{N-1}a_m\left(\frac{d^2}{\la}\right)^{m}+ \Rem_N\right),
\es 
with
$$\max_{m=0,\dots,N-1}|a_m|\leq C(C_{\z,0},\dots, C_{\z,2N+2}, N),\qquad |\Rem_L|\leq C(C_{\z,0},\dots, C_{\z,2N+2}, N, \CR)(d^2/\la)^{N}.$$
In turn, the $C_{\z,\ell}$'s depend on $C_2,\dots,C_{2N+2}$ and $\CR$ depends on $C_0$ and $N$, as we have shown above. The proof of~\eqref{a1gengauss} is given in Appendix~\ref{app:term2}. This concludes the proof.
\end{proof}
 \begin{lemma}\label{lma:omz}
Let $z(x,t)=\|x\|^2/2+(t+1)^2/2$ and $\z(x,y)=z(x,y+\psi(x))$. Under the conditions~\eqref{rho1-gauss3},~\eqref{eq:assume:gauss2}, we have
\bs
&R\e(\omega_{1,1}+\omega_{0,2}(R\e,R\e)) + (R\e)^2\omega_{2,1}(R\e,0)\leq1/2,\\
&R\e\omega_{3,0}(R\e,0) \leq 3,\quad \omega_{3,0}\leq\delta_3,\\
&\omega_{3,0}^2+\omega_{4,0}(R\e,0)\les\delta_4(R\e)+ \delta_2^2 +\delta_3^2.
\es
\end{lemma}
\begin{proof}Using~\eqref{gen-psi-derivs}, we have $\omega_{1,1}=0$, $\omega_{0,2}(R\e,R\e)=1$, $\omega_{2,1}(R\e,0)=\delta_2(R\e)$. Therefore, using~\eqref{rho1-gauss3}, we have
$R\e(\omega_{1,1}+\omega_{0,2}(R\e,R\e)) + (R\e)^2\omega_{2,1}(R\e,0)=R\e(1+R\e\delta_2(R\e))\leq \frac12$. Next, using~\eqref{del12tay}, we have $\delta_0(R\e) \leq (R\e)^2\delta_2(R\e)/2 \leq 1/12$ and $\delta_1(R\e)\leq R\e\delta_2(R\e)$. Thus by~\eqref{w-deltas-gauss} and~\eqref{rho1-gauss3} we have
\be
R\e\omega_{3,0}(R\e,0) \leq R\e\delta_3(R\e)(1+\delta_0(R\e)) + 3(R\e\delta_2(R\e))^2 \leq \frac{13}{12}+ \frac 34 \leq 3.
\ee
Next, using~\eqref{w-deltas-gauss} and noting $\delta_1(0)=\|\nabla\psi(0)\|_H=0$ and $\delta_0(0)=\psi(0)=0$, we have
$$\omega_{3,0}\leq \delta_3(0)(1+\delta_0(0))+3(\delta_1\delta_2)(0) = \delta_3(0)=\delta_3.$$ Using~\eqref{w-deltas-gauss} with $\ell=4$ and that $\delta_0(R\e)\leq 1/12$, we have
\bs
\omega_{4,0}(R\e,0)+\omega_{3,0}^2 &\les\delta_4(R\e) + (\delta_1\delta_3)(R\e) + \delta_2(R\e)^2+\delta_3^2\\
&\les \delta_4(R\e) + R\e\delta_2(R\e)\delta_3(R\e) + \delta_2(R\e)^2+\delta_3^2\\
&\les \delta_4(R\e) + \delta_2(R\e)^2+\delta_3(R\e)^2 \les \delta_4(R\e)+\delta_2^2+\delta_3^2.
\es To get the second line we used~\eqref{del12tay}. To get the last line we used that $R\e\leq1$, and a Taylor expansion to bound $\delta_2(R\e),\delta_3(R\e)$.
 \end{proof}

\begin{proof}[Proof of Proposition~\ref{prop:gauss:3}]We check the conditions of Proposition~\ref{L2:cvx}, with $M=1$, $z(x,t)=\frac12\|x\|^2+\frac12(t+1)^2$. We showed in Section~\ref{app:gauss:2} that Assumption~\ref{assume1} is satisfied, and we supposed in Assumption~\ref{assume:D} that $D$ is star-shaped. The function $z$ is obviously convex. 

The conditions~\eqref{D} are satisfied by~\eqref{rho1-gauss3}. Lemma~\ref{lma:omz} and~\eqref{eq:assume:gauss2} show the condition $R\e\omega_{3,0}(R\e,0)\leq3$ and~\eqref{Rwc34eps} are satisfied, with $\CL=C(c_\delta)$. We conclude that
\be
\mathbb P\left(\mathcal N\left(0_{d+1}, \,I_{d+1}\right)\in  D_\la\right) =(\la/2\pi)^{(d+1)/2}\int_{D}e^{-\la z(u)}\dd u = \frac{e^{-\la/2}(1+\Rem_1)}{\sqrt{2\pi\la\det H}}.
\ee
where $\Rem_1$ is bounded as in~\eqref{bd-Rem2-g1}. Thus it remains to simplify the bound~\eqref{bd-Rem2-g1} on $\Rem_1$. Using Lemma~\ref{lma:omz} and the derivative calculations in~\eqref{gen-psi-derivs}, this remainder bound becomes
\bs
|\Rem_1| \les_{R,c_\delta}& \,\left(\delta_4(R\e) + \delta_3^2+\delta_2^2\right)\frac{d^2}{\la} + \delta_2(R\e)\frac{d}{\la}+\frac{1}{\la}\\
\les_{R,c_\delta} &\,\left(\delta_4(R\e) +\delta_3^2+ \delta_2^2 \right)\frac{d^2}{\la} + \left(\delta_2+\e\delta_3+(d/\la)\delta_4(R\e)\right)\frac{d}{\la}+\frac{1}{\la}.
\es To get the second line, we used that $\delta_2(R\e) \leq\delta_2+R\e\delta_3+(R\e)^2\delta_4(R\e)\les_R \delta_2+\e\delta_3+\e^2\delta_4(R\e)$, by a Taylor expansion. Finally, we collect like terms. We have
\bs\label{d24}
\frac{d^2}{\la}\delta_4(R\e) + \frac{d^2}{\la^2}\delta_4(R\e)&\les \frac{d^2}{\la}\delta_4(R\e),\\
\frac{1}{\la}\left(1+d(\delta_2+\e\delta_3) + d^2(\delta_2^2+\delta_3^2)\right) &\les\frac{1}{\la}+\frac{d^2}{\la}(\delta_2^2+\delta_3^2).
\es Summing the terms on the righthand sides of~\eqref{d24} gives
\be
|\Rem_1| \les_{R,c_\delta} \left(\delta_4(R\e) + \delta_3^2 + \delta_2^2\right)\frac{d^2}\la + \frac1\la.
\ee
Finally, since $R=24+6(\log\la)/d$, the dependence on $R$ can be written as a dependence on $(\log\la)/d$. This concludes the proof.
\end{proof}

\noindent\textbf{Details from Example~\ref{ex:quart}.} Recall that $\psi(x)=\frac{1}{4!}\langle S, x^{\otimes4}\rangle$ and $H=I_d+\nabla^2\psi(0_d)=I_d$. Note that $\nabla^2\psi(x)=\frac12\langle S, x^{\otimes2}\rangle$ and $\nabla^3\psi(x)=\langle S, x\rangle$. Thus 
\be\label{delquart}\delta_2(\rho)=\|S\|\rho^2/2,\quad\delta_3(\rho)=\|S\|\rho,\quad\delta_4(\cdot)\equiv\|S\|,\ee and $\delta_\ell(\cdot)\equiv0$ for all $\ell\geq5$. Suppose that Assumption~\ref{a14} is satisfied for some $\rho_0\leq\frac13(1\vee\|S\|)^{-1}$, in which case $\delta_2(\rho_0)\rho_0^2/2<2$ from Assumption~\ref{assume:D} is also satisfied. Next, suppose $R\e \leq  \rho_0$ and $d^2/\la\leq C_0$. Then, using the bound on $\rho_0$ and~\eqref{delquart}, we see that~\eqref{deltas:gauss} is satisfied with $C_2=\|S\|/18$, $C_3=\|S\|/3$, $C_4=\|S\|$, and $C_\ell=0$ for all $\ell\geq5$. Then $R\e\leq\rho_0$ implies~\eqref{cond:gauss} is satisfied, since $\rho_0\leq \min(1/3, 1/\|S\|, \rho_0) \leq \min(1/3, 2/(4C_2+C_3), \rho_0)$.

 Finally, note that~\eqref{a1gengauss} gives $d^2a_1=-1-\frac18\langle S, I_d\otimes I_d\rangle$. We conclude by Proposition~\ref{prop:gauss:1} that under~\eqref{Re-quart}, the expansion~\eqref{P-quart} and remainder bounds~\eqref{amquart} hold. 

\begin{proof}[Proof of Proposition~\ref{prop:lb}]Recall from Example~\ref{ex:quad2} that $\psi(x)=\|x\|^2/2$ and the boundary parameterization holds globally, so that $D=\{(x,t)\,:t\geq\|x\|^2/2\}$. Also, recall that $\delta_2(\cdot)\equiv\delta_2=1/2$. We check the conditions of Proposition~\ref{prop:gauss:1} with $N=2$. Thus we define $R=24+2(1+2N)(\log\la)/d=24+10(\log\la)/d$. It is clear that $D$ satisfies Assumption~\ref{assume:D}; in particular, since $\delta_2(\cdot)\equiv1/2$, we can take $\rho_0=2$ to satisfy $\delta_2(\rho_0)\rho_0^2<2$. Also, note that $\delta_\ell\equiv0$ for all $\ell\geq2$. Thus~\eqref{deltas:gauss} holds with $C_2=1/2$ and $C_\ell=0$ for all $\ell\geq3$. The bounds~\eqref{cond:gauss} then hold, with $C_0=1$, since we assumed $d^2\leq\la$ and $R\e\leq 1/3$, and we have $\min(1/3, 2/(4C_2+C_3), \rho_0) = 1/3$. Thus Proposition~\ref{prop:gauss:1} with $N=2$ applies, and we obtain
\bs\label{expansion-gauss-tight}
\mathbb P(Z/\sqrt{\la}\in D_0) = \frac{e^{-\la/2}}{\sqrt{2\pi\la\det(2I_d)}}(1+\Rem_1) = \frac{e^{-\la/2}2^{-d/2}}{\sqrt{2\pi\la}}(1+\Rem_1),
\es 
where
\bs\label{R1}
\Rem_1& =\frac{d^2}{\la}a_1 + \Rem_2, \\
|a_1| &\leq C, \quad |\Rem_2|\leq C\left(\frac{d^2}{\la}\right)^2.
\es for an absolute constant $C$ (since $N$, $\CR$, and the $C_\ell$'s are all absolute).~\eqref{R1} implies
\be\label{Cda1}
\frac{d^2}{\la}|a_1| - C\left(\frac{d^2}{\la}\right)^2 \leq |\Rem_1| \leq C\frac{d^2}{\la}.
\ee Now, we apply~\eqref{a1quad}, recalling that $b_i$ are the eigenvalues of $B$, where $\psi(x)=\frac12x^\top Bx$. In our case $B=I_d$, so $b_i=1$ for all $i$.  Thus~\eqref{a1quad} gives
\be
-d^2a_1=1 + \frac12\cdot\frac{d}{2} +\frac18(d/2)^2 + \frac 14\cdot\frac d4 = \frac{d^2}{32} + \frac{5}{16}d + 1,
\ee  and therefore $|a_1|\geq 1/32$. The proof follows from this,~\eqref{Cda1}, and the fact that $d^2/\la$ is assumed small enough.
\end{proof}
\section{Omitted proofs from Appendix~\ref{app:outline:TV} and Section~\ref{sec:sample} }\label{app:sec:sample}
Recall the notation $\us:=\us(R\e), I:= [0,(R\e)^2]$ from Appendix~\ref{app:outline:TV}, and that we assume without loss of generality that $z(0_{d+1})=0$. Recall from Lemma~\ref{lma:mu} that the conditions of Proposition~\ref{prop:explicitL2} and~\ref{L2:cvx}, respectively, imply the conditions of Corollary~\ref{corr:genprelim}. In turn, the conditions of the corollary imply the conditions of Lemmas~\ref{lma:D123}-\ref{lma:D3}. Thus we are justified in applying these lemmas, which we do below without further comment. 
\begin{proof}[Proof of Lemma~\ref{lma:AB}]
Note that $\pi\vert_D(u) = a(u)\ind_D(u)/A$ and $\hat\pi(u)=b(u)\ind_{\hat D}(u)/B$. Now, we observe that
\bs\label{sum-pf}
\mathrm{TV}(\pi\vert_D, \hat\pi) =&\frac12\int_{\br^{d+1}}|\pi\vert_D(u)-\hat\pi(u)|\dd u \\
 \leq&\int_{\br^{d+1}}|a(u)\ind_D(u)/A-b(u)\ind_{\hat D}(u)/B|\dd u\\
\leq &|A^{-1}-B^{-1}|\int_{\br^{d+1}}a(u)\ind_D(u)\dd u + B^{-1}\int_{\br^{d+1}} |a(u)\ind_D(u)-b(u)\ind_{\hat D}(u)|\dd u \\
= &|1-A/B| + \frac1B\int_{\br^{d+1}} |a(u)\ind_D(u)-b(u)\ind_{\hat D}(u)|\dd u\\
\leq &|1-A/B| + \frac1B\int_{D_2\cup D_3}a(u)\dd u + \frac1B\int_{\hat D\setminus D_1}b(u)\dd u \\
&+\frac1B\int_{D_1\setminus \hat D}a(u)\dd u + \frac1B\int_{D_1}|a(u)-b(u)|\dd u,
\es as desired. Here, we used that $D\setminus D_1\subset D_2\cup D_3$, by Lemma~\ref{lma:D123}.
\end{proof}
\begin{proof}[Proof of Lemma~\ref{lma:tailAB}]
Note that $B=\mathcal Z/\la$, recalling from~\eqref{Z-rho-def} the definition of $\mathcal Z$. Thus using Lemmas~\ref{lma:D2},~\ref{lma:D3}, we have
\bs
 \frac1B\int_{D_2\cup D_3}a(u)du&= \frac{\la}{\mathcal Z}\int_{D_2\cup D_3}e^{-\la z(u)}\dd u \\
 &\leq \mu e^{-Rd/2} + d\sqrt{\la}e^{d-R d/4} \les (\mu\vee1)d\sqrt{\la}e^{d-R d/4}.
\es
Using that $R\geq 4(2+(M+1/2)(\log\la)/d)$, we have $Rd/4 -d-\log d \geq (M+1/2)\log\la$, and therefore $de^{d-Rd/4}\leq \la^{-M-1/2}$. Also, we have from Lemma~\ref{lma:mu} that $\mu\leq C(R, \CL)$. We conclude~\eqref{agbd}. Next, note that
\bs\label{hatD}
\hat D\setminus D_1 \subseteq \{(x,t)\in&\br^{d+1}\,:\, x\in\us,\, t-\psi(x)\geq(R\e)^2,\, t\geq\hat\psi(x)\}\\
&\;\cup \;\{(x,t)\in\br^{d+1}\,:\, x\notin\us, t\geq\hat\psi(x)\}
\es
Furthermore, using the assumption that $R\e\delta_3(R\e)\leq3$, we have
\be\label{supus}\sup_{x\in\us}|\psi(x)-\hat\psi(x)|\leq \frac16(R\e)^3\delta_3(R\e)\leq (R\e)^2/2.\ee Therefore, the first set on the righthand side of~\eqref{hatD} is contained within the set $\{(x,t)\in\br^{d+1}\,:\, t-\hat\psi(x)\geq(R\e)^2/2\}$. Thus
\bs\label{1BE}
 \frac1B\int_{\hat D\setminus D_1}b(u)\dd u \leq  &\frac{\sqrt{\det H}}{(2\pi/\la)^{d/2}}\int_{\br^d}e^{-\frac{\la}{2}x^\top Hx}\dd x\int_{t\geq\hat\psi(x)+(R\e)^2/2}\la e^{-\la(t-\hat\psi(x))}\dd t\\
 &+\frac{\sqrt{\det H}}{(2\pi/\la)^{d/2}}\int_{\us^c}e^{-\frac{\la}{2}x^\top Hx}\dd x\int_{t\geq\hat\psi(x)}\la e^{-\la(t-\hat\psi(x))}\dd t\\
 = & e^{-R^2d/2} +  \mathbb P((\la H)^{-1/2}Z \notin \us)\\
  \leq& e^{-R^2d/2}+e^{-(R-1)^2d/2} \les \la^{-M}.
\es To get the second-to-last bound, we used the Gaussian tail inequality~\eqref{gausstail}. To get the last bound we used $(R-1)^2d/2\geq Rd\geq M\log\la$. Thus,~\eqref{agbd2} is satisfied.
\end{proof}
\begin{remark}\label{rk:bB}
Note that $ \frac1B\int_{\hat D\setminus D_1}b(u)\dd u =\hat\pi(\hat D\setminus D_1)$, and furthermore, $\hat\pi(\hat D)=1$. Thus~\eqref{1BE} implies $\hat\pi(\hat D\cap D_1)\geq 1-2e^{-(R-1)^2d/2}$. Finally, recall from Lemma~\ref{lma:D123} that $D_1\subset D$. Then $\hat D\cap D_1\subseteq D$ as well, so $\hat\pi(D)\geq 1-2e^{-(R-1)^2d/2}$, as claimed in Remark~\ref{rk:conc}.
\end{remark}
For the proof of Lemma~\ref{lma:BD}, we first prove the following auxiliary result.
\begin{lemma}\label{lma:t12}
Let $X$ be a random variable. Then
\be
\E[|e^{f(X)}-e^{g(X)}|\ind_{U}(X)] \leq \max_{t_1,t_2\in[0,1]} \E[|f(X)-g(X)|\exp((t_1f+t_2g)(X))\ind_U(X)].
\ee
\end{lemma}
\begin{proof}We have
$$e^f-e^g =e^g \int_0^1(f-g)e^{t(f-g)}\dd t = \int_0^1(f-g)e^{tf+(1-t)g}\dd t$$ and similarly, $e^g-e^f = \int_0^1(g-f)e^{tf+(1-t)g}\dd t$. Thus $|e^f-e^g| =|f-g|\int_0^1e^{tf+(1-t)g}\dd t$. Multiplying by $\ind_U$ and taking the expectation on both sides gives
\bs
\E[|e^{f(X)}-e^{g(X)}|\ind_{U}(X)]&=\E\left[|f(X)-g(X)|\int_0^1e^{(tf+(1-t)g)(X)}\dd t \,\ind_{U}(X)\right] \\
&=\int_0^1\E\left[|f(X)-g(X)|e^{(tf+(1-t)g)(X)}\ind_{U}(X)\right]\dd t\\
&\leq \max_{t_1,t_2\in[0,1]} \E[|f(X)-g(X)|\exp((t_1f+t_2g)(X))\ind_U(X)],
\es as desired.
\end{proof}

\begin{proof}[Proof of Lemma~\ref{lma:BD}]
 Note that $D_1\setminus \hat D_1 = \{(x,t)\in\br^{d+1}\,:\,x\in\us, \psi(x)\leq t \leq \hat\psi(x)\}$. Thus
 \bs\label{1BD}
\frac1B\int_{D_1\setminus \hat D_1}a(u)\dd u &= \frac{\la}{\mathcal Z}\int_{\substack{x\in\us\,:\\ \Delta(x)<0}}\int_{t-\psi(x)\in[0,-\Delta(x)]}e^{-\la z(x,t)}\dd t\dd x\\
&= \frac{\la}{\mathcal Z}\int_{\substack{x\in\us\,:\\ \Delta(x)<0}}\int_{0}^{-\Delta(x)}e^{-\la \z(x,y)}\dd y\dd x\\
&\leq\frac{\la}{\mathcal Z}\int_{\us}\int_{0}^{|\Delta(x)|}e^{-\la \z(x,y)}\dd y\dd x.
\es Next,~\eqref{supus} gives that $|\Delta(x)|\leq (R\e)^2/2$ for all $x\in \us$, and we know $\pa_y\z(x,y)\geq1/2$ for all $x\in\us, 0\leq y\leq (R\e)^2$. Indeed, this is one of the conditions of Lemma~\ref{lma:D1}, which as discussed above, are implied by the conditions of Proposition~\ref{prop:explicitL2} and~\ref{L2:cvx}. Thus, recalling that $q(x) = \z(x,0)-\frac12x^\top Hx$ and $X\sim\mathcal N(0_d, (\la H)^{-1})$, we have
\bs\label{1BD2}
\frac{\la}{\mathcal Z}\int_{\us}\int_{0}^{|\Delta(x)|}&e^{-\la \z(x,y)}\dd y\dd x\leq \frac{1}{\mathcal Z}\int_{\us}e^{-\la\z(x,0)}\int_{0}^{|\Delta(x)|}\la e^{-\la y/2}\dd y\dd x\\
&=2\frac{\sqrt{\det H}}{(2\pi/\la)^{d/2}}\int_{\us}\left(1-e^{-\la|\Delta(x)|/2}\right)e^{-\la\z(x,0)}\dd x \\
&=2\E\left[\left(1-e^{-\la|\Delta(X)|/2}\right)e^{-\la q(X)}\ind_{\us}(X)\right]\\
&\leq \max_{|t_1|,|t_2|\leq1}\E\left[\la|\Delta|e^{\la(t_1\Delta+t_2q)}\ind_{\us}\right],
\es using Lemma~\ref{lma:t12} to get the last line. Next, let $\beta(x,y)=r(x,y)+q(x)= \z(x,y) - x^\top Hx/2 -  y$, and note that $\frac{\la\sqrt{\det H}}{(2\pi/\la)^{d/2}}\exp(-\la x^\top Hx/2 - \la y)$ is the joint density of $(X,Y)$. Therefore,
\bs\label{1BD3}
 \frac1B\int_{D_1}|a(u)-b(u)|\dd u&=\frac1B\int_{\us}\int_{I} |a(x,\psi(x)+y)-b(x,\psi(x)+y)|\dd y \dd x\\
 &= \frac{\la\sqrt{\det H}}{(2\pi/\la)^{d/2}}\int_{\us}\int_{I}\left|e^{-\la\z(x,y)} - e^{-\frac{\la}{2}x^\top Hx - \la(y+\Delta(x))}\right|\dd y\dd x\\
 &=\E\left[\left|e^{-\la\beta(X,Y)} - e^{-\la\Delta(X)}\right|\ind_{\us}(X)\ind_{I}(Y)\right]\\
 &\leq \max_{|t_1|,|t_2|\leq 1}\E\left[\la(|\beta| +|\Delta|)e^{\la(t_1\Delta +t_2\beta)}\ind_{\us\times I}\right]\\
 &\leq \max_{|t_1|,|t_2|\leq 1}\E\left[\la(|q| +|r|+|\Delta|)e^{\la(t_1\Delta +t_2q + |r|)}\ind_{\us\times I}\right]
\es To get the last two lines we used Lemma~\ref{lma:t12} followed by the fact that $\beta(X,Y)=q(X)+r(X,Y)$. Combining~\eqref{1BD},~\eqref{1BD2}, and~\eqref{1BD3} gives the desired result.
\end{proof}
\begin{proof}[Proof of Lemma~\ref{lma:Er2}]
Let $v=\nabla_x\pa_y\z(0_{d+1})$. Note that for $\bar Y\sim \mathrm{Exp}(\la/2)$ and $X\sim\mathcal N(0_d, (\la H)^{-1})$, we have 
\be\label{EYX}\begin{gathered}
\E[\bar Y^2]\les\la^{-2},\quad\E[\bar Y^4]\les\la^{-4},\\
\E[(X^\top v)^2]=\la^{-1}\|H^{-1/2}v\|^2=\la^{-1}\omega_{1,1}^2,\\
\E[\|H^{1/2}X\|^4]\leq 3d^2/\la^2.
\end{gathered}\ee Thus using the bound in Lemma~\ref{lma:rxy}, we have
\bs
\E[ r&(X,\bar Y)^2\ind_{\us\times I}(X,\bar Y)]^{1/2} \\
&\les \E\left[(X^\top v)^2\bar Y^2\right]^{1/2} +\omega_{2,1}(R\e,0)\E\left[\|H^{1/2}X\|^4\bar Y^2\right]^{1/2} +\omega_{0,2}(R\e, (R\e)^2)\E\left[\bar Y^4\right]^{1/2}\\
&\les \la^{-1/2}\omega_{1,1}\times\la^{-1}+ \omega_{2,1}(R\e,0)\frac d\la \times \la^{-1} + \omega_{0,2}(R\e, (R\e)^2)\la^{-2}.
\es Multiplying through by $\la$ concludes the proof.
\end{proof}

\begin{proof}[Proof of Corollary~\ref{corr:gTV}]Let $Z\sim \pi_\la$ and $\bar Z\sim \hat\pi_\la$. Let $U=\la^{-1/2}Z - (0_d,1)$ and $\bar U=\la^{-1/2}\bar Z - (0_d,1)$. Then by the invariance of TV distance to bijective coordinate changes, we have
\be
\mathrm{TV}(\pi_\la,\hat\pi_\la) = \mathrm{TV}(\mathrm{Law}(Z),\mathrm{Law}(\bar Z)) = \mathrm{TV}(\mathrm{Law}(U),\mathrm{Law}(\bar U)).
\ee Let $\pi\vert_D$ be the density of $U$ and $\hat\pi$ be the density of $\bar U$. Since $\pi_\la$ is the restriction of $\mathcal N(0_{d+1},I_{d+1})$ to $D_\la=\sqrt\la D+(0_d,\sqrt\la)$, we have that $\pi\vert_D$ is the restriction of $\mathcal N((0_d,-1), \la^{-1} I_{d+1})$ to $D$. Thus $\pi\vert_D(u)\propto e^{-\la z(u)}\ind_D(u)$, where $z(x,t)=\|x\|^2/2+(t+1)^2/2$. Next, let $\pi_H,\hat\Theta$ be as in~\eqref{pi-H} and note that $\hat\Theta=\la^{-1/2}\hat\Theta_\la-(0_d,1)$. Since the law of $U$ is $\hat\pi_\la=\hat\Theta_\la\#\pi_H$, we conclude the law of $\bar U$ is $\hat\pi = \hat\Theta\#\pi_H$. Thus $\mathrm{TV}(\pi_\la,\hat\pi_\la)=\mathrm{TV}(\pi\vert_D,\hat\pi)$, where $\pi\vert_D,\hat\pi$ are as in Proposition~\ref{prop:TV}, with $z(x,t)=\|x\|^2/2+(1+t)^2/2$. Now, we show in the proof of Proposition~\ref{prop:gauss:3} that the conditions of that proposition imply the conditions of Proposition~\ref{L2:cvx}, with $M=1$ and $\CL=C(c_\delta)$. Thus, we may apply Proposition~\ref{prop:TV}. We have from~\eqref{gen-psi-derivs} and Lemma~\ref{lma:omz} that
\be\label{omegam}\begin{gathered}
\omega_{1,1}=0,\quad\omega_{0,2}(R\e,(R\e)^2)=1,\quad\omega_{2,1}(R\e,0)=\delta_2(R\e),\quad\omega_{3,0} \leq\delta_3,\\
\omega_{4,0}(R\e,0)\leq \omega_{3,0}^2+\omega_{4,0}(R\e,0) \les_{R,c_{\psi}} \delta_4(R\e)+ \delta_2^2 +\delta_3^2.\end{gathered}\ee
Substituting~\eqref{omegam} into~\eqref{eq:TV} 
gives 
\bs
\mathrm{TV}(\pi\vert_D,\hat\pi) &\les_{R,c_{\psi}} \delta_3\frac{d}{\sqrt\la} +\delta_2(R\e)\frac d\la +\frac1\la + \left( \delta_4(R\e)+ \delta_2^2 +\delta_3^2\right)\frac{d^2}{\la}\\\
&\les_{R,c_{\psi}} \delta_3\frac{d}{\sqrt\la}+\delta_4(R\e)\frac{d^2}{\la} +\frac1\la\left(1+d\delta_2(R\e)+d^2\delta_2^2\right).
\es
Next, we have by a Taylor expansion and the fact that $R\e\leq1$ that
$$\delta_2(R\e)\frac d\la \leq (\delta_2+\delta_3+\delta_4(R\e))\frac d\la \leq \delta_2\frac d\la + \delta_3\frac{d}{\sqrt\la}+ \delta_4(R\e)\frac{d^2}{\la}.$$ Substituting this above gives
\bs
\mathrm{TV}(\pi\vert_D,\hat\pi) &\les_{R,c_{\psi}}\delta_3\frac{d}{\sqrt\la}+\delta_4(R\e)\frac{d^2}{\la} +\frac1\la\left(1+d\delta_2+d^2\delta_2^2\right)\\
&\les_{R,c_{\psi}}\delta_3\frac{d}{\sqrt\la}+\delta_4(R\e)\frac{d^2}{\la} +\frac1\la\left(1+d^2\delta_2^2\right).
\es
Finally, since $R=24+6(\log\la)/d$, we can write the suppressed constant in terms of $c_\delta$ and $(\log\la)/d$. This gives the desired bound.
\end{proof}

\section{Derivative bounds} \label{app:deriv}
We prove Lemma~\ref{lma:deriv-change}, which shows that the bounds~\eqref{forq},~\eqref{forw} on $\g$ and $\z$ imply the bounds needed to apply the high-dimensional Laplace expansion of $\int\g_ke^{-\la\z}\dd x$ used in the proof of Proposition~\ref{prop:main}. In this section only, we use $a\les b$ to denote $a\leq Cb$ for some constant $C$ that may depend on $C_{\z,\ell}$, $\ell=0,\dots,2N+2$ and on $C_{\g,\ell}$, $\ell=0,\dots,2N$ from~\eqref{forq} and~\eqref{forw}. We rewrite Lemma~\ref{lma:deriv-change} here for convenience, as a proposition, using the new meaning of $\les$.

\begin{proposition}\label{prop:deriv-change}The derivative bounds~\eqref{forq},~\eqref{forw} of Proposition~\ref{prop:main} imply 
\begin{alignat}{2}
\sup_{x\in\us(R\e)}\|\nabla_x^\ell\z(x,0)\|_H&\les d^{\lceil\ell/2\rceil-2}, &\quad&\ell=3,\dots,2N+2,\label{w0bd}\\
\sup_{(x,y)\in \us(R\e)\times[0,(R\e)^2]}|D_{\z}^{k-1}\g(x,y)|&\les 1,& &k=1,\dots,N+1,\label{qk-0-cond-1}\\
\sup_{x\in\us(R\e)}\|\nabla_x^\ell\g_k(x)\|_H&\les d^{\lceil\ell/2\rceil},& &\ell=0,\dots,2(N+1-k),\;k=1,\dots,N.\label{qk-0-cond-2}
\end{alignat} 
\end{proposition}
Here, recall that
\be
\g_k(x) = \frac{(D_{\z}^{k-1}\g)(x,0)}{\pa_y\z(x,0)},\quad (D_{\z}\g)(x,y) =\pa_y\left(\frac{\g(x,y)}{\pa_y\z(x,y)}\right).
\ee
The bounds~\eqref{w0bd} are immediate by taking $\ell=3,\dots,2N+2$ and $j=j_\ell =0$ in~\eqref{forw}. To prove~\eqref{qk-0-cond-1} and~\eqref{qk-0-cond-2}, we start with a lemma about the form of $D_{\z}^{k-1}\g$.
\begin{lemma}\label{lma:Dwk}For all $k=1,2,3,\dots$, the function $D_w^{k-1}\g$ is a sum of products of terms of the form $\pa_y^j\g \prod_{j=1}^pw_j$, where $j\in\{0,1,\dots,k-1\}$ and each of the $w_j$ is either $(\pa_yw)^{-1}$ or $\pa_y^{\ell}w$ for $\ell\in\{2,\dots,k\}$. The function $\g_k = (D_w^{k-1}\g)/\pa_yw$ is therefore of the same form.
\end{lemma}
\begin{proof}We proceed by induction. For $k=1$ we have the function $\g$, which is of the required form. Now suppose $h:=D_w^{k-1}\g$ has the required form. Then $D_w^k\g = D_wh = \pa_y(h/\pa_y w) = \pa_yh/\pa_yw -h\pa_y^2w/(\pa_yw)^2$. To conclude, it suffices to show $\pa_yh$ is of the desired form, and this is clear upon differentiating $\pa_y^j\g \prod_{j=1}^pw_j$.
\end{proof}
We can now prove~\eqref{qk-0-cond-1}.
\begin{proof}[Proof of~\eqref{qk-0-cond-1}]
Using the form of $D_w^{k-1}\g$ given in Lemma~\ref{lma:Dwk}, we see that the functions $D_w^{k-1}\g$, $k=1,\dots,N+1$ are products of subsets of the following functions:
 $(\pa_yw)^{-1}$, $\pa_y^j\g$, $j=0,\dots,N$, and $\pa_y^j \z$, $j=2,\dots, N+1$. The assumption that $R\e(C_{\z,0}+C_{\z,1})\leq1/2$ from~\eqref{BC} of Proposition~\ref{prop:main} implies $(\pa_y\z)^{-1}\leq 2$ for all $(x,y)\in \us(R\e)\times[0,R\e]$; this is shown in~\eqref{wlb}. Also, from~\eqref{forq} and~\eqref{forw}, we have that $\pa_y^j\g$, $j=0,\dots,N$, and $\pa_y^j w$, $j=2,\dots, N+1$ are bounded in $ \us(R\e)\times[0,(R\e)^2]$. Thus products of such terms are also bounded. 
\end{proof}
To prove~\eqref{qk-0-cond-2}, we introduce a new convention.
\begin{definition}We say $f\in\mathcal D_M^{\z}$ if
\be\label{fDM}
\sup_{x\in\us(R\e)}\|\nabla^\ell f(x)\|_H\les {d}^{(\lceil \ell/2\rceil-2)_+},\quad\forall \ell=0,1,\dots,M.
\ee We say $f\in\mathcal D_M^{\g}$ if
\be\label{fDMq}
\sup_{x\in\us(R\e)}\|\nabla^\ell f(x)\|_H\les {d}^{\lceil \ell/2\rceil},\quad\forall \ell=0,1,\dots,M.
\ee
\end{definition}
Note that the bounds~\eqref{qk-0-cond-2} we need to prove are the statement that $\g_k\in\mathcal D_{2N+2-2k}^{\g}$ for all $k=1,\dots,N$.
\begin{lemma}Under~\eqref{forq} and~\eqref{forw} we have $\pa_y^j\g\in\mathcal D_{2N-2j}^{\g}$ for all $j=0,\dots, N-1$ and $\pa_y^j w\in \mathcal D_{2N+2-2j}^{\z}$ for all $j=1,\dots, N$.
\end{lemma}
\begin{proof}
We need to show
\begin{align}
\sup_{x\in\us(R\e)}\|\nabla_x^\ell\pa_y^j\g(x,0)\|_H&\les d^{\lceil\ell/2\rceil},\quad\ell=0,\dots,2N-2j,\; j=0,\dots,N-1,\label{forqk-q-2}\\
\sup_{x\in\us(R\e)}\|\nabla_x^\ell\pa_y^j\z(x,0)\|_H&\les d^{(\lceil\ell/2\rceil-2)_+},\quad\ell=0,\dots,2N+2-2j,\; j=1,\dots,N\label{forqk-w-2}
\end{align} 
To see why~\eqref{forqk-q-2} is true, note that the cases $\ell=0$, $j=0,\dots,N-1$ in~\eqref{forqk-q-2} follows from the first line of~\eqref{forq}, while the remaining cases in~\eqref{forqk-q-2} follow from the second line of~\eqref{forq}. This is because the set of integer pairs $(j,\ell)$ such that $1\leq\ell\leq2N-2j$ and $0\leq j\leq N-1$ is the same as the set of integer pairs $(j,\ell)$ such that $0\leq j\leq \lfloor N-\ell/2\rfloor$ and $1\leq\ell\leq 2N$.

To check~\eqref{forqk-w-2}, we note that the cases $\ell=0$, $j=2,\dots,N$ follow from the first line of~\eqref{forw}. The case $\ell=0,j=1$ follows by noting that since $\pa_y\z(0_{d+1})=1$, we have
\bs
\sup_{x\in\us(R\e)}|\pa_y\z(x,0)|\leq 1+\sup_{x\in\us(R\e)}|\pa_y\z(x,0)-\pa_y\z(0_{d+1})|\leq 1+R\e\omega_{1,1}(R\e,0) \leq 1+C_{\z,1}.
\es
The remaining cases in~\eqref{forqk-w-2} can be equivalently described as all $\ell\in\{1,\dots,2N+2-2\}$ and $j\in\{1,\dots,\lfloor N+1-\ell/2\rfloor\}$. These cases are covered by the second line of~\eqref{forw}.
\end{proof}

\begin{lemma}\label{lma-w-inv}Let $f(x)>1/2$ for all $x\in\us(\rho)$. If $f\in\mathcal D_M^{\z}$ then $1/f\in \mathcal D_M^{\z}$.
\end{lemma} 
\begin{proof}Fix $\ell\in\{0,1,\dots,M\}$. Note that $\nabla ^\ell (f^{-1})$ is given by a linear combination of terms of the form
\bs
f^{-k}\bigotimes_j\nabla ^{k_j}f,
\es where $\sum_jk_j=\ell$, and each of the $k_j\leq M$. Thus the derivative bounds~\eqref{fDM} apply for each $k=k_j$. Since $\pa_y\z$ is bounded from below by $1/2$, we conclude that
\be\label{grad-winv}
\|\nabla ^\ell(1/f)\|_H \les \sum_{k_j\geq0,\sum_jk_j=\ell}\prod_j\|\nabla ^{k_j}f\|_H.
\ee
If $\ell\leq2$ then so are the $k_j$, and using that $\|\nabla ^{k}f\|_H\leq C$ for $k=0,1,2$, we obtain $\|\nabla ^\ell(1/f)\|_H\leq C$. Next consider $\ell\geq3$. We can absorb factors $\|\nabla ^{k_j}\pa_y\z\|_H$ with $k_j\leq2$ into the $\les$ sign, to get
\be\label{grad-winv-2}
\|\nabla ^\ell(1/f)\|_H \les \sum_{k_j\geq3,\sum_jk_j\leq \ell}\prod_j\|\nabla ^{k_j}f\|_H.
\ee
Now, consider any of the summands in the above sum, which corresponds to fixing a set of $k_j$'s. Let $O$ and $E$ be the number of them that are odd and even, respectively. Note that $O+E\geq1$. Also, let $L=\sum_jk_j$, so that $3\leq L\leq \ell$. Using the bounds~\eqref{fDM} we have
\be
\prod_j\|\nabla ^{k_j}f\|_H \les {\sqrt d}^{L- 3O-4E}.
\ee
For a given $L$, this upper bound is maximized by taking $O$ and $E$ as small as possible. If $L$ is even, the maximum is achieved either at $O=2,E=0$ or $O=0,E=1$, and plugging in the numbers, we see that $O=0,E=1$ gives a larger result, namely, ${\sqrt d}^{L- 4}$. If $L$ is odd, the maximum is achieved at $O=1,E=0$ and we get ${\sqrt d}^{L- 3}$. Finally, note that $L$ can range between $3$ and $\ell$. Thus is it clear that the maximum possible value achieved in each summand of~\eqref{grad-winv-2} is ${\sqrt d}^{\ell- 3}$ if $\ell$ is odd, and ${\sqrt d}^{\ell- 4}$ if $\ell$ is even. This proves $1/f$ satisfies the bounds~\eqref{fDM}.\end{proof}

\begin{lemma}\label{lma:Wderiv} Let $\z_1,\dots, \z_p\in \mathcal D_M^{\z}$. Then $W:=\z_1\dots\z_p\in \mathcal D_M^{\z}$.
\end{lemma}
\begin{proof}Note that
\bs
\|\nabla_x^\ell W\|_H \leq\sum_{\substack{k_1,\dots,k_p\geq0\\ k_1+\dots+k_p=\ell}}\prod_{j=1}^p\|\nabla_x^{k_j}\z_j\|_H.
\es 
We now proceed exactly as in the proof of Lemma~\ref{lma-w-inv}.
\end{proof}
\begin{proof}[Proof of~\eqref{qk-0-cond-2}]
It suffices to show $q_k\in\mathcal D_{2N+2-2k}^{\g}$ for each $k=1,\dots,N$. Fix $k\in\{1,\dots,N\}$. Using Lemma~\ref{lma:Dwk}, $\g_k$ is given by sums of the form $W\pa_y^j\g$, where $j=0,1,\dots,k-1$ and $W=\prod_{j=1}^pw_j$, where each of the $w_j$ is either $(\pa_yw)^{-1}$ or $\pa_y^{\ell}w$ for some $\ell\in\{2,\dots,k\}$. Thus it suffices to show that $W\pa_y^j\g \in\mathcal D_{2N+2-2k}^{\g}$ for all $j=0,\dots,k-1$.

Since $\pa_yw\in\mathcal D_{2N+2-2}^{\z}$, we have $(\pa_yw)^{-1}\in\mathcal D_{2N+2-2}^{\z}\subseteq\mathcal D_{2N+2-2k}^{\z}$, by Lemma~\ref{lma-w-inv}. Also, for each $\ell\in\{2,\dots,k\}$ we have $\pa_y^\ell w\in\mathcal D_{2N+2-2\ell}^{\z}\subseteq\mathcal D^{\z}_{2N+2-2k}$. Since $W$ is a product of functions in $\mathcal D_{2N+2-2k}^{\z}$, we conclude by Lemma~\ref{lma:Wderiv} that $W\in\mathcal D_{2N+2-2k}^{\z}$. Next, for each $j\in\{0,\dots,k-1\}$ we have that $\pa_y^j\g\in\mathcal D_{2N-2j}^{\g}\subseteq\mathcal D_{2N+2-2k}^{\g}$.

Thus to summarize, $W\in\mathcal D_{2N+2-2k}^{\z}$ and $\pa_y^j\g\in\mathcal D_{2N+2-2k}^{\g}$. To check that $W\pa_y^j\g\in\mathcal D_{2N+2-2k}^{\g}$, we fix $\ell\in\{0,1,\dots,2N+2-2k\}$ and bound $\|\nabla_x^\ell(W\pa_y^j \g)\|_H$. Using the definitions~\eqref{fDM} and~\eqref{fDMq} of $W$ and $\pa_y^j\g$ belonging to $\mathcal D_{2N+2-2k}^{\z}$ and $\mathcal D_{2N+2-2k}^{\g}$, respectively, we have
\bs
\|\nabla_x^\ell(W\pa_y^j \g)\|_H\les &\sum_{m=0}^\ell\|\nabla_x^mW\|_H\|\nabla_x^{\ell-m}\pa_y^j\g\|_H\\
\les &\sum_{m=0}^2 d^{\lceil(l-m)/2\rceil}+\sum_{m=3}^{\ell}d^{\lceil(l-m)/2\rceil}d^{\lceil m/2\rceil-2}\\
\les &d^{\lceil\ell/2\rceil}+d^{\ell/2-1}\les d^{\lceil\ell/2\rceil}.\es This confirms $W\pa_y^j\g\in\mathcal D_{2N+2-2k}^{\g}$.
\end{proof}

\section{Results from~\cite{A24}}\label{app:lap} 
In this section, we consider the integral 
$$\int_{\us(R\e)}\g(x)e^{-\la\z(x)}\dd x,\qquad \la\gg1,$$ where $H=\nabla^2\z(0_d)\succ0$, $\e=\sqrt{d/\la}$, and $\us$ is defined in~\eqref{UH}. Let
\be\mathcal Z=\frac{(2\pi/\la)^{d/2}}{e^{\la \z (0_d)}\sqrt{\det(H)}}.\ee 
\begin{theorem}\label{thm:lap:orig}Suppose $\g\in C^{2L}(\us(R\e)), \z \in C^{2L+2}(\us(R\e))$, $R^4d^2/\la\leq\CR$, $R\geq1$, and
\bs\label{ck-cond}
\sup_{x\in\us(R\e)}\|\nabla^\ell\z(x)\|_H&\leq C_{k,\z} d^{\lceil \ell/2\rceil-2}\qquad\forall \ell=3,\dots,2L+2,\\
\sup_{x\in\us(R\e)}\|\nabla^\ell\g(x)\|_H&\leq C_{\ell,\g} d^{\lceil \ell/2\rceil}\qquad\forall \ell=0,\dots,2L.
\es
Then
\bs\label{eq:lap:orig}
\frac{1}{\mathcal Z}\int_{\us(R\e)}\g(x)e^{-\la \z (x)}\dd x = \sum_{k=0}^{L-1}\nu_k(\g)\la^{-k} + \Rem_L,
\es where
\bs\label{A2kRemL-orig}
\max_{k=0,\dots,L-1}d^{-2k}|\nu_k(\g)|&\leq C_{L,\z,\g},\\
|\Rem_L|&\leq C_{L,\z,\g,R}\left((d^2/\la)^L+e^{-(R-1)^2d/4}\right).
\es Here, $C_{L,\z,\g}$ depends on $L$, on $C_{\ell,\z}$, $\ell=3,\dots,2L+2$, and on $C_{\ell,\g}$, $\ell=0,\dots,2L$. The constant $C_{L,\z,\g,R}$ is analogous and also depends on $\CR$. We have $\nu_0(\g)=\g(0_d)$, and the formula for $\nu_k(\g)$, $k\geq1$ is given by $A_{2k}$ in~\cite[(2.20)]{A24}.\end{theorem}
\begin{proof}We use Theorem 2.12 of~\cite{A24}. Note that $\e=d/\sqrt{\la}$ in that theorem, unlike in the present paper where $\e=\sqrt{d/\la}$. We continue to use $\e$ to denote $\sqrt{d/\la}$. It is straightforward to see from the proof of the theorem that if we integrate over $\us(R\e)$ instead of over $\br^d$ then the only changes are that (1) the bound on $\Rem_L$ does not contain the term $\nu_{\us^c}$, (2) $R\geq40$ is not required, and (3) $0_d$ need not be the global minimizer of $\z$. Now, the fact that $\sup_{x\in\us(R\e)}\|\nabla^3\z(x)\|_H\leq C_{3,\z}$ and $\sup_{x\in\us(R\e)}\|\nabla^4\z(x)\|_H\leq C_{4,\z}$ and $R^4d^2/\la\leq\CR$ shows that $\exp([R^4c_3(R)^2+c_4(R)]d^2/\la)\leq f(\CR, C_{3,\z},C_{4,\z})$. Thus the bounds in Theorem 2.12 give
\bs\label{remLprelim}
|\nu_k(\g)|\la^{-k}\les_k &\CA_{2k}(0)(d^2/\la)^k,\quad k=1,\dots,L-1,\\
|\Rem_L|\les_L &f(\CR, C_{3,\z},C_{4,\z})\CA_{2L}(R)(d^2/\la)^L+ \max_{k=0,\dots,2L-1}\CA_k(0)(d^2/\la)^ke^{-(R-1)^2d/4}\\
\leq& f(L,\CR,C_{3,\z},C_{4,\z})\left(\CA_{2L}(R)(d^2/\la)^L+ \max_{k=0,\dots,2L-1}\CA_k(0)e^{-(R-1)^2d/4}\right).
\es To get the last line we used $d^2/\la\leq R^4d^2/\la\leq \CR$. Next, Lemma 2.14 of~\cite{A24} applied with $\tau=1$ shows that under~\eqref{ck-cond}, we have $\CA_j(0)\leq\CA_j(R)\leq C$ for all $j=0,\dots,2L$, where $C$ is some function of $C_{\ell,\z}$, $\ell=3,\dots,2L+2$ and $C_{\ell,\g}$, $\ell=0,\dots,2L$. Note that the condition (2.24) in Lemma 2.14 is not required for this statement, as is clear from the proof. Substituting these bounds into~\eqref{remLprelim} concludes the proof.
\end{proof}
In the next theorem, we use the following notation analogous to that used in~\cite{A24}:
\be\begin{gathered}
c_{0,\g}(R)=\sup_{x\in\us(R\e)}|\g(x)|,\\
c_k(R)=\sup_{x\in\us(R\e)}\|\nabla^k\z(x)\|_H,\qquad c_{k,\g}(R)=\sup_{x\in\us(R\e)}\|\nabla^k\g(x)\|_H,\\
\bar c_k(R)=c_k(0)+\frac{d}{\sqrt\la} c_{k+1}(R),\qquad\bar c_{k,\g}(R)=c_{k,\g}(0)+\frac{d}{\sqrt\la} c_{k+1,\g}(R).
\end{gathered}\ee
\begin{theorem}\label{thm:lapexplicit}Suppose $\g\in C^{2}(\us(R\e)), \z \in C^{4}(\us(R\e))$, $R\geq1$, and
\be\label{Rc34eps}
\left(R^2c_3(R)\frac{d}{\sqrt\la}\right)^2 + c_4(R)\frac{d^2}{\la}\leq \Clap.
\ee Then
$$
\frac{1}{\mathcal Z}\int_{\us(R\e)}\g(x)e^{-\la \z (x)}\dd x = \g(0_d)+ \Rem_1,
$$ where
\bs
|\Rem_1| \les_{\Clap} &\bigg\{c_{0,\g}(R)(c_3(0)^2+c_4(R))+d^{-1}c_{2,\g}(R) + d^{-1}c_{1,\g}(R)\big(c_3(0)+\frac{d}{\sqrt\la}c_4(R)\big)\bigg\}\frac{d^2}{\la} \\
&+ \left(|\g(0_d)|+ \frac{d}{\sqrt\la}c_{1,\g}(0)\right)e^{-(R-1)^2d/4}.
\es 
\end{theorem}
\begin{proof}
Theorem 2.12 of~\cite{A24} and~\eqref{Rc34eps} give
\be\label{Rem1CA2}
|\Rem_1|\les_{\Clap} \CA_2(R)\frac{d^2}{\la} + \max(\CA_0(0), \CA_1(0)d/\sqrt\la)e^{-(R-1)^2d/4},
\ee  where the $\CA_k$ are defined as follows:
\bs
\CA_0(0)&=|\g(0_d)|,\\
 \CA_1(0) &=c_{1,\g}(0)+ |\g(0_d)|c_3(0),\\
\CA_2(R)& =\left(d^{-1}c_{2,\g}+d^{-1}\bar c_{1,\g}\bar c_3 + c_{0,\g}(\bar c_3^2+c_4)\right)(R).
\es Note that 
\bsn
\bar c_3(R)^2+c_4(R) &\les c_3(0)^2 + \frac{d^2}{\la}c_4(R)^2 + c_4(R) \\
&= c_3(0)^2 + (1+c_4(R)d^2/\la)c_4(R) \les_{C_{34}} c_3(0)^2 + c_4(R),
\esn 
using~\eqref{Rc34eps} to get the last line. Also, we write
\bsn
(\bar c_{1,\g}\bar c_3+c_{2,\g})(R) &= c_{1,\g}(0)\bar c_3(R) + c_{2,\g}(R)\left(1+\frac{d}{\sqrt\la}\bar c_3(R)\right) \\
&\les_{C_{34}}  c_{1,\g}(0)\bar c_3(R)+c_{2,\g}(R).
\esn To get the inequality, we used that $\frac{d}{\sqrt\la}\bar c_3(R)=c_3(0)\frac{d}{\sqrt\la} + c_4(R)\frac{d^2}{\la}\leq C_{34}$ by~\eqref{Rc34eps} and the fact that $R\geq1$. Thus to summarize, we have
\bs\label{CA2bd}
\CA_2(R) \les_{C_{34}} &\;d^{-1}c_{2,\g}(R) + d^{-1}c_{1,\g}(0)\left(c_3(0)+\frac{d}{\sqrt\la} c_4(R)\right) + c_{0,\g}(R)(c_3(0)^2 + c_4(R)).
\es
Finally, we have
\bs\label{CA01bd}
\max(\CA_0(0),\CA_1(0)d/\sqrt\la) &= \max(|\g(0_d)|, \frac{d}{\sqrt\la}c_{1,\g}(0) + |\g(0_d)|\frac{d}{\sqrt\la}c_3(0))\\
&\les_{C_{34}} |\g(0_d)|+ \frac{d}{\sqrt\la}c_{1,\g}(0).
\es  Combining~\eqref{Rem1CA2},~\eqref{CA2bd}, and~\eqref{CA01bd} finishes the proof.
\end{proof}
\begin{corollary}\label{corr:norm-const} Suppose $\z \in C^{4}(\us(R\e))$ and~\eqref{Rc34eps} holds, with $R\geq1$. Then
$$
\frac{1}{\mathcal Z}\int_{\us(R\e)}e^{-\la \z (x)}\dd x \leq C(C_{3,4}).
$$
\end{corollary}
Next, we present some results which follow from Appendix D of~\cite{A24}. The key proof element in these results is to use the idea of high-dimensional Gaussian concentration, which states that, for $Z\sim\mathcal N(0_d, I_d)$, the variation of $h(Z)$ about its expectation is controlled by the maximum value of $\|\nabla h(x)\|$ achieved in the region of integration. This gives more precise control then do other approaches based on bounding the function $h(x)$ itself. See~\cite{A24} for more details.
\begin{lemma}\label{lma:Ef2}
Let $f$ be a function on $\br^d$ and $h(x) = f(x) -f(0_d)-\nabla f(0_d)^\top x - \frac12x^\top\nabla^2f(0_d)x$. Let $X\sim \mathcal N(0_d, (\la H)^{-1})$ and $\us = \us(R\e)$. Let $f_3 = \|\nabla^3f(0_d)\|_H$ and $f_4=\sup_{x\in\us}\|\nabla^4f(x)\|_H$. Then 
\begin{align}
\la\E[h(X)^2\ind\{X\in\us\}]^{1/2} &\les f_3\frac{d}{\sqrt\la} + f_4\frac{d^2}{\la},\label{Er2}
\end{align} Now, suppose that  $R$ is large enough that $\mathbb P(\|Z\|\leq R\sqrt d)\geq1/2$, and 
\be\label{f34}
(f_3^2+f_4)\frac{d^2}{\la}\leq \Cf.\ee Then
\begin{align}
\E[\exp(\la h(X))\ind\{X\in\us\}]&\leq C(R, \Cf).\label{Erexp}
\end{align}
\end{lemma}
\begin{proof}A Taylor expansion gives that $h(x) = \frac16\langle\nabla^3f(0_d), x^{\otimes3}\rangle + q_4(x)$, where 
\be\label{q4b}
|q_4(x)| \leq f_4\|H^{1/2}x\|^4\qquad\forall x\in\us.\ee Let $Z\sim\mathcal N(0,I_d)$, so that we may write $X=(\la H)^{-1/2}Z$. We then have
\bs\label{laqdef}
\la h(X) &= \la h((\la H)^{-1/2}Z) = \la^{-1/2}\frac16\langle S, Z^{\otimes3}\rangle + r_4(Z),\\
r_4(x)&:=\la q_4((\la H)^{-1/2}x),
\es where $S$ is the tensor such that $\langle S, x^{\otimes3}\rangle = \langle \nabla^3f(0), (H^{-1/2}x)^{\otimes3}\rangle$ for all $x$. Thus
\bs\label{laqx}
\E[(\la h(X))^2\ind\{X\in\us\}]& = \E\left[\left( \la^{-1/2}\frac16\langle S, Z^{\otimes3}\rangle + r_4(Z)\right)^2\ind\{\|Z\|\leq R\sqrt d\}\right]\\
&\les \la^{-1}\E[\langle S, Z^{\otimes3}\rangle^2] + \E[r_4(Z)^2\ind\{\|Z\|\leq R\sqrt d\}].
\es Now, note that for all $\|x\|\leq R\sqrt d$, we have $(\la H)^{-1/2}x\in\us$, and therefore 
\be\label{r4b}|r_4(x)|\leq \la f_4\|H^{1/2}(\la H)^{-1/2}x\|^4 = \la^{-1}f_4\|x\|^4\qquad\forall \|x\|\leq R\sqrt d.\ee using~\eqref{q4b}. Thus 
\be\label{laqx4}
\E[r_4(Z)^2\ind\{\|Z\|\leq R\sqrt d\}]\leq \la^{-2}f_4^2\E[\|Z\|^8] \les \la^{-2}f_4^2d^4.\ee Also, Proposition D.4 of~\cite{A24} gives 
\be\label{laqx3}\E[\langle S, Z^{\otimes3}\rangle^2]\les (\|S\|d)^2 =(\|\nabla^3f(0)\|_Hd)^2=(f_3d)^2.\ee Substituting~\eqref{laqx4},~\eqref{laqx3} in~\eqref{laqx} now gives
$$
\E[(\la h(X))^2\ind\{X\in\us\}] \les f_3\frac{d^2}{\la}  + f_4^2\frac{d^4}{\la^2}.
$$ Taking the square root gives the desired bound~\eqref{Er2}. Next, note that
\bs
\E[\exp(\la h(X))\ind\{X\in\us\}] &= \E[\exp(r(Z))\ind\{\|Z\|\leq R\sqrt d\}],\\
 r(x)&:=\la h((\la H)^{-1/2}x).
\es We apply Lemma~\ref{lma:Erexp}. To do so, we bound $\|\nabla r\|$ over $\|x\|\leq R\sqrt d$. Fix $u$ such that $\|H^{1/2}u\|=1$. A Taylor expansion gives 
\bs
\nabla h(x)^\top u&=(\nabla f(x) - \nabla f(0))^\top u -x^\top\nabla^2f(0)u \\
&=\frac12\langle\nabla^3f(0), x\otimes x\otimes u\rangle + q_4(x,u),\\  
|q_4(x,u)| &\leq f_4\|H^{1/2}x\|^3\leq f_4(R\e)^3\qquad\forall x\in\us.\es We conclude that 
$$\sup_{x\in\us}\|H^{-1/2}\nabla h(x)\|=\sup_{x\in\us, \|H^{1/2}u\|=1}\nabla h(x)^\top u\leq (R\e)^2f_3 + (R\e)^3f_4.$$ Therefore,
\be\label{supnab}
\sup_{\|x\|\leq R\sqrt d}\|\nabla r(x)\| = \sqrt\la\sup_{x\in\us}\|H^{-1/2}\nabla h(x)\|\leq R^2f_3\frac{d}{\sqrt\la} + \frac{R^3}{\sqrt d}f_4\frac{d^2}{\la}.
\ee
Next, using~\eqref{laqdef} and~\eqref{r4b}, and recalling $r(Z)=\la h((\la H)^{-1/2}Z)$, we have 
\bs\label{Erz}
\left|\E\left[r(Z)\ind\left\{\|Z\|\leq R\sqrt d\right\}\right]\right| &=\left|\E\left[r_4(Z)\ind\left\{\|Z\|\leq R\sqrt d\right\}\right]\right|\leq f_4\frac{d^2}{\la}.
\es 
Using~\eqref{supnab} and~\eqref{Erz} in~\eqref{eq:Erexp} now gives
\be
\E[\exp(\la h(X))\ind\{X\in\us\}] \leq \exp\left(2f_4d^2/\la + \left(R^2f_3\frac{d}{\sqrt\la} + \frac{R^3}{\sqrt d}f_4\frac{d^2}{\la}\right)^2\right)
\ee
Using~\eqref{f34}, we conclude that~\eqref{Erexp} holds.
\end{proof}

\begin{lemma}\label{lma:Erexp}Let $r$ be $C^1$ in $\{\|x\|\leq R\sqrt d\}$, where $R$ is large enough that $\mathbb P(\|Z\|\leq R\sqrt d)\geq1/2$. Then
\be
\log\E\left[\exp(r(Z))\ind\{\|Z\|\leq R\sqrt d\}\right]\leq 2\E\left[r(Z)\ind\{\|Z\|\leq R\sqrt d\}\right]+\sup_{\|x\|\leq R\sqrt d}\|\nabla r(x)\|^2.\label{eq:Erexp}
\ee
\end{lemma}
This is a restatement of Corollary D.3 in~\cite{A24}.

\section{The term $a_1$}\label{app:term2}
In this section, we use shorthand such as
\be\label{shortzqderiv}
\z_{xy}= \nabla_x\pa_y\z(0_{d+1}),\quad \z_{xxy}=\nabla_x^2\pa_y\z(0_{d+1}),\quad \g_x = \nabla_x\g(0_{d+1}),\quad \g_y=\pa_y\g(0_{d+1}),
\ee and so on. Recall from Proposition~\ref{prop:main} that the terms of the expansion are given by $a_md^{2m}=\sum_{k=1}^{m+1}\nu_{m+1-k}(\g_k)$, where the $\nu_{m+1-k}$ are as in Definition~\ref{def:lap}. In particular, $$d^2a_1 = \nu_0(\g_2)+\nu_1(\g_1)=\g_2(0_d)+\nu_1(\g_1),$$ where the functions $\g_1,\g_2$ are given explicitly in~\eqref{g12def}. For convenience, we repeat~\eqref{g12def} here:
\be\label{app:g12def}
\g_1(x) = \frac{\g(x,0)}{\pa_y\z(x,0)},\qquad \g_2(x) =  \frac{\pa_y\g(x,0)}{\left(\pa_y\z(x,0)\right)^2}- \frac{\g(x,0)\pa_y^2\z(x,0)}{\left(\pa_y\z(x,0)\right)^3}.
\ee Using that $\pa_y\z(0_{d+1})=\pa_yz(0_{d+1})=1$, we see that
\be\label{nu0g2}
\g_2(0) = \g_y - \g(0_{d+1})\z_{yy}.\ee To define $\nu_1(\g_1)$, let $f(x) = \g_1(H^{-1/2}x)$ and $v(x)=\z(H^{-1/2}x)$. As shown in Appendix A of~\cite{A24}, the formula for $\nu_1(\g_1)$ is given as follows:
\bs\label{nu1-gen}
\nu_1(\g_1)  =&\frac12\Delta f(0)-\frac12\sum_{i,j=1}^d\partial_if(0)\partial_{ijj}^3v(0_d) + \frac{f(0)}{12}\|\nabla^3v(0_d)\|^2_{F} \\
&+ \frac{f(0)}8\sum_{i=1}^d\left(\partial_i\Delta v(0_d)\right)^2 - \frac{f(0)}8\sum_{i,j=1}^d\partial_i^2\partial_j^2v(0_d)\\
=&\frac12\langle\nabla^2f(0), I_d\rangle - \frac12\langle\nabla^3v(0_d), \nabla f(0)\otimes I_d\rangle + \frac{f(0)}{12}\|\nabla^3v(0_d)\|^2_{F} \\
&+\frac{f(0)}8\|\langle\nabla^3v(0_d), I_d\rangle\|^2 - \frac{f(0)}8\langle\nabla^4v(0_d), I_d\otimes I_d\rangle.
\es
We write this formula slightly more explicitly, keeping $v$ as is but using $\g_1$ instead of $f$. This gives
\bs\label{nu1g1}
\nu_1(\g_1)=& \frac12\langle\nabla^2 \g_1(0_d),H^{-1}\rangle-\frac12\left\langle\nabla^3v(0_d),I_d\otimes H^{-1/2}\nabla \g_1(0) \right\rangle \\
&+ \g_1(0_d)\left(\frac{1}{12}\|\nabla^3v(0_d)\|^2_{F} + \frac18\|\langle\nabla^3v(0_d),I_d\rangle\|^2 - \frac18\langle\nabla^4v(0_d), I_d\otimes I_d\rangle\right).
\es 
Now, we have $\g_1(0_d)=\g(0_{d+1})$, and we compute
\bs
\nabla\g_1 = &\frac{\nabla_x\g}{\pa_y\z} - \frac{\g\nabla_x\pa_y\z}{(\pa_y\z)^2},\\
\nabla^2\g_1 = &\frac{\nabla_x^2\g}{\pa_y\z} - \frac{\nabla_x\g\nabla_x\pa_y\z^\top + \nabla_x\pa_y\z\nabla_x\g^\top + \g\nabla_x^2\pa_y\z}{(\pa_y\z)^2}+2\frac{\g\nabla_x\pa_y\z\nabla_x\pa_y\z^\top}{(\pa_y\z)^3}.
\es
Using that $\pa_y\z(0_{d+1})=\pa_yz(0_{d+1})=1$, and the shorthands~\eqref{shortzqderiv}, we obtain
\bs\label{nabla12g1}
\nabla\g_1(0)=&\g_x-\g(0_{d+1})\z_{xy},\\
\nabla^2\g_1(0)=&\g_{xx}-\g(0_{d+1})\z_{xxy}-(\g_x\z_{xy}^\top + \z_{xy}\g_x^\top)+2\g(0_{d+1})\z_{xy}\z_{xy}^\top.
\es 
Substituting~\eqref{nabla12g1} and $\g_1(0_d)=\g(0_{d+1})$ into~\eqref{nu1g1}, and combining~\eqref{nu1g1} with~\eqref{nu0g2}, we obtain
\bs\label{a2-gen}
d^2a_1 = &\nu_0(\g_2)+\nu_1(\g_1)=\g_2(0_d)+\nu_1(\g_1) \\
= &\g_y - \g(0_{d+1})\z_{yy} + \frac12\langle \g_{xx},H^{-1}\rangle -\frac{\g(0_{d+1})}{2}\langle\z_{xxy}, H^{-1}\rangle\\
&-\z_{xy}^\top H^{-1}\g_x +\g(0_{d+1})\z_{xy}^\top H^{-1}\z_{xy}-\frac12\left\langle\nabla^3v(0_d),I_d\otimes H^{-1/2}(\g_x-\g(0_{d+1})\z_{xy}) \right\rangle\\
&+\g(0_{d+1})\left(\frac{1}{12}\|\nabla^3v(0_d)\|^2_{F} + \frac18\|\langle\nabla^3v(0_d),I_d\rangle\|^2 - \frac18\langle\nabla^4v(0_d), I_d\otimes I_d\rangle\right).
\es
\subsection{Constant $g$}\label{app:a2:constg}
When $g\equiv1$, we have $\g\equiv1$ as well. Then all derivatives of $\g$ vanish, and~\eqref{a2-gen} simplifies as follows:
\bs\label{a2-const-g}
d^2a_1 = &\nu_0(\g_2)+\nu_1(\g_1) \\
= & - \z_{yy}  -\frac{1}{2}\langle\z_{xxy}, H^{-1}\rangle +\z_{xy}^\top H^{-1}\z_{xy}+\frac12\left\langle\nabla^3v(0_d),I_d\otimes H^{-1/2}\z_{xy} \right\rangle\\
&+\frac{1}{12}\|\nabla^3v(0_d)\|^2_{F} + \frac18\|\langle\nabla^3v(0_d),I_d\rangle\|^2 - \frac18\langle\nabla^4v(0_d), I_d\otimes I_d\rangle.
\es
\subsection{Constant $g$, Gaussian case}\label{app:a2:gauss}
Finally, we consider the Gaussian setting of Section~\ref{sec:gauss}. Recall from~\eqref{z-gauss} that $z(x,t)=\|x\|^2/2+(t+1)^2/2$, and therefore $\z(x,y)=\|x\|^2/2+(\psi(x)+y+1)^2/2$, where $\psi$ parameterizes the boundary $\pa D$ near $0_{d+1}$. Using the computations in Appendix~\ref{app:gauss:elem} and recalling $\nabla\psi(0_d)=0_d$, we obtain
\bs\label{zyy-g}
 \z_{yy}&=1, \qquad \z_{xy} = 0, \qquad \z_{xxy}=\nabla^2\psi(0_d),\\
 \nabla_x^3\z(0_{d+1})&=\nabla^3\psi(0_d),\quad\nabla_x^4\z(0_{d+1})=\nabla^4\psi(0_d)+3\mathrm{Sym}(\nabla^2\psi(0_d)\otimes \nabla^2\psi(0_d)),
\es
where Sym is defined in~\eqref{sym}. Next, note that because $v(x)=\z(H^{-1/2}x,0)$ and $\nabla_x^3\z(0_{d+1})=\nabla^3\psi(0_d)$, we have $\nabla^3v(0_d)=\nabla^3\psi_H(0_d)$, where $\psi_H(x)=\psi(H^{-1/2}x)$. Thus
\be\label{v3psi}
\frac{1}{12}\|\nabla^3v(0_d)\|^2_{F} + \frac18\|\langle\nabla^3v(0_d),I_d\rangle\|^2 =\frac{1}{12}\|\nabla^3\psi_H(0_d)\|^2_{F} + \frac18\|\langle\nabla^3\psi_H(0_d),I_d\rangle\|^2.
\ee
We also have
\bs\label{vwpsi}
\langle\nabla^4v(0_d), I_d\otimes I_d\rangle &= \langle\nabla_x^4\z(0_{d+1}), H^{-1}\otimes H^{-1}\rangle \\
&=  \langle\nabla^4\psi(0_d), H^{-1}\otimes H^{-1}\rangle + 3\langle\mathrm{Sym}(\nabla^2\psi(0_d)\otimes \nabla^2\psi(0_d)), H^{-1}\otimes H^{-1}\rangle\\
&= \langle\nabla^4\psi(0_d), H^{-1}\otimes H^{-1}\rangle + \mathrm{Tr}(\nabla^2\psi(0_d)H^{-1})^2 + 2\mathrm{Tr}\left((\nabla^2\psi(0_d)H^{-1})^2\right)\\
&=\langle\nabla^4\psi_H(0_d), I_d\otimes I_d\rangle +  \mathrm{Tr}(\nabla^2\psi_H(0_d))^2 + 2\mathrm{Tr}\left(\nabla^2\psi_H(0_d)^2\right)
\es
Here, we used that
\be\label{AB}
\langle\mathrm{Sym}(A\otimes A), B\otimes B\rangle = \frac13\langle A, B\rangle^2 + \frac23\langle AB, AB\rangle
\ee for two symmetric matrices $A$ and $B$, which can be checked by a straightforward calculation. Using~\eqref{zyy-g},~\eqref{v3psi},~\eqref{vwpsi} in~\eqref{a2-const-g} gives
\bs
d^2a_1 = & - 1  -\frac12\mathrm{Tr}(\nabla^2\psi_H(0_d)) +\frac{1}{12}\|\nabla^3\psi_H(0_d)\|^2_{F} + \frac18\|\langle\nabla^3\psi_H(0_d),I_d\rangle\|^2 \\
&-\frac18\langle\nabla^4\psi_H(0_d), I_d\otimes I_d\rangle-\frac18 \mathrm{Tr}(\nabla^2\psi_H(0_d))^2 -\frac14\mathrm{Tr}\left(\nabla^2\psi_H(0_d)^2\right).
\es This proves~\eqref{a1gengauss}.

We now consider two special cases. First, consider a pure quartic boundary, $\psi(x)=\frac{1}{4!}\langle S, x^{\otimes4}\rangle$ for a symmetric tensor $S$. Then $\nabla^2\psi(0_d)$ and $\nabla^3\psi(0_d)$ are the zero tensors. Thus $H=I_d$, so that $\psi_H(x)=\psi(x)$. We then have $\nabla^4\psi_H(0_d)=\nabla^4\psi(0_d)=S$. We conclude that
\be
d^2a_1 =-1 -\frac18\langle S, I_d\otimes I_d\rangle. \ee For example, if $\psi(x)=\|x\|^4/24$, then we can write $\psi$ as $$\psi(x)=\frac{1}{24}\langle I_d\otimes I_d, x^{\otimes4}\rangle = \frac{1}{24}\langle \mathrm{Sym}(I_d\otimes I_d), x^{\otimes4}\rangle.$$
Thus
\be
d^2a_1 =-1 -\frac18\langle \mathrm{Sym}(I_d\otimes I_d), I_d\otimes I_d\rangle = -\frac{d^2}{24} - \frac{d}{12},
\ee
using~\eqref{AB}.
Second, consider a quadratic boundary, $\psi(x)=\frac12x^\top Bx$. Then $\psi_H(x)=\frac12x^\top B_Hx$, where $B_H = H^{-1/2}BH^{-1/2} = (I_d+B)^{-1/2}B(I_d+B)^{-1/2}$. Thus $\nabla^3\psi_H(0_d)$ and $\nabla^4\psi_H(0_d)$ are both zero tensors, and we obtain
\bs\label{a1quad}
d^2a_1 &= -1-\frac12\mathrm{Tr}(B_H)-\frac18\mathrm{Tr}(B_H)^2-\frac14\mathrm{Tr}(B_H^2)\\
&=-1-\frac12\sum_{i=1}^d\frac{b_i}{1+b_i}-\frac18\left(\sum_{i=1}^d\frac{b_i}{1+b_i}\right)^2-\frac14\sum_{i=1}^d\left(\frac{b_i}{1+b_i}\right)^2.
\es
In the second line, the $b_i$ are the eigenvalues of $B$.
\section{Classical results: Gaussian tails and Watson's Lemma}
Throughout the text, we use the standard Gaussian tail bound
\be\label{gausstail}
\mathbb P(\|Z\|\geq R\sqrt d) \leq \exp(-(R-1)^2d/2),\qquad R\geq1,\quad Z\sim\mathcal N(0_d,I_d).
\ee
See e.g.~\cite[Example 2.12]{gaussconc}.

Next, we formulate Watson's lemma, which is well-known and classical. Wee state and prove the result in a form which is convenient for our purposes, and which is not readily available in the literature. 
\begin{lemma}[Watson's Lemma]\label{lma:gen-wat}
Let $f,h\in C^L([0,T])$ and suppose  $h(0)=0$ and $h'(t)>b>0$ for all $t\in[0,T]$. Define the operator $D_h$ by $(D_hf)(t):=\frac{d}{dt}\left(\frac{f(t)}{h'(t)}\right)$. Then
\be\label{gen-wat-1}
\int_0^Tf(t)e^{-\la h(t)}\dd t=\frac{1}{h'(0)}\sum_{k=1}^{L-1}(D_h^{k-1}f)(0)\la^{-k}+\Rem_{L},
\ee where
\bs\label{RL-gen}
|\Rem_{L}|\leq  b^{-1}\|D_h^{L-1}f\|_\infty\la^{-L} + b^{-1}e^{-\la b T}\sum_{k=1}^{L-1}\la^{-k}|(D_h^{k-1}f)(T)|.
\es Here, $\|\cdot\|_\infty$ denotes the sup norm over the interval $[0,T]$.
\end{lemma}
This formulation differs from standard results in that our error bound is explicit in its dependence on derivatives of the involved functions. We apply the lemma to functions $f$ and $h$ given by functions of $d+1$ variables, with $d$ of the $d+1$ arguments frozen at some fixed value. Thus $f$ and $h$ can depend on $d$, which is why we must quantify the bound's dependence on their derivatives.
\begin{proof}
Let $I(f)=\int_0^Tf(t)e^{-\la h(t)}\dd t$. Integrating by parts, we have
\bs
I(f) &= -\la^{-1}\int_0^T\frac{f(t)}{h'(t)}\frac{d}{dt}\left(e^{-\la h(t)}\right)\dd t \\
&= \la^{-1}\frac{f(0)}{h'(0)}-\la^{-1}\frac{f(T)}{h'(T)}e^{-\la h(T)} + \la^{-1}I(D_hf).
\es
Iterating this rule, we have
\bs
I(f)&= \int_0^Tf(t)e^{-\la h(t)}\dd t =\frac{1}{h'(0)}\sum_{k=1}^{L-1}\la^{-k}(D_h^{k-1}f)(0) + \Rem_L^\Wat,\\
\Rem_L^\Wat&:= - \frac{e^{-\la h(T)}}{h'(T)}\sum_{k=1}^{L-1}\la^{-k}(D_h^{k-1}f)(T) + \la^{-(L-1)}I(D_h^{L-1}f).
\es
Finally, note that
\bs
|I(D_h^{L-1}f)| &\leq \|D_h^{L-1}f\|_\infty\int_0^Te^{-\la h(t)}\dd t \\
&\leq \|D_h^{L-1}f\|_\infty\int_0^Te^{-\la bt}\dd t \leq \frac{\|D_h^{L-1}f\|_\infty}{\la b}.
\es
\end{proof}

\section{Bringing general case into the form of Assumption~\ref{assume1}}\label{app:gen}
Consider the integral $\int_{\bar D}e^{-\bar\la\bar z(u)}\dd u$, where
$$
u^*=\arg\min_{u\in\bar D}\bar z(u) \in \pa\bar D,\qquad \|\nabla\bar z(u^*)\|\neq 0, \qquad \bar D=\{u\in\br^{d+1}\,:\,F(u)\geq F_0\}.
$$
Let $n=\nabla \bar z(u^*)/\|\nabla\bar z(u^*)\|$. Let $U=(U_1,n)\in\br^{d+1}$ be an orthonormal matrix whose first $d$ columns form the matrix $U_1\in\br^{(d+1)\times d}$, and whose last column is $n$. Now, define
\be\label{zbar}z(u)= \bar z(u^* + Uu)/\|\nabla\bar z(u^*)\|,\qquad \la =\bar\la\|\nabla\bar z(u^*)\|,\qquad D=U^\top(\bar D-u^*).\ee Then
$$\int_{\bar D}e^{-\bar\la\bar z(u)}\dd u = \int_D e^{-\la z(u)}\dd u.$$ Furthermore, $0_{d+1}$ is the minimizer of $z$ over $D$, and we have $\nabla z(0_{d+1}) =(0_d, 1)$. We compute
\be\label{nablax-gen}
\nabla_x^2z(0_{d+1})=\frac{U_1^\top\nabla^2\bar z(u^*)U_1}{\|\nabla\bar z(u^*)\|}.
\ee
Now, note that the boundary $\pa\bar D$ is described by points $u$ such that $F(u)=F_0$. Furthermore, suppose $\pa D$ can be locally parameterized as $(x,\psi(x))$. Then for $u\in\pa\bar D$, we have $u=u^* + U((x,\psi(x))^\top) = u^* + U_1x +\psi(x)n$. But then
$$F(u^* + U_1x +\psi(x)n)=F_0.$$ We now compute the quadratic form $\langle\nabla^2\psi(x), b^{\otimes2}\rangle$ by substituting $x=x+sb$ above and differentiating twice with respect to $s$ at zero. We obtain
\bs\label{F1}
&\frac{d}{ds}F(u^*+U_1(x+sb)+\psi(x+sb)n) \\
&= \nabla F(u^*+U_1(x+sb)+\psi(x+sb)n)^\top\left(U_1b + n\nabla\psi(x+sb)^\top b\rangle\right),\\
0=&\frac{d}{ds}F(u^*+U_1(x+sb)+\psi(x+sb)n)\big\vert_{s=0}\\
&=\nabla F(u)^\top U_1b+ \left(\nabla\psi(x)^\top b \right) \nabla F(u)^\top n,
\es where $u=u^*+U_1x+\psi(x)n$, and
\bs\label{F2}
0=&\frac{d^2}{ds^2}F(u^*+U_1(x+sb)+\psi(x+sb)n)\big\vert_{s=0} \\
= &\langle \nabla^2 F(u), \, (U_1b + n\nabla\psi(x)^\top b)^{\otimes2}\rangle\\
&+\langle\nabla F(u), \, n\rangle b^\top\nabla^2\psi(x)b.
\es
We now take $x=0_d$, so that $u=u^*$, in~\eqref{F2}. We use that $\nabla\psi(0_d)=0_d$ to get
$$
b^\top\nabla^2\psi(0_d)b =\frac{b^\top U_1^\top\nabla^2F(u^*)U_1b}{\nabla F(u^*)^\top n}
$$ Thus we conclude that
\be\label{psi-gen}
\nabla^2\psi(0_d) = \frac{U_1^\top\nabla^2F(u^*)U_1}{\nabla F(u^*)^\top n}=\|\nabla\bar z(u^*)\|\frac{U_1^\top\nabla^2F(u^*)U_1}{\nabla F(u^*)^\top \nabla\bar z(u^*)}=\frac{U_1^\top\nabla^2F(u^*)U_1}{\|\nabla F(u^*)\|}.\ee To get the last equality we used that $\nabla F(u^*)$ is parallel to $\nabla \bar z(u^*)$, since $u^*$ minimizes $\bar z$ over $\{F(u)\geq F_0\}$. Thus $\nabla F(u^*)^\top \nabla\bar z(u^*)=\|\nabla F(u^*)\|\| \nabla\bar z(u^*)\|$. Combining~\eqref{nablax-gen} and~\eqref{psi-gen} now gives
\be\label{Hbar}
H=\nabla_x^2z(0_{d+1})+\nabla^2\psi(0_d)=U_1^\top\left(\frac{\nabla^2\bar z(u^*)}{\|\nabla\bar z(u^*)\|}+\frac{\nabla^2F(u^*)}{\|\nabla F(u^*)\|}\right)U_1.
\ee The normalization by $\nabla\bar z(u^*)$ arises because we have forced the new function $z$ from~\eqref{zbar} to have gradient norm 1 at the (new) instanton $0_{d+1}$. The normalization of $\nabla^2F(u^*)$ by $\|\nabla F(u^*)\|$ is natural, since the set $\{F(u)\geq F_0\}$ is invariant to rescaling $F$ by a constant. 

\subsection{The general sampling procedure}\label{app:gen:sample}We now consider the problem of sampling from $\bar\pi\vert_{\bar D}$, the restriction of the probability density $\bar\pi(u)\propto e^{-\bar\la\bar z(u)}$ to the set $\bar D=\{u\in\br^{d+1}\,:\,F(u)\geq F_0\}$. Let $\la, z, D$ be as in~\eqref{zbar}, and let $(X,T)\sim\pi\vert_D$. Then 
$$u^* + U\begin{pmatrix}X\\T\end{pmatrix} = u^*+U_1X + Tn \sim \bar\pi\vert_{\bar D}.$$ Correspondingly, the density of the approximation to $\bar\pi\vert_{\bar D}$ is the law of $u^*+U_1\hat X + \hat T n$, where $(\hat X,\hat T)\sim\hat\pi$ and $\hat\pi$ is the approximation to $\pi\vert_D$ from~\eqref{pi-hat}. We can now summarize the procedure to approximately sample from $\bar\pi\vert_{\bar D}$:
\begin{enumerate}
\item Draw $\hat X \sim \mathcal N(0_d, (\la H)^{-1})$, where $\la=\bar\la\|\nabla z(u^*)\|$ and $H$ is as in~\eqref{Hbar}.
\item Draw $\hat Y\sim\mathrm{Exp}(\la)$ independently of $\hat X$.
\item Set $\hat T = \hat Y + \frac12\hat X^\top\nabla^2\psi(0_d)\hat X$, where $\nabla^2\psi(0_d)$ is as in~\eqref{psi-gen}.
\item Return $u^*+U_1\hat X + \hat T n$.
\end{enumerate} 

\bibliographystyle{abbrv}
\bibliography{bibliogr_Lap}

\begin{thebibliography}{10}

\bibitem{agapiou2017importance}
S.~Agapiou, O.~Papaspiliopoulos, D.~Sanz-Alonso, and A.~M. Stuart.
\newblock Importance sampling: Intrinsic dimension and computational cost.
\newblock {\em Statistical Science}, pages 405--431, 2017.

\bibitem{au2001estimation}
S.-K. Au and J.~L. Beck.
\newblock Estimation of small failure probabilities in high dimensions by
  subset simulation.
\newblock {\em Probabilistic engineering mechanics}, 16(4):263--277, 2001.

\bibitem{bleistein1975asymptotic}
N.~Bleistein and R.~A. Handelsman.
\newblock {\em Asymptotic expansions of integrals}.
\newblock Ardent Media, 1975.

\bibitem{gaussconc}
S.~Boucheron, G.~Lugosi, and P.~Massart.
\newblock {\em Concentration Inequalities: A Nonasymptotic Theory of
  Independence}.
\newblock Oxford University Press, 02 2013.

\bibitem{breitung1984asymptotic}
K.~Breitung.
\newblock Asymptotic approximations for multinormal integrals.
\newblock {\em Journal of Engineering Mechanics}, 110(3):357--366, 1984.

\bibitem{cerou2012sequential}
F.~C{\'e}rou, P.~Del~Moral, T.~Furon, and A.~Guyader.
\newblock Sequential {M}onte {C}arlo for rare event estimation.
\newblock {\em Statistics and computing}, 22(3):795--808, 2012.

\bibitem{cerou2007adaptive}
F.~C{\'e}rou and A.~Guyader.
\newblock Adaptive multilevel splitting for rare event analysis.
\newblock {\em Stochastic Analysis and Applications}, 25(2):417--443, 2007.

\bibitem{choi2007reliability}
S.-K. Choi, R.~A. Canfield, and R.~V. Grandhi.
\newblock {\em Reliability-based structural design}.
\newblock Springer, 2007.

\bibitem{dematteis2019extreme}
G.~Dematteis, T.~Grafke, and E.~Vanden-Eijnden.
\newblock Extreme event quantification in dynamical systems with random
  components.
\newblock {\em SIAM/ASA Journal on Uncertainty Quantification},
  7(3):1029--1059, 2019.

\bibitem{ditlevsen1996structural}
O.~Ditlevsen and H.~O. Madsen.
\newblock {\em Structural Reliability Methods}.
\newblock John Wiley \& Sons, Chichester, 1996.

\bibitem{el2021improvement}
M.~El~Masri, J.~Morio, and F.~Simatos.
\newblock Improvement of the cross-entropy method in high dimension for failure
  probability estimation through a one-dimensional projection without gradient
  estimation.
\newblock {\em Reliability Engineering \& System Safety}, 216:107991, 2021.

\bibitem{embrechts1997modelling}
P.~Embrechts, C.~Kl{\"u}ppelberg, and T.~Mikosch.
\newblock {\em Modelling Extremal Events: for Insurance and Finance}, volume~33
  of {\em Applications of Mathematics}.
\newblock Springer, Berlin, 1997.

\bibitem{FriedliLinde2024RareEvents}
L.~Friedli and N.~Linde.
\newblock Rare event probability estimation for groundwater inverse problems
  with a two-stage sequential monte carlo approach.
\newblock {\em Water Resources Research}, 60, 2024.

\bibitem{grafke2019numerical}
T.~Grafke and E.~Vanden-Eijnden.
\newblock Numerical computation of rare events via large deviation theory.
\newblock {\em Chaos: An Interdisciplinary Journal of Nonlinear Science},
  29(6), 2019.

\bibitem{helin2022non}
T.~Helin and R.~Kretschmann.
\newblock Non-asymptotic error estimates for the {L}aplace approximation in
  {B}ayesian inverse problems.
\newblock {\em Numerische Mathematik}, 150(2):521--549, 2022.

\bibitem{hu2021second}
Z.~Hu, R.~Mansour, M.~Olsson, and X.~Du.
\newblock Second-order reliability methods: a review and comparative study.
\newblock {\em Structural and multidisciplinary optimization},
  64(6):3233--3263, 2021.

\bibitem{kahn1953methods}
H.~Kahn and A.~W. Marshall.
\newblock Methods of reducing sample size in {M}onte {C}arlo computations.
\newblock {\em Journal of the Operations Research Society of America},
  1(5):263--278, 1953.

\bibitem{bp}
M.~J. Kasprzak, R.~Giordano, and T.~Broderick.
\newblock How good is your {L}aplace approximation of the {B}ayesian posterior?
  {F}inite-sample computable error bounds for a variety of useful divergences.
\newblock {\em Journal of Machine Learning Research}, 26(87):1--81, 2025.

\bibitem{katskew}
A.~Katsevich.
\newblock The {L}aplace approximation accuracy in high dimensions: a refined
  analysis and new skew adjustment.
\newblock {\em arXiv preprint arXiv:2306.07262}, 2023.

\bibitem{A24}
A.~Katsevich.
\newblock The {L}aplace asymptotic expansion in high dimensions.
\newblock {\em arXiv preprint arXiv:2406.12706}, 2024.

\bibitem{katsBVM}
A.~Katsevich.
\newblock Improved dimension dependence in the {B}ernstein von {M}ises theorem
  via a new {L}aplace approximation bound.
\newblock {\em Information and Inference: A Journal of the IMA}, to appear,
  2025.

\bibitem{katspok}
A.~Katsevich and V.~Spokoiny.
\newblock A unified theory of the high-dimensional {L}aplace approximation with
  application to {B}ayesian inverse problems.
\newblock {\em arXiv preprint arXiv:2509.07952}, 2025.

\bibitem{lapinski2019multivariate}
T.~M. {\L}api{\'n}ski.
\newblock Multivariate {L}aplace's approximation with estimated error and
  application to limit theorems.
\newblock {\em Journal of Approximation Theory}, 248:105305, 2019.

\bibitem{Li2017LargeDeviations}
H.~Li.
\newblock Large deviations.
\newblock In E.~Serpedin, T.~Chen, and D.~Rajan, editors, {\em Mathematical
  Foundations for Signal Processing, Communications, and Networking}, pages
  295--315. CRC Press, Boca Raton, FL, 2017.

\bibitem{normalization2020}
Z.-H. Lu, C.-H. Cai, Y.-G. Zhao, Y.~Leng, and Y.~Dong.
\newblock Normalization of correlated random variables in structural
  reliability analysis using fourth-moment transformation.
\newblock {\em Structural Safety}, 82:101888, 2020.

\bibitem{owen2013monte}
A.~B. Owen.
\newblock Monte carlo theory, methods and examples, 2013.

\bibitem{BenRached2016UnifiedIS}
N.~B. Rached, A.~Kammoun, M.-S. Alouini, and R.~Tempone.
\newblock Unified importance sampling schemes for efficient simulation of
  outage capacity over generalized fading channels.
\newblock {\em IEEE Journal of Selected Topics in Signal Processing},
  10(2):376--388, 2016.

\bibitem{rubinstein2004cross}
R.~Y. Rubinstein and D.~P. Kroese.
\newblock {\em The cross-entropy method: a unified approach to combinatorial
  optimization, {M}onte-{C}arlo simulation and machine learning}.
\newblock Springer Science \& Business Media, 2004.

\bibitem{schillings2020convergence}
C.~Schillings, B.~Sprungk, and P.~Wacker.
\newblock On the convergence of the {L}aplace approximation and
  noise-level-robustness of {L}aplace-based {M}onte {C}arlo methods for
  {B}ayesian inverse problems.
\newblock {\em Numerische Mathematik}, 145:915--971, 2020.

\bibitem{siegmund1976importance}
D.~Siegmund.
\newblock Importance sampling in the {M}onte {C}arlo study of sequential tests.
\newblock {\em The Annals of Statistics}, pages 673--684, 1976.

\bibitem{spokoiny2023dimension}
V.~Spokoiny.
\newblock Dimension free nonasymptotic bounds on the accuracy of
  high-dimensional {L}aplace approximation.
\newblock {\em SIAM/ASA Journal on Uncertainty Quantification},
  11(3):1044--1068, 2023.

\bibitem{Straub2016BayesianRareEvents}
D.~Straub, I.~Papaioannou, and W.~Betz.
\newblock Bayesian analysis of rare events.
\newblock {\em Journal of Computational Physics}, 314:538--556, 2016.

\bibitem{tang2025laplace}
Y.~Tang and N.~Reid.
\newblock Laplace and saddlepoint approximations in high dimensions.
\newblock {\em Bernoulli}, 31(3):1759--1788, 2025.

\bibitem{tong2023large}
S.~Tong and G.~Stadler.
\newblock Large deviation theory-based adaptive importance sampling for rare
  events in high dimensions.
\newblock {\em SIAM/ASA Journal on Uncertainty Quantification}, 11(3):788--813,
  2023.

\bibitem{tong2021extreme}
S.~Tong, E.~Vanden-Eijnden, and G.~Stadler.
\newblock Extreme event probability estimation using {PDE}-constrained
  optimization and large deviation theory, with application to tsunamis.
\newblock {\em Communications in Applied Mathematics and Computational
  Science}, 16(2):181--225, 2021.

\bibitem{uribe2021cross}
F.~Uribe, I.~Papaioannou, Y.~M. Marzouk, and D.~Straub.
\newblock Cross-entropy-based importance sampling with failure-informed
  dimension reduction for rare event simulation.
\newblock {\em SIAM/ASA Journal on Uncertainty Quantification}, 9(2):818--847,
  2021.

\bibitem{robRES}
R.~J. Webber, D.~A. Plotkin, M.~E. O’Neill, D.~S. Abbot, and J.~Weare.
\newblock Practical rare event sampling for extreme mesoscale weather.
\newblock {\em Chaos: An Interdisciplinary Journal of Nonlinear Science},
  29(5):053109, 05 2019.

\bibitem{symmtens}
X.~Zhang, C.~Ling, and L.~Qi.
\newblock The best rank-1 approximation of a symmetric tensor and related
  spherical optimization problems.
\newblock {\em SIAM Journal on Matrix Analysis and Applications},
  33(3):806--821, 2012.

\end{thebibliography}
\end{document}